\documentclass[12pt]{book}

\usepackage{amsthm,amsmath,amssymb}
\usepackage{graphicx}
\usepackage{mystyle}
\usepackage{graphicx}
\usepackage[margin=1.5in]{geometry} %to set margins

\usepackage[toc,page]{appendix}
\title{Functorial Approach to Graph and Hypergraph Theory}
\author{Martin Schmidt}

\begin{document}
\maketitle
\tableofcontents
\chapter*{}
\textbf{Abstract}\\\\
In this thesis we provide a new approach to categorical graph and hypergraph theory by using categorical syntax and semantics. For each monoid $M$ and action on a set $X$, there is an associated presheaf topos of $(X,M)$-graphs where each object can be interpreted as a generalized uniform hypergraph where each edge has cardinality $\#X$ incident vertices (including multiplicity) and where the monoid informs what type of cohesivity the edges possess. One distinguishing feature of $(X,M)$-graphs is the presence of unfixed edges.  We prove that unfixed edges are a necessary feature of a category of graphs or uniform hypergraphs if one wants exponentials and effective equivalence relations to exist in the category. The main advantage of separating syntax (the $(X,M)$-graph theories) from semantics (the categories of $(X,M)$-graphs) is the ability to interpret the theory in any cocomplete category. This interpetation functor then yields a nerve-realization adjunction and allows us to transfer structure between the category of $(X,M)$-graphs and the receptive cocomplete category.  In particular, each morphism of $(X,M)$-graph theories induces an essential geometric morphism between the categories of $(X,M)$-graphs. Simple graphs and labeled/colored graphs are easily constructed by taking the separated objects under the double negation topology and slice categories respectively. Thus the categories of simple $(X,M)$-graphs are Grothendieck quasi-toposes and the categories of labeled/colored $(X,M)$-graphs are Grothendeick toposes. The classically defined categories of hypergraphs, uniform hypergraphs, and graphs are often cocomplete allowing obvious interpretations to take place in these categories. We show the existence or non-existence of structure in each of these classically constructed categories using the transfer of structure under the adjunctions. This framework then gives us a way to describe any type of category whose objects consist of a set of vertices and a set of edges. 

\chapter*{Introduction}
\addcontentsline{toc}{chapter}{Introduction}
\label{intro}
Graphs, hypergraphs and uniform hypergraphs are general mathematical objects which have a wide array of applications. The structure of the various categories of graphs, hypergraphs and uniform hypergraphs have been examined in \cite{rBS}\cite{rB}\cite{wD}\cite{wL}\cite{pH}\cite{cJ}\cite{sL}\cite{dP}\cite{kW}. More recently, categorical methods in graph and hypergraph theory have been employed category theory foundations \cite{mE}, graph and hypergraph transformations \cite{mB}, control systems \cite{sMG}, and graph theory itself \cite{pB}\cite{jF}\cite{jFcT}\cite{wG}\cite{aD}. Thus it is imperative to clarify the issue regarding the categorical structures involved and provide a general framework for graph, uniform hypergraph, and hypergraph theory which is intuitive, robust, and useful. 

Previous attempts (\cite{dP}\cite{cJ}\cite{pH}\cite{wG},\cite{wD}) to describe the categories of graphs assume the objects are classically defined.  We instead adopt the philosophy of Grothendieck and define "nice" categories with a few "bad" objects in order to have available the categorical constructions available in topos theory. Thus our aim is to provide a framework for categorical graph theory using the categories of $(X,M)$-graphs and reflexive $(X,M)$-graphs, which are categories of presheaves\footnote{Recall a category of presheaves is one which is equivalent to a functor category $[\C C^{op},\mathbf{Set}]$ for some small category $\C C$.} on two-object categories (Definition \ref{D:XMGraph}) which can be thought of as generic containers for the set of vertices and set of arcs and where a monoid $M$ informs the type of coherence involved. 

The categories of (reflexive) $(X,M)$-graphs should be thought of as categories of generalized $k$-uniform hypergraphs where $k$ is the cardinality of $X$. In particular, the categories of (reflexive) $(X,M)$-graphs have incidence in multisets which have been examined recently in \cite{kI}\cite{hH}. The other distinguishing feature of (reflexive) $(X,M)$-graphs is the presence of unfixed edges. In the case $X$ is a two-elements set, the only unfixed edges are $2$-loops, which are called bands in \cite{rBS}\cite{jS}. We prove in Part \ref{S:MainResults} that unfixed edges are necessary for the construction of exponentials in conventional categories of graphs (Corollaries \ref{C:HyperExp}, \ref{C:Exponentials}, \ref{C:NoExpR}). 

One motivation for defining the categories of $(X,M)$-graphs is to address the problem that the category of $k$-uniform hypergraphs (as defined in \cite{wD}) lacks connected colimits, exponentials and does not continuously embed into the category of hypergraphs.  However, there is a continuous embedding of the category of $k$-uniform hypergraphs in a category of $(X,M)$-graphs (Proposition \ref{P:HyperGSG}) which preserves any relevant categorical structures (e.g., colimits, exponentials, injectives, projectives). Therefore, working in a category of (reflexive) $(X,M)$-graphs would provide a better categorical environment for constructions on uniform hypergraphs.
 
In Chapter \ref{S:XMG}, we define (reflexive) $(X,M)$-graph theories and $(X,M)$-graph categories and examine their categorical structures. In particular, an equivalence between the category of monoid actions, the category of $(X,M)$-graph theories, and the (meta)category of $(X,M)$-graph toposes is shown (Proposition \ref{P:Equiv}). Thus statements about monoid and group actions may be rephrased into graphical and topos theoretic statements (see Examples \ref{E:FFC}(\ref{E:OrbitStab},\ref{E:Lagrange}) Example \ref{E:CG}(\ref{E:CayleyGraph})). This opens the door to new crossovers of research where a broader view may be taken and more general statements about graphs, uniform hypergraphs, and monoid actions can be proven. Also, by separating syntax ($(X,M)$-graph theories) from semantics ($(X,M)$-graph categories), functorial constructions between $(X,M)$-graph categories are often induced by taking the obvious morphisms between theories. In particular, since a morphism between monoid actions  induces an essential geometric morphism between categories of $(X,M)$-graphs, bridges of structure via adjunctions can be constructed between the different categories of (reflexive) $(X,M)$-graphs in a general way. In particular, the adjunctions between the usual categories of graphs and hypergraphs (Corollary \ref{C:Essential}), evaluation at components (Examples \ref{E:FFC}(\ref{E:3}), cf, \cite{wG}) and constructions of injective hulls and projective covers (Chapter \ref{S:InjProj}) can all be obtained by using the obvious morphisms of monoid actions as well as obvious interpretations of $(X,M)$-graph theories (Part \ref{S:MainResults}). 

In Chapter \ref{S:XGraphs} we show how the categories of (reflexive) $(X,M)$-graphs introduced in this paper are able to describe the types of incidence in the various definitions of graphs and hypergraphs, e.g., oriented\footnote{This is not to be confused with oriented hypergraphs as defined in \cite{lR} (see Example \ref{E:Labeled}(\ref{E:Oriented}))}, unoriented \cite{aB}, directed \cite{gG}, and hereditary \cite{dM}, by taking the monoid $M$ to be a submonoid of endomaps on a set $X$. Thus, when $X$ is a two-element set, the categories of $(X,M)$-graphs generalize the various categories of graphs and undirected graphs found in \cite{rB} \cite{rBS} \cite{wL}. Therefore, any results proven for (reflexive) $(X,M)$-graphs in general hold for these categories as well. By taking a two set partition of a set $X$ and taking submonoids of the endomaps on $X$ which restrict to these subsets give us a way to define directed uniform hypergraphs (\cite{gG}).

In Chapter \ref{S:ComparisonFunctors}, we show how the natural inclusions of submonoids endomaps which induce essential geometric morphisms between the categories of $X$-graphs is functorial on the category of sets with injective maps. This extends the constructions in \cite{rB} and exemplifies the difference between modifications of structure as opposed to modification of data. Many concrete examples of these essential geometric morphisms are then given. 

In Chapter \ref{S:Colored}, we show that the categories of labeled $(X,M)$-graphs are generalizations of labeled and colored graphs and obtained by taking appropriate slice categories of $(X,M)$-graphs. Therefore labeled and colored $(X,M)$-graphs are also categories of presheaves. Thus by embedding $k$-uniform hypergraphs in the appropriate category of $(X,M)$-graphs, we can describe the categories of colored and labeled $k$-uniform hypergraph as slice categories which is otherwise unavailable in the classical definition. This allows us to formulate Hedetniemi's  conjecture in functoral terminology (Example \ref{E:Labelings}(\ref{E:HC})) in the general case. 

Categories of hybrid and mixed $(X,M)$-graphs are discussed in Chapter \ref{S:HyM}. In particular, the category of bipartite graphs is a hybrid $(X,M)$-graph (Example \ref{E:Bipart}). The well-known functor from the category of hypergraphs to bipartites graphs is shown to  respect the dualization of data given (Example \ref{E:Hybrid}(\ref{E:DualBi})) and characterizes the usual dualization of hypergraphs. However, the construction of a hypergraph from a bipartite graph is shown to not extend to a functor. Using hybrid $(X,M)$-graphs also allows us to generalize constructions of intersection graphs (also known as line graphs) of hypergraphs which in turn are used for constructions of Kneser hypergraphs. Each of these cases supports the idea that the natural categorical context in which to study these constructions should be in the category of bipartite graphs rather than the category of hypergraphs.

In Chapter \ref{S:ToposStructure} we look at the topos structures of the categories of (reflexive) $(X,M)$-graphs. The categories of (reflexive) $(X,M)$-graphs are instances of categories of cohesive or variable sets as described in \cite{wLrR} where the sets of vertices are binded by arcs and the monoid $M$ informs the type of cohesivity between vertices. We generalize the result concerning the construction of exponentials in \cite{jS} and \cite{rBS}, provide a set construction of the arc set (Chapter \ref{S:Exponential}), and show the importance of maintaining a distinction between unfixed loops (also called bands in the case $X$ is a two element set) and fixed loops in the construction of exponentials (Examples \ref{E:Exponentials}(\ref{E:XLoops}) and \ref{E:Exponentials}(\ref{E:1Loops}), Corollaries \ref{C:Exponentials} and \ref{C:NoExpR}). The topos theoretic properties (e.g., connected, \'etendue, etc.) for a (reflexive) $(X,M)$-graph are considered (Chapter \ref{S:PointsPieces}).  This generalizes \cite{wL} and exhibits the difference between the categories of reflexive and non-reflexive $(X,M)$-graphs.

Simple $(X,M)$-graphs are defined as the categories of separated presheaves under the double negation topology in Chapter \ref{S:Simple} and generalize the case of simple graphs. This implies that simple $(X,M)$-graphs are Grothendieck quasi-toposes which are reflective subcategories of $(X,M)$-graphs which preserve exponentials (Proposition \ref{P:SimpleNegNeg}). Moreover, we show in Example \ref{E:Hybrid}(\ref{E:Bipart}) that the functor from hypergraphs to bipartite graphs factors through the category of simple bipartite graphs and thus better exemplifies the difference between hypergraph and bipartite graph morphisms (cf \cite{tW}).

The characterization injective and projective objects in a category of graphs in \cite{kW} for undirected graphs and more recently in \cite{wG} for directed graphs is improved upon by using the more abstract framework of (reflexive) $(X,M)$-graphs in Chapter \ref{S:InjProj}. We show there are natural functorial injective and projective refinements of objects (Chapter \ref{S:InjProj}). This also allows us to construct the injective hulls and projective covers for (reflexive) $(X,M)$-graphs.  

In the Part \ref{S:MainResults}, we model $(X,M)$-graph theories in categories of $F$-graphs \cite{cJ} and reflexive $F$-graphs which generalizes \cite{dP}. In particular, we focus on models of (reflexive) $(X,M)$-graphs in the category of hypergraphs (Chapter \ref{S:PGraphs}) and generalizations of categories of undirected graphs to show how a bridge of structure between uniform hypergraphs, hypergraphs, and graphs can be made via nerve-realization adjunctions induced by the obvious interpretations (Propositions \ref{P:HyperGSG}, \ref{P:EENerve} and \ref{P:EVNerve}). 

\ignore{
A hypergraph $H=(H(V),H(E), \varphi)$ consists of a set of vertices $H(V)$, a set of edges $H(E)$ and an incidence map $\varphi\colon H(E)\to \mathcal P(H(V))$ where $\mathcal P\colon \mathbf{Set}\to \mathbf{Set}$ is the covariant power-set functor. Given hypergraphs $H=(H(V),H(E),\varphi)$ and $H'=(H'(V),H'(E),\varphi')$ a morphism $f\colon H\to H'$ consists of a set map $f_V\colon H(V)\to H'(V)$ and a set map $f_E\colon H(E)\to H'(E)$ such that $\varphi'\circ f_E=\mathcal P(f_V)\circ \varphi$. In other words, the following diagram commutes.
\[
\xymatrix{ H(E) \ar[d]_{\varphi} \ar[r]^{f_E} & H'(E) \ar[d]^{\varphi'} \\ \mathcal P(H(V)) \ar[r]^-{\mathcal P(f_V)} & \mathcal P(H'(V))}
\]
Thus the category of hypergraphs $\C H$ is the comma category $\mathbf{Set}\ls \mathcal P$ (\cite{eR}, p 22). The full subcategory of $\C H$ consisting of $k$-uniform hypergraphs $\kH$ is not well-behaved. It lacks connected colimits, exponentials, and even though it is complete, the inclusion functor $j\colon \kH\to \C H$ does not preserve limits. We address this issue by showing the inclusion functor $j\colon \kH\to \C H$ factors through a category of presheaves $\widehat{\DD G}_{(X,M)}$, called a category of $(X,M)$-graphs,
\[
\xymatrix@!=.3em{ & \widehat{\DD G}_{(X,M)} \ar[dr]^{h} & \\ \kH \ar@{>->}[ur]^{i} \ar@{>->}[rr]^j && \C H}
\]
such that $i\colon \kH\hookrightarrow \widehat{\DD G}_{(X,M)}$ is a continuous embedding which preserves any exponentials which exist and $h$ is a left adjoint (Proposition \ref{P:HyperGSG}). 
}

In Chapter \ref{C:FGraphs} we use the functorial semantics introduced in \cite{hA} to show which categorical structures exist and do not exist in the categories of (reflexive) $F$-graphs. In particular, we show that the generalized categories of undirected graphs and reflexive undirected graphs are $\infty$-positive geometric categories with a natural numbers object and subobject classifier (Proposition \ref{P:InftyGeo}).  This exemplifies the difference between the categories of (reflexive) $(X,M)$-graphs and more categories defined based on the set-theoretical approach. Using Giraud's theorem, we can capture this difference by either of the following conceptual equations:
\begin{align*}
	\text{Graphs}+\text{Exponentials}&=(X,M)\text{-Graphs}\\
	\text{Graphs}+\text{Effective Equivalence Relations}&=(X,M)\text{-Graphs}.
\end{align*}
where Graphs is the corresponding category of uniform hypergraph classically defined. 

 This thesis is a step toward a "universal graph theory" where general constructions and results based in various specialized areas of graph theory can be formulated. Our approach differs from previous attempts at describing categorical graph theory by being able to describe structures relevant to graph theorists. In this paper you will find $(X,M)$-graph generalizations and categorical reformulations for each of the following:
 \begin{enumerate}
 	\item Cayley and Schreier Graphs, Examples \ref{E:CG}(\ref{E:CayleyGraph}, \ref{E:SchreierGraph})
 	\item Non-rainbow and Non-monochromatic hypergraphs, Examples \ref{E:Labelings}(\ref{E:HC}, \ref{E:NRainbow})
 	\item Intersection Graphs, Example \ref{E:Hybrid}(\ref{E:Int})
 	\item Pultr Functors, Example \ref{E:NRAdj}(\ref{E:Pultr})
 	\item Hypergraph Dualization, Example \ref{E:OSG}(\ref{StructDual})
 	\item Bipartite Graph Double Covers, Example \ref{E:Exponentials}(\ref{E:BPDC})
 	\item Regular Graphs, Example \ref{S:HyM}(\ref{E:RegularGraphs})
 	\item Reconstruction Conjectures, Example \ref{E:TNN}(\ref{E:DR})
 	\item Hedetniemi's Conjecture, Example \ref{E:Labelings}(\ref{E:HC})
 	\item Kneser Hypergraphs, Example \ref{E:TNN}(\ref{E:Kneser})
 	\item The Various Products of Graphs, Section \ref{S:PPP}
 \end{enumerate}
 We want to stress that this is not mere pedantry for the sake of reformulation of graph theory into category theory. Rather, it is an illumination of the structures involved and a capturing of a changed perspective of what is meant by "graphical structure". Our methodology exemplifies the intimate connection between monoid actions and "graphs", and how we can study them via transfer of structure along adjoint situations induced by interpretation functors.

\ignore{In the final part, we characterize the various categories of hypergraphs found in the literature by isolating five main factors: (1) Does it have strict morphisms? (2) Do the objects allows multiplicities of vertices in the incidence of edges? (3) Do the objects have non-empty set of incidence vertices? (4) Do the objects have oriented incidence of edges? (5) Are the incidence of edges of objects bounded by some fixed cardinality? Then we show the existence of adjunctions between them which connects to the categories of $(X,M)$-graphs given in the first two parts. This provides a complete picture of the various types of categories used in combinatorics along with the adjunctions connecting them. }
\part{Categories of $(X,M)$-Graphs}

\chapter{A General Framework}\label{S:XMG}

\section{$(X,M)$-Graphs}
We begin with a definition.

\begin{definition}\label{D:XMGraph} \mbox{}
	\begin{enumerate}
		\item Let $M$ be a monoid and $X$ a right $M$-set. 
		The theory for $(X,M)$-graphs, $\DD G_{(X,M)}$, is the category with two objects $V$ and $A$ and homsets given by
		\begin{align*}
		   & \DD G_{(X,M)}(V,A) \defeq X, \\
		  & \DD G_{(X,M)}(A,V) \defeq \empset, \\
		  & \DD G_{(X,M)}(V,V)\defeq \{\Id_V\},\\
		 & \DD G_{(X,M)}(A,A)\defeq M.
		\end{align*}
		Composition is defined as  $m\circ x=x.m$ (the right-action via $M$), $m\circ m'=m'm$ (monoid operation of $M$).
		\item Let $M$ be a monoid such that $\Fix(M)\defeq \setm{m'\in M}{\forall m\in M, m'm=m'}$ is non-empty. Let $X\defeq \setm{x_{m'}}{m'\in \Fix(M)}$ be the right $M$-set with right-action $x_{m}.m'\defeq x_{mm'}$ for each $m\in M$ and $x_{m'}\in X$. The theory for reflexive $(X,M)$-graphs, $\rG_{(X,M)}$ is the same as for $\DD G_{(X,M)}$ but with 
		\begin{align*}
		\rG_{(X,M)}(A,V)\defeq \{\ell\},
		\end{align*}
		and composition $\ell\circ m=\ell$, $\ell\circ x_{m'}=\Id_V$, and $x\circ \ell=x$ for each $m\in M$ and $m'\in \Fix(M)$.
	\end{enumerate}
	The category of $(X,M)$-graphs (resp. reflexive $(X,M)$-graphs) is defined to be the category of presheaves $\widehat{\DD G}_{(X,M)}\defeq [\DD G_{(X,M)}^{op},\mathbf{Set}]$ (resp. $\widehat{\rG}_{(X,M)}\defeq [\rG_{(X,M)}^{op},\mathbf{Set}].$) 
\end{definition}

We represent the theory for $(X,M)$-graphs and and reflexive $M$-graphs as follows.  
\[\xymatrix{ V  \ar[r]|-{X} & A \ar@(ur,dr)^M}\qquad \qquad
\xymatrix{ V \ar@<+.3em>[r]|-X \ar@{<-}@<-.3em>[r]|-{\ell} & A \ar@(ur,dr)^M }
\] By definition, an $(X,M)$-graph $G\colon \DD G^{op}_{(X,M)}\to \mathbf{Set}$ has a set of vertices $G(V)$ and a set of arcs $G(A)$ along with right-actions for each morphism in $\DD G_{(X,M)}$. For example, $x\colon V\to A$ in $\DD G_{(X,M)}$ yields a set map $G(x)\colon G(A)\to G(V)$ which takes an arc $\alpha\in G(A)$ to $\alpha.x\defeq G(x)(\alpha)$ which we think of as its $x$-incidence.\footnote{Note that we use the categorical notation of evaluation of a presheaf as a functor for the set of vertices $G(V)$ and set of arcs $G(A)$ rather than the conventional graph theoretic $V(G)$ and $E(G)$ for the vertex set and edge set.} For an element $m$ in the monoid $M$, the corresponding morphism $m\colon A\to A$ in $\DD G_{(X,M)}$ yields a right-action $\alpha.m\defeq G(m)(\alpha)$ which we think of as the $m$-associated partner of $\alpha$.  If $G$ is a reflexive $(X,M)$-graph, the $\ell$-action can be thought of as the extraction of a loop from a vertex. We call a loop equal to $x.\ell$ a distinguished loop for vertex $x$. It can be thought of as the arc-proxy for the vertex. This will allow us to map arcs to vertices, or more precisely, arcs to distinguished loops.

Each $(X,M)$-graph $G$ induces a set map $\partial_G\colon G(A)\to G(V)^X$ such that $\partial_G(\alpha)\colon X\to G(V)$ is the parametrized incidence of $\alpha$, i.e., $\alpha.x=\partial_G(\alpha)(x)$. The $x$-incidence can be recovered from a parametrized incidence by precomposition of the map $\named x\colon 1\to X$ which names the element $x$ in $X$. Observe that the $m$-associated partner of an arc $\alpha$ in $G$ has the parametrized incidence such that the following commutes
\[
\xymatrix@C=3em{ X \ar@/_1em/[rrr]_{\partial_{G}(\alpha.m)} \ar[r]^-{\langle \Id_X, \named m \rangle} & X\x M \ar[r]^-{\text{\tiny action}} & X \ar[r]^-{\partial_{G}(\alpha)} & G(V).}
\]
If $G$ is a reflexive graph, $\partial_{G}(x.\ell)=\named x\circ !_X$  where $!_X$ is the terminal set map. 

\section{Characterization Results}
By definition, the category of $(X,M)$-graphs is a presheaf topos and thus we have the following characterization for an abstract category to be a (reflexive) category of $(X,M)$-graphs.\footnote{Compare this to Section 4 in \cite{dP} where an elementary theory of the category of graphs is given.} We follow the terminology in \cite{pJ}.

\newpage
\begin{proposition}\mbox{}
	\begin{enumerate}
	\item A category $\C A$ is equivalent to a category of $(X,M)$-graphs iff the following conditions hold:
		\begin{enumerate}
			\item $\C A$ is effective regular\footnote{also called \textit{Barr-exact}} and extensive,
			\item There is a subterminal object $\underline V$  which is regular projective and connected.
			\item There is a regular projective and connected object $\underline A$ which does not admit a morphism to $\underline V$ such that the set $\{\underline V,\underline A \}$ is a separating set of smallest cardinality.
		\end{enumerate}
		\item 	A category $\C B$ is equivalent to a category of reflexive $(X,M)$-graphs iff the following conditions hold.
		\begin{enumerate}
			\item $\C B$ is effective regular and extensive.
			\item There is a regular projective, connected, non-terminal separator $\underline A$.
		\end{enumerate}
		\end{enumerate}
\end{proposition}
\begin{proof}
	The conditions are equivalent to $\C A$ and $\C B$ being a categories of presheaves on the full subcategories with object set $\{\underline V,\underline A\}$ (\cite{aC}, Corollary 43) where in the reflexive case $\underline V$ is the terminal object. Let $X\defeq \C A(\underline V,\underline A)$ (resp. $X\defeq \C B(\underline V, \underline A)$) and $M\defeq \C A(\underline A,\underline A)$ (resp. $M\defeq \C B(\underline A, \underline A)$). Then we have $\C A\simeq \widehat{\DD G}_{(X,M)}$ and $\C B\simeq \widehat{\rG}_{(X,M)}$.\qed
\end{proof}

Since categories of $(X,M)$-graphs are categories of presheaves, they enjoy many properties of the category of sets. We list a few without comment to show relevance to their practical use in such fields as logic, combinatorics, and computer science. 
\begin{corollary}\label{C:ToposXM}
	A category of (reflexive) $(X,M)$-graphs  
	\begin{enumerate}
		\item is complete, cocomplete, well-powered, well-copowered and locally cartesian closed.
		\item has enough injectives and projectives. 	
		\item is an adhesive HLR category (\cite{hE}, Definition 4.9).\footnote{with respect to monomorphisms} 
		\item is an $\infty$-positive geometric category with a separating set of objects (\cite{pJ}, A1.4.19, A1.2.4).
		\item is an $\infty$-lextensive category.
	\end{enumerate}
\end{corollary}

\section{Essential Geometric Morphisms}

Recall that an essential geometric morphism $f\colon \C E\to \C E'$ between toposes consists of adjunctions $f_!\dashv f^*\colon \C E'\to \C E$ and $f^*\dashv f_*\colon \C E\to \C E'$ which we write as $f_!\dashv f^*\dashv f_*\colon \C E\to \C E'$ and call the $f$-extension, $f$-restriction, and $f$-corestriction respectively.  A functor $F\colon \C C\to \C C'$ between small categories $\C C$ and $\C C'$ induces an essential geometric morphism $F_!\dashv F^*\dashv F_*\colon \widehat{\C C}\to \widehat{\C C'}$ where $F_!$ is the realization of $y_{\C C'}\circ F$, $F^*$ is the nerve of $y_{\C C'}\circ F$ and $F_*$ is the nerve of $F^*\circ y_{\C C'}$ where $y_{\C C}\colon \C C\to \widehat{\C C}$ and $y_{\C C'}\colon \C C'\to \widehat{\C C'}$ are the Yoneda embeddings (see \cite{mR} pp 194-198). In the subsequent, we denote the representables $\underline V\defeq \widehat{\DD G}_{(X,M)}(-,V)$ and $\underline A\defeq \widehat{\DD G}_{(X,M)}(-,A)$.

\begin{example}\label{E:FFC} \mbox{}
	\begin{enumerate}
		\item \label{E:FreeForget} (The Free, Forgetful, Cofree Adjunction for Categories of $(X,M)$-Graphs)\\ Consider the $(X,M)$-graph theory $\DD G_{(\empset, 1)}$, i.e., the discrete category with two objects $V$ and $A$. Then for an $(X,M)$-graph theory $\DD G_{(X,M)}$, there is the inclusion functor $\iota\colon \DD G_{(\empset,1)}\to \DD G_{(X,M)}$. Thus there is an essential geometric morphism $\iota_!\dashv \iota^*\dashv \iota_*\colon \mathbf{Set}^2\to \widehat{\DD G}_{(X,M)}$. The $\iota$-extension $\iota_!$ takes the pair of sets $(S(V),S(A))$ to the coproduct $\bigsqcup_{S(V)}\underline V\sqcup \bigsqcup_{S(A)}\underline A$ since the category of elements for $(S(V),S(A))$ lacks internal cohesion. The $\iota$-restriction $\iota^*$ takes an $(X,M)$-graph $G$ to the pair of sets $(G(V),G(A))$. By \cite{eR} (Proposition 3.3.9), it creates all limits and colimits.  The $\iota$-coextension $\iota_*$ sends $(S(V),S(A))$ to the $(X,M)$-graph with vertex set $S(V)$ and arc set $\mathbf{Set}^2((X,|M|),(S(V),S(A)))=S(V)^X\x S(A)^{|M|}$ where $|M|$ is the underlying set of $M$. The right-actions are given by $(f,s).x=f(x)$, $(f,s).m=(f\circ m,s.m)$ where $m\colon X\to X$ is the right-action by $m\in M$ and $s.m\colon |M|\to S(A)$ is defined $s.m(m')\defeq s(mm')$. 
		
		The counit $\varepsilon\colon \iota_!\iota^*\Rightarrow \Id$ of the adjunction $\iota_!\dashv \iota^*$ on a component $\varepsilon_G\colon \bigsqcup_{G(V)}\underline V\sqcup \bigsqcup_{G(A)}\underline A \to G$ is the epimorphism induced by the classification maps $v\colon \underline V\to G$ and $\alpha\colon \underline A\to G$ for vertices $v\in G(V)$ and arcs $\alpha\in G(A)$. The unit $\eta\colon \Id\Rightarrow \iota_*\iota^*$ of the adjunction $\iota^*\dashv \iota_*$ on a component $\eta_G\colon G\to \iota_*\iota^*(G)$ is the identity on vertices and sends arc $\alpha\in G(A)$ to $(\partial_G(\alpha),\overline \alpha)$ where $\overline \alpha\colon |M|\to G(A)$ is the constant map. Thus for each $(X,M)$-graph $G$ the component $\eta_G$ is a monomorphism. These natural transformations will be used to characterize the injective and projective objects in the category of $(X,M)$-graphs (see Chapter \ref{S:InjProj} below).
		
		\item \label{E:FreeForget2} (The Free, Forgetful, Cofree Adjunction for Categories of Reflexive $(X,M)$-Graphs)\\
		In the case of reflexive $(X,M)$-graphs, the inclusion $\iota\colon \DD G_{(\empset, 1)} \to \rG_{(X,M)}$ induces a similar essential geometric morphism $\iota_!\dashv \iota^*\dashv \iota_*\colon \mathbf{Set}^2\to \widehat{\rG}_{(X,M)}$ which behaves the same for the adjunction $\iota_!\dashv \iota^*$ as the non-reflexive case. However, for the $\iota$-coextension, $(S(V),S(A))$ is sent to the reflexive $(X,M)$-graph with vertex set $S(V)\x S(A)$ and arc set $S(V)^X\x S(A)^{|M|}$. The right-actions are given by $(f,s).x=(f(x),s(x\circ \ell))$ for each $x\in X$ and $(f,s).m=(f\circ m, s.m)$ as above. 
		
		Similar to the non-reflexive case, the counit $\varepsilon\colon \iota_!\iota^*\Rightarrow \Id$ of the adjunction $\iota_!\dashv \iota^*$ on a component $\varepsilon_G$ is the canonical epimorphism. The unit $\eta\colon \Id\Rightarrow \iota_*\iota^*$ on a component $\eta_G\colon G\to \iota_*\iota^*(G)$ is morphism which on vertices takes $v\in G(V)$ to $(v, v.\ell)$ and on arcs takes $\alpha\in G(A)$ to $(\partial_G(\alpha), \overline \alpha)$ where $\overline \alpha$ is the constant map. Therefore $\eta_G$ is a monomorphism for each reflexive $(X,M)$-graph $G$. 
		 
		\item \label{E:3} (Evaluation at Components) Let $1$ be the terminal category and consider the inclusion functor $V\colon 1\to 2$ which names the object $V$ in $\DD G_{(\empset, 1)}$. This induces the essential geometric morphism $V_!\dashv V^*\dashv V_*\colon \mathbf{Set}\to \mathbf{Set}^2$. The $V$-extension sends a set $S$ to $(S,\empset)$. The $V$-restriction sends the pair $(S(V),S(A))$ to $S(V)$. The $V$-coextension sends $S$ to $(S,1)$. 
		
		Similarly, for the inclusion functor $A\colon 1\to 2$ which names the object $A$ in $\DD G$, we have the essential geometric morphism $A_!\dashv A^*\dashv A_*\colon \mathbf{Set} \to \mathbf{Set}^2$. The $A$-extension sends a set $S$ to $(\empset, S)$. The $A$-restriction sends a pair $(S(V),S(A))$ to $S(A)$. The $A$-coextension sends $S$ to $(1,S(A))$.
		
		Let $\DD T$ be a reflexive or non-reflexive $(X,M)$-graph theory. Note that the evaluation functors $(-)_V, (-)_A\colon \widehat{\DD T}\to \mathbf{Set}$ are the compositions $V^*\iota^*$ and $A^*\iota^*$ respectively. Therefore, the evaluations are the restriction functors of essential geometric morphisms $\iota_!V_!\dashv (-)_V\dashv \iota_*V_*\colon \mathbf{Set} \to \widehat{\DD T}$ and $\iota_!A_!\dashv (-)_A\dashv \iota_*A_*\colon \mathbf{Set} \to \widehat{\DD T}$. This will be used to define vertex-labeled $(X,M)$-graphs in the subsequent (Definition \ref{D:VertexLabeled}).\footnote{For reflexive $(X,M)$-graphs the essential geometric morphism $\iota_!V_!\dashv (-)_V\dashv \iota_*V_I$ is equivalent to the essential geometric morphism $\Delta \dashv \Gamma\dashv B$ given in Chapter \ref{S:PointsPieces}.}
		\item \label{E:OrbitStab} (The Orbit-Stabilizer Theorem) Let $X$ be right $M$-set. Then for each $x\in X$ there is a functor $\named x\colon \DD G_{(1,1)}\to \DD G_{(X,M)}$ which sends $v\colon V\to A$ in $\DD G_{(1,1)}$ to $x\colon V\to A$. This induces the essential geometric morphism such that the restriction functor $x^*\colon \widehat{\DD G}_{(X,M)}\to \widehat{\DD G}_{(1,1)}$ gives a bouquet of $x$-incidences for an $(X,M)$-graph $H$, i.e., $x^*(H)(V)=H(V)$ and $ x^*(H)(A)=H(A)$ such that $\alpha.v$ is the vertex $\alpha.x$ in $H$. The $\widehat{\DD G}_{(X,M)}$-representable $\underline A$ has $\underline A(V)=\setm{v_x}{x\in X}$ and $\underline A(A)=\setm{a_m}{m\in M}$ with right actions $a_m.x=x.m$ and $a_m.m'=a_{mm'}$ (see Chapter \ref{S:ToposStructure}). Therefore $x^*(\underline A)$ is the bouquet which captures the orbit and stabilizer of the action, i.e., the orbit of $x$ is given by the set of vertices in $x^*(\underline A)$ which are not isolated (i.e., $\mathbf{Orb(x)}=\setm{x'\in X}{\exists \alpha\in M,\ v_{x'}=\alpha.x}$) and the stabilizer $\mathbf{Stab}(x)$ is given by the arc set of the subbouquet $\neg\neg x$, \footnote{By the internal logic of the topos $\widehat{\DD G}_{(1,1)}$ $\neg\neg\colon \Sub(x^*(\underline A))\to \Sub(x^*(\underline A))$ is the double negation closure operator. } i.e., $\neg\neg x(A)=\setm{m\in M}{a_m.x=v_x}$. Since an invertible element $g\in M$ is equivalent to an isomorphism $g\colon \underline A\to \underline A$ of $(X,M)$-graphs $x^*(g)\colon x^*(\underline A)\to x^*(\underline A)$ is an isomorphism of bouquets. This isomorphism restricts to an isomorphism between subbouquets $g\colon \neg\neg x\to \neg\neg (x.g)$. Since restriction functors preserve arc sets, we have $\bigcup_{g\in M}\neg\neg(x.g)(A)=M$. If $M$ is a group, every morphism $g\colon \underline A\to \underline A$ is an isomorphism implying that each $\neg\neg(x.g)$ is isomorphic to $\neg\neg x$. This gives us a $(X,M)$-graph theoretic proof of the orbit-stabilizer theorem.
		\begin{align*}
		|M|	&=\left|\bigcup_{g\in M}\neg\neg (x.g)(A)\right|\\
		& =|\neg\neg x(A)|\cdot |\{\text{non-isolated vertices in $x^*(\underline A)$}\}|  \\
		&=|\mathbf{Stab}(x)|\cdot |\mathbf{Orb}(x)|
		\end{align*}
		\item \label{E:Lagrange} (Lagrange's Theorem) Let $H$ be a subgroup of $G$ and $G/H$ the set of right cosets of $H$. There is the obvious right $G$-action on $G/H$ giving us the theory of $(|G/H|,G)$-graphs. For the functor $\named e\colon \DD G_{(1,1)}\to \DD G_{(|G/H|,G)}$ which picks out the coset $He$ for the group unit $e\in G$ induces a restriction functor $e^*\colon \widehat{\DD G}_{(|G/H|,G)}\to \widehat{\DD G}_{(1,1)}$ which preserves the arc and vertex sets. Since the right action is transitive and by using the observation above on the orbit-stabilizer theorem we have $|G|\iso |G/H||H|$ giving us a $(X,M)$-graph theoretic proof of Lagrange's Theorem. 
	\end{enumerate}
\end{example}

\section{Equivalence Results}
Since the theories and categories of $(X,M)$-graphs depend on the monoid and action, it is natural to see how monoid action morphisms transfers to the corresponding theories and categories. We define the category of monoid actions $\C M$ to consist of pairs $(X,M)$ where $M$ is a monoid and $X$ is a right $M$-set for objects and morphisms $(f,\varphi)\colon (X,M)\to (X',M')$ where $f\colon X\to X'$ is a set map and $\varphi\colon M\to M'$ is a morphism of monoids such that the following commutes
\[
\xymatrix{ X\x M \ar[r]^-{\text{\tiny action}} \ar[d]_{f\x \varphi} & X \ar[d]^{f} \\ X'\x M' \ar[r]^-{\text{\tiny action}} & X'.}
\]
The category of $(X,M)$-graph theories $\mbf{Th}$ we define to be the subcategory of small categories consisting of $(X,M)$-graph theories with morphisms those functors which preserve the vertex object $V$ and arc object $A$. We also define the (meta)category of $(X,M)$-graph toposes $\mbf{GT}$ to be the (meta)subcategory of categories of presheaves of the form $\widehat{\DD G}_{(X,M)}$ with morphisms as those essential geometric morphisms (defined up to natural isomorphism) such that the restriction functor preserves vertex and arc sets. 

\begin{proposition}\label{P:Equiv}
	The categories of monoid actions $\C M$, $(X,M)$-graph theories $\mbf{Th}$, and $(X,M)$-graph toposes $\mbf{GT}$ are equivalent. 
\end{proposition}
\begin{proof}
	Consider the functor $\DD G\colon \C M\to \mbf{Th}$, $(X,M)\mapsto \DD G_{(X,M)}$. Since any functor between theories $\gamma\colon \DD G_{(X,M)}\to \DD G_{(X',M')}$ is equivalent to a morphism of monoid actions $\gamma\colon (X,M)\to (X',M')$, $\DD G$ is an equivalence. 
	
	Next, consider the functor $\widehat{(-)}\colon \mbf{Th}\to \mbf{GT}$, $\DD G_{(X,M)}\mapsto \widehat{\DD G}_{(X,M)}$. A morphism $j\colon \DD G_{(X,M)}\to \DD G_{(X',M')}$ between theories yields an essential geometric morphism  $j_!\dashv j^*\dashv j_*$ such that the restriction functor $j^*\colon \widehat{\DD G}_{(X',M')}\to \widehat{\DD G}_{(X,M)}$ which preserves vertex and arc sets for each $(X,M)$-graph $G$, i.e., $j^*(G)(V)=\widehat{\DD G}_{(X,M)}(j(\underline V),G)=G(V)$ and $j^*(G)(A)=\widehat{\DD G}_{(X,M)}(j(\underline A),G)=G(A)$. Therefore $\widehat{(-)}$ is well-defined. If $\gamma_!\dashv \gamma^*\dashv \gamma_*\colon \widehat{\DD G}_{(X,M)}\to \widehat{\DD G}_{(X',M')}$ is a $\mbf{GT}$-morphism it is induced by a unique functor $\gamma\colon \DD G_{(X,M)}\to \DD G_{(X',M')}$ which preserves vertex and arc objects and sends a morphism $x$ to the unique $\DD G_{(X',M')}$-morphism $\gamma(x)$ such that $\underline{\gamma(x)}\iso \gamma_!(\underline x)$. Therefore, $\widehat{(-)}$ is an equivalence. 
	\end{proof}

	Since the functor $\C M\to \mbf{Monoid}$, $(X,M)\mapsto M$, is a fibration and opfibration, we have $\mbf{Th}$ and $\mbf{GT}$ also admit a fibration and opfibration\footnote{called a cofibration in \cite{jG}} to the category of monoids with fibers in categories of $M$-sets, i.e., presheaves on one-object categories. We follow the terminology and notation in \cite{jG}. 

\begin{lemma}\cite{jG} \label{P:FunctorFibration}
	A functor $P\colon \C E\to \C B$ is a fibration iff  for each small category $\DD I$, the functor $P^{\DD I}\colon \C E^{\DD I}\to \C B^{\DD I}$ is a fibration. Moreover, in this case a $\C E^{\DD I}$-morphism $\alpha$ is $P^{\DD I}$-cartesian iff each component $\alpha_x$ is $P$-cartesian.
\end{lemma}

\begin{proposition}
	Let $P\colon \C E\to \C B$ be a fibration (resp. opfibration) and $D\colon \DD I \to \C E$ be a diagram.  Suppose the following conditions hold:
	\begin{ienum} 
		\item The limit $L\defeq \lim_{\DD I}PD$ (resp. colimit $K$) exists in $\C B$. 
		\item The fiber $\C E(L)$ has $\DD I$-limits (resp. $\DD I$-colimits).
		\item For each $\C B$-morphism $f\colon Px\to L$, the pullback $f^*\colon \C E(L)\to \C E(Px)$ preserves $\DD I$-limits (resp., for each $f\colon K\to Px$ the pushout $f_!\colon \C E(K)\to \C E(Px)$ preserves $\DD I$-colimits). 
	\end{ienum}
	Then the limit $L'\defeq \lim_{\DD I}D$ (resp.  $\colim_{\DD I}D$) exists and $P$ preserves it. 
\end{proposition}
\begin{proof} By duality, it is enough to prove the case for limits. 
	By Lemma \ref{P:FunctorFibration}, $P^{\DD I}\colon \C E^{\DD I}\to \C B^{\DD I}$ is a fibration. Let $p\colon \Delta L\Rightarrow PD=P^{\DD I}(D)$ be the universal cone in $\C B$ and $\overline p\colon p^*(D)\Rightarrow D$ the cartesian lift of $p$ (which exists since $P^{\DD I}$ is a fibration). Therefore, $P^{\DD I}(p^*(D))=P\circ p^*(D)=\Delta L$ and thus by the universal mapping property of the fiber category $P\circ p^*(D)\colon \DD I \to \C E$ factors through the fiber $\C E(L)$:
	\[
	\xymatrix{ \DD I \ar@/^/[drr]^{p^*(D)} \ar@{..>}[dr] \ar@/_/[ddr] \\ & \C E(L) \ar@{>->}[r] \ar[d] \ar@{}[dr]|-<\pb & \C E \ar[d]^{P} \\ & \mbf 1 \ar@{>->}[r]^{\named{L}} & \C B }
	\]
	By assumption, $\C E(L)$ has $\DD I$-limits and thus a universal cone $q\colon \Delta L'\Rightarrow p^*(D)$ exists over $p^*(D)$ in $\C E(L)$. We claim the composition $\xymatrix@1@C=1em{\Delta L' \ar@{=>}[r]^-q & p^*(D) \ar@{=>}[r]^-{\overline p} & D}$ is the universal cone over $D$ in $\C E$. Consider the following diagram (note that we omit the symbol $\Delta$ on $\C E$-morphisms):
	\[
	\xymatrix{ \Delta x \ar@{..>}@/_/[ddr]|-\hole_>>>>>>>>r \ar@/^/[drr]^-n \ar@/^/[dr]|->>>>{k=gh} \ar@{..>}[d]|-{h'=h} \ar@{..>}@/_2em/[dd]_{\ell} \\ \Delta L'' \ar[d]|-{Pr^*(q)} \ar[r]^{g=\overline s}_{\text{\tiny cartesian}} & \Delta L' \ar[d]^q \ar[r]^{\overline p q} & D \ar[d]^{\Id_D} \\ Pr^*(p^*(D)) \ar@{|->}[d] \ar[r]^-{\overline{Pr}}_-{\text{\tiny cartesian}} & p^*(D) \ar@{|->}[d] \ar[r]^{\overline p}_{\text{\tiny cartesian}} & D \ar@{|->}[d] \\ \Delta Px \ar[r]^{Pr=\Delta Pg} & \Delta L \ar[r]^{p} & PD}
	\]
	Given a cone $n\colon \Delta x\Rightarrow D$, by the universal mapping property of the cartesian lift of $\overline{p}$, there is a unique factorization $r\colon \Delta x\Rightarrow p^*(D)$ such that $\overline p\circ r=n$. Let $\overline{Pr}\colon Pr^*(p^*(D))\Rightarrow p^*(D)$ and $\overline s\colon p^*(\Delta L')\Rightarrow \Delta L'$ be cartesian lifts of $\Delta Pr$. By assumption $Pr^*(q)\colon Pr^*(\Delta L')\Rightarrow Pr^*(p^*(D))$ is the universal cone over $Pr^*(p^*(D))$ in the fiber $\C E(Px)$. Therefore, $Pr^*(\Delta L')=\Delta L''$ for some $\C E(Px)$-object $L''$ and $\overline s=\Delta g$ for some $\C E$-morphism $g\colon L''\to L'$.  Note that by definition of the functor $Pr^*\colon \C E(\Delta L)\to \C E(\Delta Px)$, we have $\overline{Pr}\circ Pr^*(q)=q\circ \Delta g$.  Since $\overline{Pr}$ is cartesian, the morphism $r\colon \Delta x\Rightarrow p^*(D)$ induces a unique $\C E^{\DD I}$-morphism $\ell\colon \Delta x \Rightarrow Pr^*(p^*(D))$ such that $\overline{Pr}\circ \ell=r$. Therefore by the universal mapping property of the limit $L''$, there is a unique $\C E$-morphism $h\colon x\to L''$ such that $\ell=Pr^*(q)\circ \Delta h$. Hence, 
	\[
	\overline p\circ q\circ \Delta(g\circ h)=\overline p\circ q\circ \Delta g\circ \Delta h=\overline p\circ \overline{Pr}\circ Pr^*(q)\circ \Delta h=\overline p\circ \overline{Pr}\circ \ell= \overline p\circ r=n.
	\]
	Thus $g\circ h$ is a factorization of $n$ through $\overline p\circ q$. For uniqueness, suppose $k\colon x\to L'$ were another factorization. Since $\Delta g$ is cartesian over $Pr$, there is a unique $\C E$-morphism $h'\colon x\to L''$ such that $g\circ h'=k$.  Since $\overline p\circ q\circ \Delta k=n=\overline p\circ r$, by universal property of $\overline p$ being cartesian we have $q\circ \Delta k=r$. Then $\overline{Pr}\circ Pr^*(q)\circ \Delta h'=q\circ \Delta g\circ \Delta h'=q\circ \Delta k=r=\overline{Pr}\circ Pr^*(q)\circ \Delta h$. Since $\overline{Pr}$ is cartesian, we have $Pr^*(q)\circ \Delta h=Pr^*(q)\circ \Delta h'$. Then by universal mapping property of the limit $L''$, we have $h'=h$ showing $k=g\circ h'=g\circ h$. The condition that $P$ preserves the limit is by construction, i.e., $P^{\DD I}(\overline pq)=p$.
		
\end{proof}

\begin{corollary}
	The categories $\C M$, $\mbf{Th}$, and $\mbf{GT}$ are (small) complete and cocomplete.
\end{corollary}
\begin{proof}
	Follows from the proposition above, noting that each fiber over $M$ is the category of $M$-sets (a category of presheaves) and the pullback and pushout functors are the restriction and extension of scalar functors. 
\end{proof}

\begin{corollary}
	The category of monoids is a reflective and coreflective subcategory of $\C M$, $\mbf{Th}$ and $\mbf{GT}$.
\end{corollary}
\begin{proof}
	The functors $\lambda\colon \mbf{Monoid}\to \C M$, $M\mapsto (\empset, M)$ and $\rho\colon \mbf{Monoid}\to \C M$, $M \mapsto (1,M)$ are readily verified to be full and faithful left and right adjoints to the bifibration functor $\C M\to \mbf{Monoid}$.
	
\end{proof}

For the reflexive case, let $\mbf{rTh}$ be the subcategory of small categories consisting of reflexive $(X,M)$-graph theories with morphisms those functors which preserve the vertex object $V$ and arc object $A$. Let $\text{r}\C M$ be the full subcategory of $\C M$ consisting of objects $(X,M)$ such that $X\iso \Fix(M)$ as right $M$-sets. We also define the (meta)category of reflexive $(X,M)$-graph toposes $\tbf{rGT}$ to be the of categories of presheaves

There is an inclusion natural transformation $\DD G\Rightarrow \rG$ where $\DD G, \rG\colon \text{r}\C M\to \Cat$ are the functors which assign the $(X,M)$-theory and reflexive $(X,M)$-theory to each monoid action $(X,M)$ in $\text{r}\C M$. Therefore, for each reflexive $(X,M)$-graph theory $\rG_{(X,M)}$, there is an inclusion $r\colon \DD G_{(X,M)}\to \rG_{(X,M)}$ which induces an essential geometric morphism $r_!\dashv r^*\dashv r_*\colon \widehat{\DD G}_{(X,M)}\to \widehat{\rG}_{(X,M)}$.

\begin{proposition} Let $M$ be a monoid such that $\Fix(M)$ is non-empty. The category of reflexive $(X,M)$-graphs is equivalent to the category of presheaves $[\DD BM^{op},\mathbf{Set}]$ where $\DD BM$ is the one object category with homset equal to $M$.
\end{proposition}
\begin{proof}
For each $x\in X$, the morphism $x\colon V\to A$ in $\rG_{(X,M)}$ is a split monomorphism with retraction $\ell\colon A\to V$. Hence $\rG_{(X,M)}$ and $\DD BM$ have equivalent Cauchy completions. Thus by Proposition 5.2.3 \cite{mR}, $\widehat{\rG}_{(X,M)}$ is equivalent to $[\DD BM^{op},\mathbf{Set}]$. 

\end{proof}

\begin{corollary}
	Let $M$ be a monoid such that $\Fix(M)$ is non-empty. The category of reflexive $(X,M)$-graphs is monadic over $\mathbf{Set}$.
\end{corollary}
\begin{proof}
	From the previous result, it is easily shown that the category of reflexive $(X,M)$-graphs is equivalent the category of algebras on the endofunctor $M\x -\colon \mathbf{Set}\to \mathbf{Set}$. 
	
\end{proof}

\chapter{Categories of $X$-Graphs} \label{S:XGraphs}

\section{$X$-Graphs}
 Let $X$ be a set, we define the following submonoids of the endomap monoid $\End(X)$:
\begin{align*}
& \oX\defeq \{\Id_X\}\\
& \sX \defeq \Aut(X) \quad \text{(the submonoid of automaps)}\\
& \rX \defeq \setm{f\in \End(X)}{f=\Id_X\text{ or } \exists x\in X,\ \forall x'\in X,\ f(x')=x}\\
& \srX \defeq \rX\cup \sX\\
& \hX=\hrX\defeq \End(X).
\end{align*}
Thus there is the following inclusions as submonoids in $\End(X)$
\[
\xymatrix@R=1em@C=1.25em{ \oX \ar@{>->}[d] \ar@{>->}[r]  & \sX \ar@{>->}[r] \ar@{>->}[d] & \hX \ar@{=}[d] \\ \rX \ar@{>->}[r] & \srX \ar@{>->}[r] & \hrX }
\]
The right-action of $M\subseteq \End(X)$ on $X$ is given by evaluation, e.g. $x.f\defeq f(x)$.\footnote{Note that the monoid operation on $\End(X)$ is given by $f\cdot g=g\circ f$.} 

\begin{definition}  Let  $X$ be a set. 
	\begin{enumerate}
		\item The theory for oriented $X$-graphs (resp. symmetric $X$-graphs, hereditary $X$-graphs) is defined as $\oG_X\defeq \DD G_{(X,\oX)}$ (resp. $\sG_X\defeq \DD G_{(X,\sX)}$, $\hG_X\defeq \DD G_{(X,\hX)}$). The category of oriented $X$-graphs (resp. symmetric $X$-graphs, hereditary $X$-graphs) is its category of presheaves $\widehat{\oG}_{X}$ (resp. $\widehat{\sG}_{X}$, $\widehat{\hG}_{X}$). 
		\item The theory for reflexive oriented $X$-graphs (resp. reflexive symmetric $X$-graphs, reflexive hereditary $X$-graphs) is defined as $\rG_X\defeq \rG_{(X,\rX)}$ (resp. $\srG_X\defeq \rG_{(X,\srX)}$, $\hrG_X\defeq \rG_{(X,\hrX)}$).\footnote{In each case, $X$ can be verified to be the submonoid of fixed elements given in the definition of a reflexive theory.} The category of reflexive oriented $X$-graphs (resp. reflexive symmetric $X$-graphs, reflexive hereditary $X$-graphs) is its category of presheaves $\widehat{\roG}_{X}$ (resp. $\widehat{\srG}_{X}$, $\widehat{\hrG}_{X}$). 
	\end{enumerate}
\end{definition}
The various categories of $X$-graphs can be thought of as models for $k$-uniform hypergraphs where $k$ is the cardinality of $X$ and the arcs take its incidence relation in multisets of vertices.  
\begin{example} \label{E:XGraphs}\mbox{}
	\begin{enumerate}
		\item (The Underlying Sets of Vertices and Arcs) When $X=\empset$, the categories of oriented, symmetric and hereditary $X$-graphs is the category $\mathbf{Set}\x \mathbf{Set}$. 
		\item \label{E:Bouq} (Bouquets) When $X=1$ is a one element set, the categories of oriented, symmetric and hereditary graphs is the category of bouquets, i.e., the category of presheaves on $\xymatrix@C=1em{V \ar[r]^s & A}$ (\cite{mR}, p 18). The categories of reflexive, reflexive symmetric and reflexive hereditary $X$-graphs is the category of set retractions.
		\item  (Conventional Categories of Graphs) When $X=\{s,t\}$, the categories of oriented, reflexive, symmetric, reflexive symmetric graphs are the categories of directed graphs, directed graphs with degenerate edges, undirected graphs with involution in \cite{rBS}. 
		
		The following is an example of a reflexive symmetric $X$-graph where $i\colon X\to X$ denotes the non-trivial automap.
				
		\begin{minipage}[t]{0.3\textwidth}
			\begin{center}\mbox{}\\\mbox{}\\
				$G\ $\framebox{ 
					$\xymatrix{ a \ar@{..>}@(ul,ur)[]^{\ell_a} \ar@<+.2em>[r]^{\alpha_0} \ar@<-.2em>@{<-}[r]_{\alpha_1} & b \ar@(d,dl)[]^{\beta_0} \ar@(d,dr)[]_{\beta_1} \ar@{..>}@(ul,ur)[]^{\ell_b} \ar@<+.2em>[r]^{\gamma_0} \ar@<-.2em>@{<-}[r]_{\gamma_1} & c \ar@{..>}@(ul,ur)[]^{\ell_c}}$}
			\end{center}
		\end{minipage}
		\begin{minipage}[t]{0.7\textwidth}
			\begin{align*}
			& G(A)= \{\alpha_0,\alpha_1,\beta_0,\beta_1,\gamma_0,\gamma_1,\ell_a,\ell_b,\ell_c \},\\
			& G(V)= \{a,\ b,\ c \}\\
			& \alpha_0.s= a,\  \alpha_0.t= b,\ \beta_0.s= b,\ \beta_0.t= b,\\
			& \gamma_0.s= c,\ \gamma_1.t= b,\\
			& a.\ell= \ell_a,\ b.\ell= \ell_b,\ c.\ell= \ell_c,\\
			& \alpha_0.i= \alpha_1,\ \beta_0.i= \beta_1,\ \gamma_0.i= \gamma_1	
			\end{align*}
		\end{minipage}
\mbox{}\\\\		
		Each loop extracted from a vertex via $\ell$ is depicted by a dotted arrow. We will call these arrows distinguished loops. They should be thought of as proxies for the vertices. Notice that for a distinguished loop $\ell_a$, we have $\ell_a.i=\ell_a$ since $\ell\circ i=\ell$ in $\srG_{X}$. However, a  non-distinguished loop may not be fixed by the right-action of $i$, as is the case with $\beta_0$ and $\beta_1$ above. If a loop $\delta$ has a distinct $i$-pair (i.e., $\delta.i\neq \delta$), we call it a unfixed loop (or a 2-loop in the case $X=2$).\footnote{In \cite{rBS}, it is called a band.} If $\delta$ is fixed by the $i$-action (i.e., $\delta.i=\delta$) it is called a fixed loop (or a 1-loop).  
		
		To connect this definition to undirected graphs, we identify edges which are $i$-pairs and define the set of edges $G(E)$ as the quotient of the set of arrows $G(A)$ under this automorphism defined by the $i$-action.\footnote{In the subsequent, we reserve the term edge for the equivalence class of arcs under the group $\sX$.} There is an incidence operator $\partial\colon G(E)\to G(V)^2$ which defines for an $i$-pair the set of boundaries. Then an undirected representation for $G$ can be given as		
		\begin{minipage}[t]{0.3\textwidth}
			\begin{center}\mbox{}\\\mbox{}\\$G\ $\framebox{ 
					$\xymatrix{ a \ar@{..}@(ul,ur)[]^{\ell_a}  \ar@{-}[r]^{\alpha_0\sim \alpha_1} & b \ar@(dl,dr)@{-}[]_{\beta_0\sim \beta_1}^2  \ar@{..}@(ul,ur)[]^{\ell_b} \ar@{-}[r]^{\gamma_0\sim \gamma_1} & c \ar@{..}@(ul,ur)[]^{\ell_c}}$
				}
			\end{center}
		\end{minipage}
		\begin{minipage}[t]{0.7\textwidth}
					\begin{align*}
			& \text{ \footnotesize $G(E)= \{\alpha_0\sim\alpha_1, \beta_0\sim\beta_1,\gamma_0\sim\gamma_1,\ell_a,\ell_b,\ell_c \}$, }\\
			& \text{ \footnotesize $G(V)= \{a,\ b,\ c \}$ }\\
			& \text{ \footnotesize $(\alpha_0\sim \alpha_1).\partial=\{a,b \},\ (\beta_0\sim\beta_1).\partial=\{b,b\},$ }\\
			& \text{ \footnotesize $(\gamma_0\sim \gamma_1).\partial= \{b,c\},\ \ell_a.\partial=\{a,a \}, \ \ell_b.\partial=\{b,b\},$}\\ 
			&\text{\footnotesize $\ell.c=\{c,c\}$.}	
			\end{align*}
		\end{minipage}

		We have placed a $2$ in the loop which came from the 2-loop $\beta_0\sim \beta_1$ even though the quotient has identified them. Keeping a distinction between fixed loops and unfixed loops is necessary for constructions of exponentials (see Chapter \ref{S:ToposStructure}, Corollary \ref{C:Exponentials} below).\footnote{In the subsequent, if a loop has no number written inside it is assumed to be a fixed loop.}
		\item (Hereditary $X$-Graphs) Since the monoid of endomaps on $2$ contains only automaps and constant maps, the category of reflexive hereditary $X$-graphs is the same as reflexive symmetric $X$-graphs. The category of hereditary $X$-graphs can be interpreted as a category of undirected graphs such that each edge is associated to a loop on its incident vertices.\footnote{Note that this is not the same as reflexive hereditary graphs. The $X$-graph with one vertex and no arcs is an example of a hereditary graph but not a reflexive hereditary graph. }
		\item \label{E:unfixedEdges}(Unfixed Edges) Let $X$ be a three element set $\{s,t,r\}$ and consider the symmetric $X$-graph $G$ with vertex set $G(V)\defeq \{a,b\}$ and arc set $G(A)\defeq \{\alpha_1,\alpha_2,\alpha_3 \}$ with right-actions given as follows
		\begin{align*}
			&\alpha_1.s=a && \alpha_1.t=b && \alpha_1.r=b&& \alpha_1.(st)=\alpha_2 && \alpha_1.(sr)=\alpha_3 && \alpha_1.(tr)=\alpha_1\\
			&\alpha_2.s=b && \alpha_2.t=a && \alpha_2.r=b && \alpha_2.(st)=\alpha_1 && \alpha_2.(sr)=\alpha_2 && \alpha_2.(tr)=\alpha_3\\
			&\alpha_3.s=b && \alpha_3.t=b && \alpha_3.r=a && \alpha_3.(st)=\alpha_3 && \alpha_3.(sr)=\alpha_1 && \alpha_3.(tr)=\alpha_2
		\end{align*}
		where $(st)$, $(sr)$ and $(rt)$ are the generators of the group of automaps of $X$ which permute the two elements.
		
		Next, consider the symmetric $X$-graph $G'$ with vertex set $G'(V)\defeq \{a,b\}$ and arc set $G'(A)=\Aut(X)$ such that $\sigma.x\defeq\begin{cases}a & \text{if } \sigma(x)=s\\ b & \text{if } \sigma(x)\neq s \end{cases}$ and $\sigma.\psi=\psi\circ \sigma$ for $\sigma, \psi\in G'(A)$.
		
		Then the undirected representation for $G$ and $G'$ can be given as 
		\begin{center}$G\ $\framebox{ 
				$\xymatrix{ a \ar@{-}[rr]_>{\tiny 2}^{\text{\tiny $\alpha_1\sim\alpha_2\sim\alpha_3$}} && b}$ 
			}$\quad\qquad $$G'\ $\framebox{ 
			$\xymatrix{ a \ar@{-}[rr]^{\text{\tiny $\alpha\defeq\Aut(X)$}} && b \ar@{-}@(ur,dr)[]_{\text{\tiny 2}}^{\alpha} }$ 
		}
		\end{center}
		
		Note that the edge in $G$ consists of $3$ arcs while the edge in $G'$ consists of $6$ arcs. We have placed a lobe at the vertex $b$ with a 2 inside to indicate that even though $b$ has multiplicity 2 in the incidence, they are not identified at the level of arcs. 
		
		Also notice that $G$ is the coequalizer of $\Id, f\colon G'\to G'$ where $f_V=\Id_V$ and $f_A(\sigma)$ is the arc corresponding to the automap 
		
		\[
		\psi(x)\defeq  \begin{cases}r & \text{if }\sigma(x)=t \\ t & \text{if }\sigma(x)=r \\ s & \text{if }\sigma(x)=s   \end{cases}.
		\]
		This shows how the edge in $G'$ allows internal symmetry (i.e., allows nontrivial automorphisms) while $G$ does not. 
		
		In general, if the arcs which represent an edge in a symmetric $X$-graph $G(V)$ can be generated by a set map $j\colon X\to G(V)$, i.e., there is exactly one arc with incidence of the form $j\circ \sigma$ for some automap $\sigma\in \Aut(X)$, then we say the edge is fixed. For unfixed edges we use a loop with the number of distinct repeated instances of the vertex is to be considered inside of it, e.g., in $G'$ the vertex $b$ should be counted twice in the incidence of the edge.  If there is more than one arc with the same incidence in the equivalence class, then we say the edge is unfixed. This generalizes fixed and unfixed loops.\footnote{In the subsequent, all edges will be considered fixed unless otherwise specified.}
		\item\label{E:MultVert} (Effect of Unfixed Edges on Morphisms)
		Let $X$ be the six element set $\{0,1,\dots, 5\}$ and consider a symmetric $X$-graph $H$ with two vertices $a$ and $b$ and one edge represented by an arc $\alpha$ such that $\alpha.0=\alpha.1=a$ and $\alpha.n=b$ for $n\neq 0,1$.\footnote{More precisely, the edge is the equivalence class $[\alpha]$ where $\alpha\sim \alpha'$ iff there is an automap $\sigma\colon X\to X$ such that $\alpha.\sigma=\alpha'$.} We represent $H$ by the following diagram
		\begin{center}
			$H\ $\framebox{ 
				$\xymatrix{ a \ar@{-}[r]^\alpha \ar@{}[r]_<{\text{\tiny 2}} \ar@{}[r]_>{\text{\tiny 4}} & b}$
			}
		\end{center}
		Notice that we have placed the incidence multiplicities under the edge $\alpha$. When $\alpha$ is a fixed edge, there are $\frac{6!}{2!4!}=15$ arcs in the equivalence class. If a symmetric $X$-graph $H'$ has a unfixed edge at $a$, i.e., another arc in the equivalence class $\beta$ such that $\partial_G(\beta)=\partial_G(\alpha)$ and $\alpha.(01)=\beta$ for the permutation which swaps the $0$ and $1$-incidence vertices, then we represent $H$ by  
				\begin{center}
					$H'\ $\framebox{ 
						$\xymatrix{ a \ar@{-}@(dl,ul)[]^{\alpha}_{\text{\tiny 2}} \ar@{-}[r]^\alpha \ar@{}[r]_>{\text{\tiny 4}} & b}$
					}
				\end{center}
				to indicate the presence of a unfixed edge. In this case, the equivalence class of $\alpha$ has $\frac{6!}{4!}=30$ arcs. We also have the following examples of unfixed edges
						\begin{center}
							\framebox{ 
								$\xymatrix{ a \ar@{-}@(dl,ul)[]^{\alpha}_{\text{\tiny 2}} \ar@{-}[r]^\alpha \ar@{}[r]_>{\text{\tiny 2}} & b \ar@{-}@(ur,dr)[]^{\alpha}_{\text{\tiny 2}}}$
							}$\qquad \qquad$\framebox{ 
							$\xymatrix{ a \ar@{-}@(dl,ul)[]^{\alpha}_{\text{\tiny 2}} \ar@{-}[r]^\alpha & b \ar@{-}@(ur,dr)[]^{\alpha}_{\text{\tiny 4}}}$
						}
						\end{center}		
						where the first indicates the unfixed edge which has kept a distinction between two incidence at vertex $b$, i.e., has $\frac{6!}{2!}=360$ arcs while the other has $6!$ arcs in the equivalence class of $\alpha$.
						
		Next, consider the symmetric $X$-graph $G$ with two vertices $a$ and $b$ and one fixed edge represented by an arc $\beta$ such that $\beta.0=a$ and $\beta.n=b$ for $n\neq 0$. 
		\begin{center}
			$G\ $\framebox{ 
				$\xymatrix{ a \ar@{-}[r]^\beta \ar@{}[r]_<{\text{\tiny 1}} \ar@{}[r]_>{\text{\tiny 5}} & b}$
			} 
		\end{center}
		Observe that while the symmetric graphs $H$ and $G$ have the same underlying set of vertices with one edge connecting them, there is no morphism between them because morphisms must preserve incidence.
		
		For symmetric $X$-graphs $K$ and $M$ below, notice that the morphism $f$ which sends vertices $a, b, c$ to $a, b, b$ must preserve the incidence multiplicities, i.e., the contraction of $b$ and $c$ to $b$ must respect the addition $3+1=4$ of multiplicities of incidence.
		\begin{center}
			$K\ $\framebox{ 
				$\xymatrix{ a \ar@{-}[r]^{\alpha} \ar@{}[r]_<{\text{\tiny 2}}  & b \ar@{-}[r]^{\alpha} \ar@{}[r]_<{\text{\tiny 3}} \ar@{}[r]_>{\text{\tiny 1}}& c }$} $\xymatrix{ \ar[r]^{f} &}$
			\framebox{ 
				$\xymatrix{ a \ar@{-}[r]^{\beta} \ar@{}[r]_<{\text{\tiny 2}} \ar@{}[r]_>{\text{\tiny 4}}& b}$} $\xymatrix{ M}$
			\end{center}

	\item (Hereditary $X$-Graphs)The category of (reflexive) hereditary $X$-graphs consists of objects $G$ such that each arc $\alpha \in G(A)$ is associated to arcs $\alpha.f$ where $f\colon X\to X$ is any endomap. In particular, for each non-bijective surjection $f$, $\alpha.f$ has incident contained in the incident set of $\alpha$. Thus the category of (reflexive) hereditary is the $X$-graph analogue of hereditary hypergraphs as defined in \cite{dM}.  Observe that in the reflexive case, for each constant map $f\colon X\to X$, $\alpha.f$ must be a distinguished loop. 

	\end{enumerate}
\end{example}

\section{Directed $X$-Graphs}\label{S:}
Next, we describe theories for directed $(X,M)$-graphs. Let $X$ be a set and $(S,T)$ be a partition of $X$, i.e., $S,T\subseteq X$ such that $S\cup T=X$ and $S\cap T=\empset$. Define the following submonoids of $\End(X)$:
\begin{align*}
& \doX\defeq \{\Id_X\}\\
& \dsX \defeq \Aut(S)\cup \Aut(T) \quad \ \text{(as submonoids of $\End(X)$)}\\
& \dhX\defeq \End(S)\cup \End(T) \quad \text{(as submonoids of $\End(X)$)}\\
& \drX \defeq \rX \quad \text{(as above)}\\
& \dsrX \defeq \drX \cup \dsX \\
& \dhrX \defeq \drX \cup \dhX
\end{align*}
In the definition of $\dsX$ and $\dhX$ an endomap $\sigma\colon S\to S$ is extended to an endomap $\overline \sigma\colon X\to X$ by assigning $\overline \sigma(x)\defeq\begin{cases}\sigma(x) & x\in S\\ x & x\nin S \end{cases}$ and similarly for $T$. 

There is the following inclusions as submonoids in $\End(X)$
\[
\xymatrix@R=1em@C=1.25em{ \doX \ar@{>->}[d] \ar@{>->}[r]  & \dsX \ar@{>->}[r] \ar@{>->}[d] & \dhX \ar@{>->}[d] \\ \drX \ar@{>->}[r] & \dsrX \ar@{>->}[r] & \dhrX }
\]

\begin{definition}
	Let $(S,T)$ be a partition of a set $X$. 
	\begin{enumerate}
		\item The theory for $(S,T)$-directed oriented $X$-graphs (resp. $(S,T)$-directed symmetric $X$-graphs, $(S,T)$-directed hereditary $X$-graphs) is defined to be $\doG_{(S,T)}\defeq \oG_X$ \footnote{The difference between $\doG_{(S,T)}$ and $\oG_X$ is in the interpretation of the incidence. In $\doG_{(S,T)}$, each incidence in $S$ is a source incidence and in $T$ is a target incidence.} (resp. $\dsG_{(S,T)}\defeq \DD G_{(X,\text{ds}(X))}$, $\dhG_{(S,T)}\defeq \DD G_{(X,\text{dh}(X))}$). The category of $(S,T)$-directed oriented $X$-graphs (resp. $(S,T)$-directed symmetric $X$-graphs, $(S,T)$-directed hereditary $X$-graphs) is its category of presheaves $\widehat{\doG}_{(S,T)}$ (resp. $\widehat{\dsG}_{(S,T)}$, $\widehat{\dhG}_{(S,T)}$).
		\item 	The theory for reflexive $(S,T)$-directed oriented $X$-graphs (resp. reflexive $(S,T)$-directed symmetric $X$-graphs, reflexive $(S,T)$-directed hereditary $X$-graphs) is defined to be \\$\droG_{(S,T)} \defeq \roG_X$ (resp. $\dsrG_{(S,T)}\defeq \DD G_{(X,\text{ds}(X))}$, $\dhrG_{(S,T)}\defeq \DD G_{(X,\text{dh}(X))}$). The category of $(S,T)$-directed oriented $X$-graphs (resp. $(S,T)$-directed symmetric $X$-graphs, $(S,T)$-directed hereditary $X$-graphs) is its category of presheaves $\widehat{\droG}_{(S,T)}$ (resp. \\$\widehat{\dsrG}_{(S,T)}$, $\widehat{\dhrG}_{(S,T)}$).
	\end{enumerate}
\end{definition}
We represent the theory for $(S,T)$-directed $X$-graphs and reflexive $(S,T)$-directed $X$-graphs by the following diagrams where $M$ is the appropriate monoid.
\[ 
\xymatrix{ V \ar@<.3em>[r]^S \ar@<-.3em>[r]_{T} & A \ar@(ur,dr)^M }\qquad\qquad \xymatrix{ V \ar@<.5em>[r]^S  \ar[r]|-T \ar@{<-}@<-.5em>[r]_{\ell} & A \ar@(ur,dr)^M }
\]
\begin{example}\mbox{}
	\begin{enumerate}
		\item  ($B$-Hypergraph)
		A $(S,1)$-directed symmetric $X$-graph (resp. $(1,T)$-directed symmetric $X$-graph) is the $(X,M)$-graph analogue of a $B$-hypergraph (resp.  a $F$-hypergraph) as defined in \cite{gG}.
		
		For instance, let $(S,T)=(\{s, t \}, p)$ and consider the $(S,T)$-directed symmetric $X$-graph $G$. There is an incidence operator $\partial$ which defines for the set of boundaries for an edge $\alpha$ as a pair $(\{a,b\}, c)$ where $c$ is the $p$-incidence vertex of $\alpha$ called the head of $\alpha$. Then a visual representation for $G$ can be given as
		\newpage
		\begin{minipage}[t]{0.3\textwidth}
				\begin{center}\mbox{}\\\mbox{}\\$G\ $\framebox{ 
					$\xymatrix@C=1em@R=.75em{ a  \ar@{..}[dd]|-{\alpha_0\sim\alpha_1} \ar@/_/[dr] & &  d \ar@<+.3em>[r]^\rho \ar@<-.3em>[r]_\omega \ar@{..}[dd]|-{\gamma_0\sim \gamma_1} \ar@/^/[dl]  & f \\ & c \ar@(ul,ur)[]^\delta \ar@(dl,dr)[]_{\beta_0\sim \beta_1}^2 & \\  b \ar@/^/[ur] & & e  \ar@/_/[ul]}$
				}
			\end{center}
		\end{minipage}
		\begin{minipage}[t]{0.7\textwidth}
			\begin{align*}
			& \text{ \footnotesize $G(E)= \{\alpha_0\sim\alpha_1,\ \beta_0\sim\beta_1,\ \gamma_0\sim\gamma_1,\ \delta,\ \rho,\ \omega \},$}\\
			& \text{ \footnotesize $G(V)= \{a,\ b,\ c, \ d,\ e,\ f \}$}\\
			& \text{ \footnotesize $(\alpha_0\sim \alpha_1).\partial=(\{a,b \},c),\  (\beta_0\sim\beta_1).\partial=(\{c\},c),$}\\
			&\text{ \footnotesize $(\gamma_0\sim \gamma_1).\partial= (\{d,e\},c),\ \delta.\partial=(\{c\},c), \ $} \\ 
			&\text{ \footnotesize $\rho.\partial=(\{d\},f),\  \omega.\partial=(\{d\},f).	$}
			\end{align*}
		\end{minipage}		
		
		\item (Directed $X$-Graphs Generalize $X$-Graphs) Any $(1,1)$-directed $X$-graph is the category of oriented $2$-graphs. Any $(X,\empset)$-directed or $(\empset, X)$-directed $X$-graph is equivalent to its $X$-graph counterpart.
	\end{enumerate}
\end{example}

\chapter{Comparison Functors}\label{S:ComparisonFunctors}

The theories for categories of $X$-graphs and directed $X$-graphs have natural comparison functors between them. This generalizes the construction of essential geometric morphisms constructed  in \cite{rB}\footnote{In \cite{rB}, $\underline A$, $\underline B$, $\underline C$, $\underline D$ (defined on pages 2, 3, 12, 13, respectively) denote the categories we denote as $\oG_{2}$, $\rG_2$, $\sG_2$, $\srG_2$ respectively. The inclusion functors $\underline A\hookrightarrow \underline B$, $\underline A\hookrightarrow \underline C$, and $\underline B\hookrightarrow \underline D$ (on pages 4, 13, 13 respectively) are the components of the natural inclusions  of the left hand square of Proposition \ref{P:GeoFunc} for $X=2$.} between the categories of directed graphs, reflexive graphs, and graphs with involution.

\begin{proposition}\label{P:GeoFunc}
	Let $\mathbf{Set}_{\text{\normalfont \tiny inj}}$, $\mbf{Set}_{\text{\normalfont \tiny inj}}^2$, $\Cat$ be the categories of sets and injective maps, the product of categories $\mathbf{Set}_{\text{\normalfont \tiny inj}}\x \mathbf{Set}_{\text{\normalfont \tiny inj}}$, and small categories. Then there are functors 
	\begin{align*}
	&\oG,\ \sG,\ \rG,\ \srG,\ \hG,\ \hrG\colon \mathbf{Set}_{\text{\normalfont \tiny inj}}\to \Cat,\\
	& \doG,\ \dsG,\ \droG,\ \dsrG,\ \dhG,\ \dhrG \colon \mbf{Set}_{\text{\normalfont \tiny inj}}^2\to \Cat
	\end{align*}
	and natural inclusions
	\[
	\xymatrix{  \oG \ar@{=>}[r] \ar@{=>}[d] & \sG  \ar@{=>}[d] \ar@{=>}[r] & \hG \ar@{=>}[d] \\ 
		\rG \ar@{=>}[r]  & \srG  \ar@{=>}[r] & \hrG }\qquad \qquad 
	\xymatrix{  \doG \ar@{=>}[r] \ar@{=>}[d] & \dsG  \ar@{=>}[d] \ar@{=>}[r] & \dhG \ar@{=>}[d] \\ 
		\droG \ar@{=>}[r]  & \dsrG  \ar@{=>}[r] & \dhrG }
	\]
	
\end{proposition}
\begin{proof}
	The inclusion of submonoids of $\End(X)$ and extensions of monoid actions induce functors of theories via Proposition \ref{P:Equiv}. 
	For example, if $j\colon X'\hookrightarrow X$ is an injective set map, then there is a functor $\overline j\colon \sG_{X'}\to \sG_{X}$ which sends the morphism $x'\colon V\to A$ to $j(x')\colon V\to A$ and $m'\colon A\to A$ to $j_!(m')\colon A\to A$ where $j_!(m')(x')=\begin{cases} x'.m' & x'\in X'\\
	x' & x'\nin X' \end{cases}$. 
\end{proof}
Thus we obtain the following corollary. 
\begin{corollary}\label{C:Essential}
	For each set inclusion $j\colon X'\hookrightarrow X$, respectively $\mathbf{Set}^2$-inclusion $k\colon (S',T')\hookrightarrow (S,T)$, there are essential geometric morphisms which commute up to natural isomorphism
	\[ 
	\xymatrix@!=.2em{ \widehat{\oG}_{X'} \ar[rrr]^{j_o} \ar[dr] \ar[dd] & && \widehat{\oG}_X \ar[dr] \ar[dd]|-\hole \\ 
		& \widehat{\rG}_{X'} \ar[rrr]^{j_r} \ar[dd] &&& \widehat{\rG}_X \ar[dd] \\ \widehat{\sG}_{X'} \ar[dd] \ar[rrr]^{j_s}  \ar[dr] & && \widehat{\sG}_{X} \ar[dr] \ar[dd]|-\hole  \\ & \widehat{\srG}_{X'} \ar[dd] \ar[rrr]^{j_{sr}}  & && \widehat{\srG}_{X}  \ar[dd] \\ \widehat{\hG}_{X'} \ar[rrr]^{j_s}  \ar[dr] & && \widehat{\hG}_X \ar[dr] \\ & \widehat{\hrG}_{X'} \ar[rrr]^{j_{sr}}  & && \widehat{\hrG}_X 
	} \qquad 
	\xymatrix@!=.2em{ \widehat{\doG}_{X'} \ar[rrr]^{k_o} \ar[dr] \ar[dd] & && \widehat{\doG}_X \ar[dr] \ar[dd]|-\hole \\ 
		& \widehat{\droG}_{X'} \ar[rrr]^{k_r} \ar[dd] &&& \widehat{\droG}_X \ar[dd] \\ \widehat{\dsG}_{X'} \ar[dd] \ar[rrr]^{k_s}  \ar[dr] & && \widehat{\dsG}_{X} \ar[dr] \ar[dd]|-\hole  \\ & \widehat{\dsrG}_{X'} \ar[dd] \ar[rrr]^{k_{sr}}  & && \widehat{\dsrG}_{X}  \ar[dd] \\ \widehat{\dhG}_{X'} \ar[rrr]^{k_s}  \ar[dr] & && \widehat{\dhG}_X \ar[dr] \\ & \widehat{\dhrG}_{X'} \ar[rrr]^{k_{sr}}  & && \widehat{\dhrG}_X 
	} 
	\]
\end{corollary}

An essential geometric morphisms between different categories of $X$-graph with a fixed set $X$, e.g., $\widehat{\sG}_X\to \widehat{\srG}_X$, are to be considered as modifications of structure and essential geometric morphisms induced by a set inclusion, e.g., $\sG_{X'}\to \sG_X$, as a modification of data.

\begin{example}\label{E:OSG}\mbox{}
	\begin{enumerate}
		\item \label{E:Reflexive} (The Free Reflexive $X$-Graph Functor) Let $r\colon \oG_{X}\to \roG_{X}$ be the inclusion functor and let $r_!\dashv r^*\dashv r_*\colon \widehat{\oG}_X\to \widehat{\roG}_X$ be the essential geometric morphism induced by $r$. The $r$-extension $r_!$ takes an oriented $X$-graph $G$ and freely adds a distinguished loop to each vertex $v\in G(V)$. The $r$-restriction functor $r^*$ forgets that distinguished loops are distinguished. The $r$-coextension functor $r_*$ gives the set of vertices $r_*(G)(V)=\setm{v_\gamma}{\gamma\text{ is a loop}}$ and the set of arcs $r_*(G)(A)=\widehat{\oG}_X(r^*(\underline A),G)$ which is the set of morphisms from the oriented graph with a loop at each vertex to $G$. A distinguished loop is the the morphism which sends every loop to a particular loop in $G$. 
		
		For instance, when $X=\{s,t\}$ we have the following assignments for an oriented $X$-graph $G$. 
\end{enumerate}

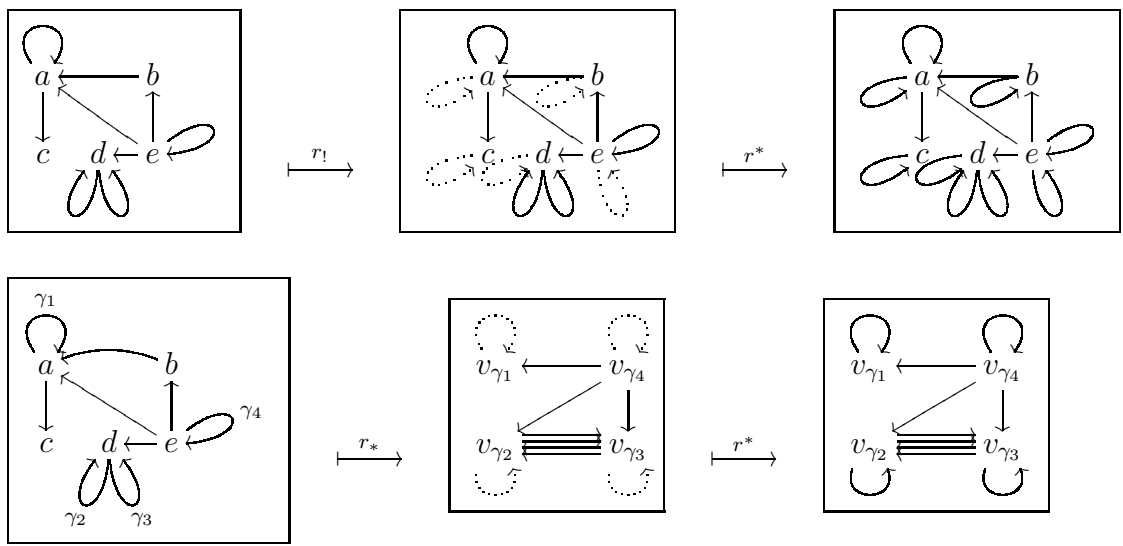
	\begin{center}
				\begin{tabular}{ c c  c c c } 
						\framebox{ 
										$\xymatrix@R=1.3em@C=.75em{  a \ar[d] \ar@(ul,ur)[]^{\mbox{}} \ar@{<-}[rr] && b\\ c & d \ar@(d,dl)[]_{\mbox{}} \ar@(d,dr)[] & e \ar[ull] \ar[l] \ar[u] \ar@(ur,r)[]^{\mbox{}} }$
									} &	$ \xymatrix{ \ar@{|->}@<-3em>[r]^{r_!} & }$ &	 \framebox{ 
									$\xymatrix@R=1.3em@C=.75em{  a \ar@{..>}@(l,dl)[]^{\mbox{}} \ar[d] \ar@(ul,ur)[]^{\mbox{}} \ar@{<-}[rr] && b \ar@{..>}@(l,dl)[]\\ c \ar@{..>}@(l,dl)[]  & d \ar@{..>}@(l,dl)[] \ar@(d,dl)[]_{\mbox{}} \ar@(d,dr)[]_{} & e \ar[ull] \ar@{..>}@(d,dr)[] \ar[l] \ar[u] \ar@(ur,r)[]^{\mbox{}} }$}  & $\xymatrix{ \ar@<-3em>@{|->}[r]^{r^*} & }$ 
								& \framebox{ $\xymatrix@R=1.3em@C=.75em{   a \ar@(l,dl)[]^{\mbox{}} \ar[d] \ar@(ul,ur)[]^{\mbox{}}  \ar@{<-}[rr] && b \ar@(l,dl)[]\\ c \ar@(l,dl)[]  & d \ar@(l,dl)[] \ar@(d,dl)[]_{\mbox{}} \ar@(d,dr)[]_{\mbox{}} & e \ar[ull]\ar@(d,dr)[] \ar[l] \ar[u] \ar@(ur,r)[]^{\mbox{}} }$ } 
								\\[1ex] 
							\end{tabular}
						\end{center}
					\begin{center}
						\begin{tabular}{ c c  c c c } 
							\framebox{ 
								$\xymatrix@R=1.3em@C=1em{  a \ar[d] \ar@(ul,ur)[]^{\gamma_1} \ar@{<-}@/^/[rr] && b\\ c & d \ar@(d,dl)[]^{\gamma_2} \ar@(d,dr)[]_{\gamma_3} & e \ar[ull] \ar[l] \ar[u] \ar@(ur,r)[]^{\gamma_4} }$
							} &	$ \xymatrix{ \ar@{|->}@<-3em>[r]^{r_*} & }$ &	 \framebox{ 
							$\xymatrix@R=1.3em@C=1em{  v_{\gamma_1} \ar@{..>}@(ul,ur)[]^{\mbox{}} && v_{\gamma_4} \ar@{..>}@(ul,ur)[] \ar[d] \ar[ll] \ar[dll] \\ v_{\gamma_2} \ar@<+.3em>[rr] \ar@<+.1em>[rr] \ar@{<-}@<-.3em>[rr] \ar@{<-}@<-.1em>[rr] \ar@{..>}@(dl,dr) && v_{\gamma_3} \ar@{..>}@(dl,dr)_{\mbox{}} }$}  & $\xymatrix{ \ar@<-3em>@{|->}[r]^{r^*} & }$ 
						& \framebox{ $\xymatrix@R=1.3em@C=1em{   v_{\gamma_1} \ar@(ul,ur)[]^{\mbox{}} && v_{\gamma_4} \ar@(ul,ur)[] \ar[d] \ar[ll] \ar[dll] \\ v_{\gamma_2} \ar@<+.3em>[rr] \ar@<+.1em>[rr] \ar@{<-}@<-.3em>[rr] \ar@{<-}@<-.1em>[rr] \ar@(dl,dr) && v_{\gamma_3} \ar@(dl,dr)_{\mbox{}} }$ } 
						\\[1ex] 
					\end{tabular}
				\end{center}

\begin{enumerate}
		\item[] Notice that there are four arcs between $v_{\gamma_2}$ and $v_{\gamma_3}$ in $r_*(G)$. These correspond to the four morphisms $\widehat{\oG}_X(r^*(\underline A),G)$ which map the two loops in $r^*(\underline A)$ to $\gamma_2$ and $\gamma_3$.
		
		In general, for an oriented $X$-graph $G$ the unit $\eta_G\colon G\to r^*r_!(G)$ of the adjunction $r_!\dashv r^*$ is a monomorphism while the counit $\varepsilon'_G\colon r^*r_*(G)\to G$ of $r^*\dashv r_*$ is neither a monomorphism nor an epimorphism.  For a reflexive oriented $X$-graph $H$, the counit $\varepsilon_H\colon r_!r^*(H)\to H$ of $r_!\dashv r^*$ is a monomorphism and the unit $\eta'_H\colon H\to r_*r^*(H)$ is neither a monomorphism not an epimorphism.

		\item \label{StructDual} (The Dualization of Data Functor) Let $\DD G_X$ be any of the theories above and let $\sigma\colon X\to X$ be an automap. In this case, the induced adjunctions $\sigma_!\dashv \sigma^*$ and $\sigma^*\dashv \sigma_*$ are isomorphisms with $\sigma_!=\sigma_*$. If $\sigma$ is an involution, we call $\sigma^*\colon \widehat{\DD G}_X\to \widehat{\DD G}_X$ a dualization of structure with respect to $\sigma$. In the case $X=2$, the dualization of $\widehat{\roG}_X$ is the usual dualization of the underlying reflexive graph of a small category (cf, dualization of data in Chapter \ref{S:HyM}).

		\item (The Target/Source Collapsing Functors) Let $\psi\defeq (\Id_S,!_T) \colon (S,T)\to (S,1)$ be the map in $\mathbf{Set}^2$ where $!_T\colon T\to 1$ is the terminal map. Then there is an essential geometric morphism $\psi_!\dashv \psi^*\dashv \psi_*\colon \dsG_{(S,T)}\to \dsG_{(S,1)}$. The $\psi$-extension takes an $(S,T)$-directed symmetric $X$-graph $G$ to the $(S,1)$-directed symmetric $X$-graph with vertex set $\psi_!G(V)=\frac{G(V)}{\sim}$ where $\sim$ is the equivalence relation on $G(V)$ such that $v\sim v'$ iff there is an arc $\alpha\in G(A)$ such that $v=\alpha.t$ and $v'=\alpha.t'$ for some $t, t'\in T$ and arc set $\psi_!G(A)=G(A)$ with restricted right-action, e.g., if $\alpha\in \psi_!G(A)$ then $\alpha.1=[\alpha.t]$ where $[\alpha.t]$ is the equivalence class of $\alpha.t$ for some $t\in T$. The $\psi$-restriction functor $\psi^*$ takes the $(S,1)$-directed symmetric $X$-graph $H$ to the $(S,T)$-directed symmetric $X$-graph with vertex set $H(V)$ and arc set $H(A)$ with right-action $\alpha.t=\alpha.1$ for each $t\in T$. The $\psi$-coextension functor $\psi_*$ takes the $(S,T)$-directed symmetric $X$-graph $G$ to the $(S,1)$-directed symmetric $X$-graph with vertex set $G(V)$ and arc set $\psi_*G(A)=\widehat{\DD G}_{(S,1)}(\psi^*(\underline A), G)$ which is the set of arcs $\alpha$ in $G$ with fixed target vertex, i.e., there is some $v\in G(V)$ such that for each $t\in T$, $\alpha.t=v$.  	
				
			Similarly if $\theta\defeq (!_S,\Id_T)\colon (S,T)\to (1,T)$ is the map in $\mathbf{Set}^2$ where $!_S\colon S\to 1$ is the terminal map. Then the essential geometric morphism $\theta_!\dashv \theta^*\dashv \theta_*\colon \dsG_{(S,T)}\to \dsG_{(1,T)}$ is analogous to that induced by $\psi$.
			
		\item \label{E:OSGraph} (The Symmetrization Functor) Let $p\colon \oG_{X}\to \sG_{X}$ be the inclusion functor and $p_!\dashv p^*\dashv p_*\colon \widehat{\oG}_X\to \widehat{\sG}_{X}$ the induced essential geometric morphism. Then the $p$-extension functor $p_!$ freely adds symmetric partners to each oriented arc, the $p$-restriction functor $p^*$ forgets the symmetric partners of each arc, and the $p$-coextension functor $p_*$ takes an oriented $X$-graph $G$ to the the symmetric $X$-graph $p_*(G)$ which has the same vertex set and the arc set $p_*(G)(A)\defeq \widehat{\oG}_X(p^*(\underline A), G)$, i.e., the set of possible symmetric partners between vertices. 
		
		For instance, when $X=\{s,t\}$ we have the following assignments for an oriented $X$-graph $G$.
\end{enumerate}
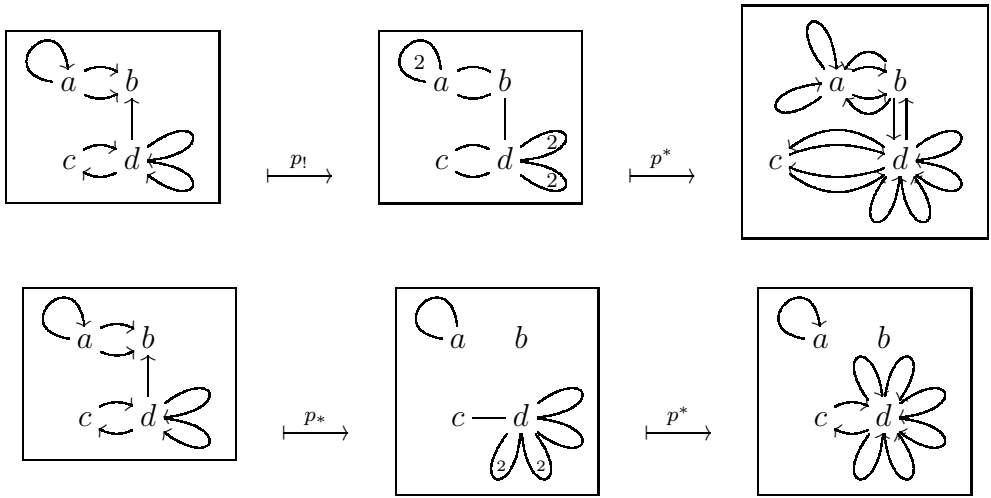
		\begin{center}
			\begin{tabular}{ c c  c c c } 
				\framebox{ 
					$\xymatrix@R=1.3em@C=1em{  a \ar@(l,u)[]^{\mbox{}} \ar@/^/[r]\ar@/_/[r]  & b  \\ c \ar@/^/[r] & d \ar@/^/[l] \ar@(ur,r)[] \ar@(r,dr)[]^{\mbox{}} \ar[u] }$
				} &	$ \xymatrix{ \ar@{|->}@<-3em>[r]^{p_!} & }$ &	 \framebox{ 
				$\xymatrix@R=1.3em@C=1em{  a \ar@{-}@(l,u)[]^{\mbox{}}_2 \ar@{-}@/^/[r]\ar@{-}@/_/[r]  & b  \\ c \ar@{-}@/^/[r] & d \ar@{-}@/^/[l] \ar@{-}@(ur,r)[]_2 \ar@{-}@(r,dr)[]^{\mbox{}}_2 \ar@{-}[u]}$}  & $\xymatrix{ \ar@<-3em>@{|->}[r]^{p^*} & }$ 
			& \framebox{ $\xymatrix@R=1.3em@C=1em{  & a \ar@(ul,u)[]^{\mbox{}} \ar@(dl,l)[]^{\mbox{}} \ar@/^/[r]\ar@/_/[r] \ar@{<-}@/^1em/[r] \ar@{<-}@/_1em/[r] & b \ar@<-.2em>[d]  \\ c  \ar@/^/[rr] && d \ar@{<-}@/^1em/[ll] \ar@/_1em/[ll] \ar@/^/[ll] \ar@(ur,r)[] \ar@(dl,d)[]_{\mbox{}} \ar@(r,dr)[]^{\mbox{}} \ar@<-.2em>[u] \ar@(d,dr)[] }$
			} 
			\\[1ex] 
		\end{tabular}
	\end{center}
\begin{center}
	\begin{tabular}{ c c  c c c } 
		\framebox{ 
			$\xymatrix@R=1.3em@C=1em{  a \ar@(l,u)[]^{\mbox{}} \ar@/^/[r]\ar@/_/[r]  & b  \\ c \ar@/^/[r] & d \ar@/^/[l] \ar@(ur,r)[] \ar@(r,dr)[]^{\mbox{}} \ar[u] }$
		} &	$ \xymatrix{ \ar@{|->}@<-3em>[r]^{p_*} & }$ &	 \framebox{ 
		$\xymatrix@R=1.3em@C=1em{  a \ar@{-}@(l,u)[]^{\mbox{}}  & b  \\ c \ar@{-}[r] & d \ar@{-}@(ur,r)[] \ar@{-}@(r,dr)[]^{\mbox{}} \ar@{-}@(dl,d)[]_{\mbox{}}^{\text{\tiny 2}} \ar@{-}@(d,dr)[]^{\text{\tiny 2}}}$}  & $\xymatrix{ \ar@<-3em>@{|->}[r]^{p^*} & }$ 
	& \framebox{ $\xymatrix@R=1.3em@C=1em{  a \ar@(l,u)[]^{\mbox{}}  & b  \\ c \ar@/^/[r] \ar@{<-}@/_/[r] & d \ar@(ur,r)[] \ar@(r,dr)[]^{\mbox{}} \ar@(dl,d)[]_{\mbox{}} \ar@(d,dr)[] \ar@(ul,u)[] \ar@(u,ur)[] }$
	} 
	\\[1ex] 
\end{tabular}
\end{center}
\begin{enumerate}[resume]
		\item \label{E:Assoc} (The Incidence Expansion Functor) Let $\gamma\colon X'\to X$ be an injective map and $\gamma_!\dashv \gamma^*\dashv \gamma_*\colon \widehat{\sG}_{X'}\to \widehat{\sG}_X$ be the induced essential geometric morphism. The $\gamma$-extension functor $\gamma_!$ takes a symmetric $X'$-graph $G$ to the symmetric $X$-graph with vertex set $\gamma_!(G)(V)=G(V)\sqcup\left(\frac{G(A)}{\sim}\x X\backslash X'\right)$ where $\sim$ is the equivalence relation induced by the automaps of $X'$. To describe the arc set of $\gamma_!(G)$, define 
		
		\[ 
		Q\defeq\setm{\partial_G(\alpha)\sqcup\overline{[\alpha]}\in (\gamma_!G(V))^X}{\alpha\in G(A)}
		\] 
		where $\overline{[\alpha]}\colon X\backslash X'\to \frac{G(A)}{\sim}\x X\backslash X'$ maps $x\mapsto ([\alpha],x)$. Then \[\gamma_!(G)(A)=\setm{f\in (\gamma_!G(V))^X}{\exists g\in Q,\exists \sigma\in \Aut(X),\  f=g\circ \sigma}\] with right-actions $f.x=f(x)$ for $x\in X$ and $f.\sigma=f\circ \sigma$ for $\sigma\in \sX$. In other words, the $\gamma$-extension extended each edge by freely adding $X\backslash X'$ vertices to the incidence of each edge. The $\gamma$-restriction functor $\gamma^*$ sends the symmetric $X$-graph to the symmetric $X'$-graph with vertex set $G(V)$ and arc set $G(A)$ with restricted incidence $\partial_{\gamma^*G}(\alpha)=\partial_G(\alpha)_{|X'}$ and induced action. In other words we lose some coherency between arcs creating more edges in $\gamma^*(G)$\footnote{In fact, for each symmetric $X$-graph edge $\alpha$, there is a symmetric $X'$-graph clique between the incidence vertices in $\alpha$.}(cf, \cite{wD} (Proposition 3.3)).  The $\gamma$-coextension $\gamma_*$ sends the symmetric $X'$-graph $G$ to the symmetric $X$-graph with vertex set $\gamma_*(G)(V)=G(V)$ and arc set $\gamma_*(G)(A)=\sG_{X'}(\gamma^*(\underline A), G)$, i.e., the set of $X'$-cliques in $G$. For example, if $X$ has cardinality $n$ and $X'$ has cardinality $n-1$, $\gamma^*(\underline A)$ is a symmetric $X'$-graph with $n$ edges such that every $n-1$ collection of vertices in $\gamma^*(\underline A)$ is connected by exactly one edge, i.e., $\gamma^*(\underline A)$ is a complete symmetric $X'$-graph on $n$ vertices.
		
		For instance, let $X=\{s,t,r\}$ and $X'=\{s,t\}$ with $\gamma\colon X'\hookrightarrow X$ and consider the $\gamma$-extension followed by the $\gamma$-restriction on a symmetric $X'$-graph $G$. 
\end{enumerate}

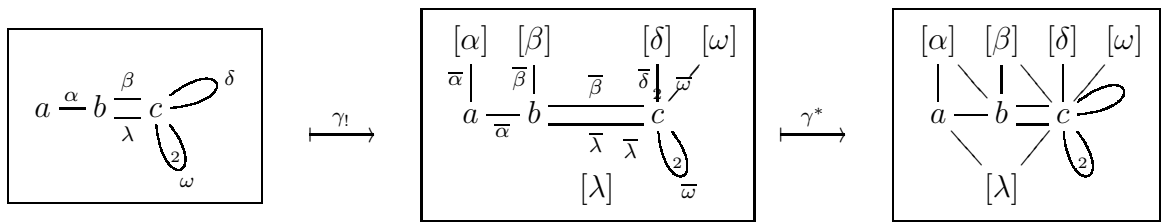
			\begin{center}
				\begin{tabular}{ c c  c c c } 
					\framebox{ 
						$\xymatrix@R=1.3em@C=.85em{ \\ a \ar@{-}[r]^\alpha & b \ar@<+.3em>@{-}[r]^\beta \ar@<-.3em>@{-}[r]_\lambda & c \ar@{-}@(ur,r)[]^\delta \ar@{-}@(dr,d)[]^\omega_{\text{\tiny 2}}}$
					} &	$ \xymatrix{ \ar@{|->}@<-3em>[r]^{\gamma_!} & }$ &	 \framebox{ 
					$\xymatrix@R=1em@C=.35em{ [\alpha] \ar@{-}[d]_{\overline \alpha} & [\beta] \ar@{-}[d]_{\overline \beta}&& [\delta] \ar@{-}[d]_{\overline \delta} \ar@{}[d]|->{\text{\tiny 2}} & [\omega] \ar@{-}[dl]|-{\overline \omega}\\ a \ar@{-}[r]_{\overline{\alpha}} & b \ar@<+.3em>@{-}[rr]^{\overline{\beta}} \ar@<-.3em>@{-}[rr]_{\overline \lambda} && c \ar@{-}[dl]|-{\overline \lambda} \ar@{-}@(dr,d)[]^{\overline\omega}_{\text{\tiny 2}} \\ && [\lambda] &&  }$}  				$ \xymatrix{ \ar@{|->}@<-3em>[r]^{\gamma^*} & }$ &	 \framebox{ 
					$\xymatrix@R=1em@C=.35em{ [\alpha] \ar@{-}[dr] \ar@{-}[d] & [\beta] \ar@{-}[dr] \ar@{-}[d] & [\delta]  \ar@{-}[d] & [\omega] \ar@{-}[dl]\\ a \ar@{-}[r] & b \ar@<+.3em>@{-}[r] \ar@<-.3em>@{-}[r] & c \ar@{-}[dl] \ar@{-}@(ur,r)[] \ar@{-}@(dr,d)[]_{\text{\tiny 2}} \\ & [\lambda] \ar@{-}[ul] &  }$}  
				\\[1ex] 
			\end{tabular}
		\end{center}
		
\begin{enumerate}[resume]
	\item[] Note that in $\gamma_!(G)$ there is an unfixed edge at $\overline \omega$ coming from the unfixed loop $\omega$. In general, for a symmetric $X'$-graph $G$ the unit $\eta_G\colon G\to \gamma^*\gamma_!(G)$ of the adjunction is a monomorphism. For a symmetric $X$-graph $H$ the counit $\varepsilon_H\colon \gamma_!\gamma^*(H)\to H$ is an epimorphism.

Next, consider the $\gamma$-coextension followed by the $\gamma$-restriction on the following symmetric $X'$-graph $G$.
\end{enumerate}
	
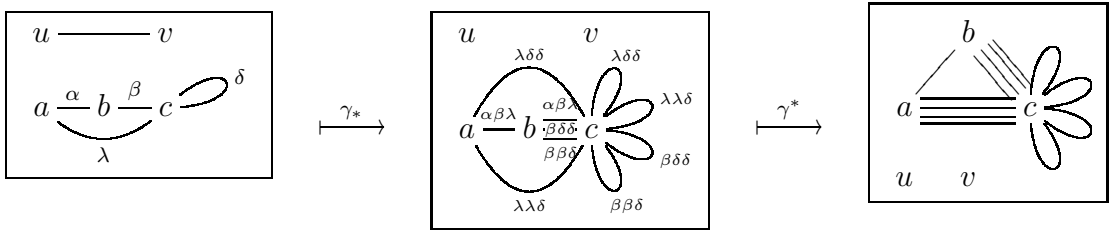
		\begin{center}
			\begin{tabular}{ c c  c  c c} 
				\framebox{ 
					$\xymatrix@R=1.3em@C=1em{ u \ar@{-}[rr] && v \\a \ar@{-}@/_1em/[rr]_{\lambda} \ar@{-}[r]^\alpha & b \ar@{-}[r]^\beta  & c \ar@{-}@(ur,r)[]^\delta }$
				} &	$ \xymatrix{ \ar@{|->}@<-3em>[r]^{\gamma_*} & }$ &	 \framebox{ 
				$\xymatrix@R=2em@C=1em{ u && v \\a \ar@{-}@/^2em/[rr]^{\text{\tiny $\lambda\delta\delta$}} \ar@{-}@/_2em/[rr]_{\text{\tiny $\lambda\lambda\delta$}}  \ar@{-}[r]^{\text{\tiny $\alpha\beta\lambda$}} & b \ar@<+.3em>@{-}[r]^{\text{\tiny $\alpha\beta\lambda$}} \ar@{-}[r]|-{\text{\tiny $\beta\delta\delta$}} \ar@<-.3em>@{-}[r]_{\text{\tiny $\beta\beta\delta$}} & c \ar@{-}@(ur,r)[]^{\text{\tiny $\lambda\lambda\delta$}}  \ar@{-}@(u,ur)[]^{\text{\tiny $\lambda\delta\delta$}} \ar@{-}@(r,dr)[]^{\text{\tiny $\beta\delta\delta$}} \ar@{-}@(dr,d)[]^{\text{\tiny $\beta\beta\delta$}}  }$}  &	$ \xymatrix{ \ar@{|->}@<-3em>[r]^{\gamma^*} & }$ &	 \framebox{ 
				$\xymatrix@R=1.3em@C=1em{ & b \ar@<+.5em>@{-}[dr] \ar@{-}[dr] \ar@{-}[dl] \ar@<-.3em>@{-}[dr] \ar@<+.3em>@{-}[dr] &\\ a \ar@{-}[rr] \ar@<-.5em>@{-}[rr] \ar@{-}@<+.3em>[rr] \ar@{-}@<-.3em>[rr]  && c \ar@{-}@(ur,r)[]^{\mbox{}}  \ar@{-}@(u,ur)[]_{\mbox{}} \ar@{-}@(r,dr)[] \ar@{-}@(dr,d)[]^{\mbox{}} \\ u & v }$}  
			\\[1ex] 
		\end{tabular}
	\end{center}

\begin{enumerate}[resume]	
	\item[] We have denoted the edges in the symmetric $X$-graph $\gamma_*(G)$ by the image of each edge in $\gamma^*(\underline A)$ (the complete symmetric $X'$-graph with three vertices) of a morphism $\widehat{\sG}_{X'}(\gamma^*(\underline A),G)$, e.g., the edge $\alpha\beta\lambda\in \gamma_*(G)$ corresponds to the symmetric $X'$-graph morphism $\alpha\beta\lambda\colon \gamma^*(\underline A)\to G$ which maps the three edges to $\alpha, \beta,$ and $\lambda$.
In general, the counit $\varepsilon'_G\colon \gamma^*\gamma_*(G)\to G$ is neither a monomorphism nor an epimorphism. The unit $\eta'_H\colon H\to \gamma_*\gamma^*(H)$ is readily verified to be an epimorphism.

		\item (The Associated Bouquet Functor) There are also essential geometric morphisms induced by non-injective set maps. For instance, let $\tau\colon X\to 1$ be the terminal set map and let $\DD G_X$ be any one of $\oG_{X}$, $\sG_{X}$, or $\hG_{X}$. Then there is an essential geometric morphism $\tau_!\dashv \tau^*\dashv \tau_*\colon \widehat{\DD G}_X\to \widehat{\DD G}_{(1,1)}$ where $\widehat{\DD G}_{(1,1)}$ is the category of bouquets (see Example \ref{E:XGraphs}\ref{E:Bouq}). The $\tau$-extension functor takes a $X$-graph $G$ to the bouquet $\tau_!(G)$ which has vertex set $G(V)/\sim$ where $\sim$ is the equivalence generated by $v\sim v'$ if there exists an arc $\alpha\in G(A)$ and $x, x'\in X$ such that $\alpha.x=v$ and $\alpha.x'=v'$ (i.e., the connected components of $G$) and arc set equal to $G(A)$ with right-action $\alpha.n$ equal to the connected component which contains $\alpha$. The $\tau$-restriction functor takes a bouquet $B$ to the graph $\tau^*(B)$ with vertex and arc set equal to $B(V)$ and $B(A)$ and such that for each $x\in X$ and $\beta\in B(A)$, $\beta.x=\beta.n$  and $\beta.m=\beta$ for each $m\in M$.  The $\tau$-coextension functor takes an $X$-graph $G$ to the bouquet $\tau_*(G)$ which has vertex set $G(V)$ and arc set equal to the loops in $G$ with the right-action given by $\alpha.n=\alpha.x$ for loop $\alpha$ and any $x\in X$. 
\end{enumerate}
\end{example}

Examples \ref{E:OSG}(\ref{E:FFC}, \ref{E:OSGraph}) along with composition by the adjunctions in Propositions \ref{P:HyperGSG} and \ref{P:EENerve} (take $X=2$) generalizes the existence of adjunctions between oriented graphs, undirected graphs, and hypergraphs.

In \cite{mR}, Proposition 13.2.1 states that for a functor $u\colon \C A\to \C B$ between small categories, the $u$-coextension functor $u_*$ admits a right adjoint $u^!$ if and only if for each $b\in \C B$, $u^*(\underline B)$ is a split subobject of a representable $\underline a$ for some $a\in \C A$. We may use this to verify that in each of the examples above the coextension functor does not admit a right adjoint. 

\chapter{Labeled $(X,M)$-Graphs}\label{S:Colored}

Let $\Gamma$ be a nonempty set of labels. A vertex labeling of a $(X,M)$-graph $G$ on a label set $\Gamma$ is given by a set map $\varphi\colon G(V)\to \Gamma $. Since the evaluation of a graph at the vertex object is left adjoint, we can define the functor $C^V\defeq \iota_*V_*\colon \mathbf{Set}\to \widehat{\DD G}_{(X,M)}$ as in Example \ref{E:FFC}(\ref{E:3}). This functor takes a set $\Gamma$ to the $(X,M)$-graph with set of vertices $C^V_{\Gamma }(V)\defeq \Gamma$ and arcs $C^V_\Gamma (A)\defeq \Gamma ^X$ with right-actions on an arc $f\colon X\to \Gamma$ given by $f.x=f(x)$ for $x\in X$ and $f.m=f\circ m$ where $m\colon X\to X$ is the right-action map of $m\in M$. In other words, the parametrized incidence map $\partial\colon C_S^V(A)=S^X\to S^X$ is the identity. Given a set map $g\colon \Gamma \to \Gamma'$,  $C^V_g\colon C^V_\Gamma  \to C^V_{\Gamma'}$ takes vertex $a\in \Gamma$ to $g(a)\in \Gamma'$ and an arc $f\colon X\to \Gamma$ to $g\circ f\colon X\to \Gamma'$.\footnote{The adjunction $(-)_V\dashv C^{V}_\Gamma$ is naturally isomorphic to the sheafification adjunction for $\neg\neg$-sheaves.} This leads us to the following definition.

\begin{definition}\label{D:VertexLabeled}
	Let $\Gamma$ be a set of labels. The category of vertex-labeled $(X,M)$-graphs on a label set $\Gamma$ is the slice category $\widehat{\DD G}_{(X,M)}\ls C^V_\Gamma $. 
\end{definition} 

By using a parametrized incidence which allows the incidence of arcs to take values in multisets, we are able to obtain greater control of what types of vertex-labellings are allowed for $k$-uniform hypergraphs than using the classical definition as the following examples show. 

\begin{example}\label{E:Labelings}\mbox{}
	\begin{enumerate}
		\item (Partitions of Unity $(X,M)$-Graphs) Let $X$ be a finite set and $[0,1]$ be the unit interval in the real numbers. We set
		\[ \textstyle
		P\defeq\setm{f\in [0,1]^X}{ \sum_{x\in X}f(x)=1}
		\]
		and define the sub-$(X,M)$-graph $G_P$ of $C_{[0,1]}^V$ with vertex set $G_P(V)=[0,1]$ and arc set $G_P(A)\defeq P$. Then the category $\widehat{\DD G}_{(X,M)}\ls G_P$ consists of $(X,M)$-graphs such that the sum of the weight of vertices (with multiplicity included) in an arc is equal to $1$. In other words, each edge represents a probability distribution of events represented by vertices. 
		
		More generally, we may replace $[0,1]$ by any abelian group $R$ and the condition that the sum equals $1$ by the condition that the application of the group operation equals $r$ for some element $r\in R$.
		\item \label{E:HC}(Non-monochromatic $(X,M)$-graphs) Let $\C M_\Gamma$ be the sub-$(X,M)$-graph of $C_{\Gamma}^V$ with the same vertex set $\Gamma$ but with arc set  
		\[
		 \C M_\Gamma(A)\defeq \setm{f\in \Gamma^X}{f\text{ is not constant}}.
		 \] 
		 In other words, $\C M_\Gamma$ has no loops. Then the category $\widehat{\DD G}_{(X,M)}\ls \widetilde{C}_\Gamma$ is the $(X,M)$-graph generalization of non-monochromatic uniform hypergraphs. Indeed, if $\varphi\colon G\to \C M_{\Gamma}$ is an object in $\widehat{\DD G}_{(X,M)}\ls \C M_\Gamma$, then an arc cannot send all its incident vertices to the same label since there are no loops in $\C M_\Gamma$.
		
		The chromatic number $\chi(G)$ of an $(X,M)$-graph $G$ we can define as the least cardinality of a set $\Gamma$ such that there is an $(X,M)$-graph morphism $G\to \C M_{\Gamma}$. If there is a morphism $f\colon G\to G'$, then clearly $\chi(G)\leq \chi(G')$ when both $\chi(G)$ and $\chi(G')$ exist.  Let $\C K$ be the category of cardinal numbers with the natural order as morphisms with a terminal object $T$ freely added. Then the chromatic number extends to a functor 
		\[
		\chi\colon \widehat{\DD G}_{(X,M)}\to \C K
		\]
		where we define $\chi(G)\defeq T$ whenever there is no morphism $G\to \C M_{\Gamma}$ for all sets $\Gamma$, e.g., whenever $G$ has a loop. A generalized Hedetniemi's conjecture can be formulated for $(X,M)$-graphs.
		\begin{quote}
			\textbf{Hedetniemi's Conjecture for $(X,M)$-Graphs}. The functor $\chi\colon \widehat{\DD G}_{(X,M)}\to \C K$ preserves finite products. 
		\end{quote}
		Since a product $G\x H$ has projection morphisms, 
		\[ 
			 \chi(G\x H)\leq \min\{\chi(G),\chi(H) \}.
		\] 
		In other words, $\chi$ is an oplax functor between cartesian monoidal categories. Thus it would be enough to prove the reverse inequality.
		
		\item \label{E:NRainbow}(Non-rainbow $(X,M)$-graphs) Let $\C R_\Gamma$ be the sub-$(X,M)$-graph of $C_{\Gamma}^V$ consisting of vertex set $\Gamma$ and arc set 
		\[ 
			\C R_{\Gamma}(A)\defeq \setm{f\in\Gamma^X}{f\text{ is not surjective}}.
			\] 
			Then the category $\DD G_{(X,M)}\ls \C R_{\Gamma}$ consists of labeled $(X,M)$-graphs whose arcs do not have incidence vertices of each label in $\Gamma$.
		\item (Non-monochromatic Non-rainbow $(X,M)$-graphs) We can combine the two conditions above by defining the sub-$(X,M)$-graph $\C S_\Gamma^V$ of $C_\Gamma^V$ to have vertex set $\Gamma$ and arc set $\C S_\Gamma^V(A)\defeq \C M_{\Gamma}^V(A)\cap \C R_{\Gamma}^V(A)$. Then the category $\DD G_{(X,M)}\ls \C S_{\Gamma}^V$ is the $(X,M)$-graph generalization of non-monochromatic, non-rainbow uniform graphs given in \cite{yC}. 
		\item (Fixing the Cardinalities of Incidence) Let $I$ be a subset of $\Gamma$ and $(n_i)_{i\in I}$ be a family of cardinal numbers. We define $\C N_\Gamma^V$ to have vertex set $\C N_{\Gamma}^V\defeq \Gamma$ and arc set $\C N_{\Gamma}^V(A)\defeq \setm{f\in \Gamma^X}{\forall i\in I,\ \#\inv f(i)=n_i}$.  Then the category $\widehat{\DD G}_{(X,M)}\ls \C N_{\Gamma}^V$ contains labeled $(X,M)$-graphs such that each arc has incidence vertices of label $i$ of cardinality $n_i$ (including multiplicities) for each $i\in I$.   
	\end{enumerate}
\end{example}

Next, we define arc-labeled $(X,M)$-graphs. Let $\Sigma$ be a nonempty set of labels. An arc labeling of a $(X,M)$-graph on a label set $\Sigma$ is given by a set map $\psi\colon G(A)\to \Sigma$ such that $\psi(\alpha.m)=\psi(\alpha)$ for each $m\in M$ and $\alpha\in G(A)$. While the case for labeling vertices was a matter of using the adjunction $(-)_V\dashv \iota_*V_*$, the case for labeling arcs is not the same since we will want the labels to respect $M$-actions. So we construct a right adjoint to the functor $\widetilde{(-)}_A\colon \widehat{\DD G}_{(X,M)}\to \mathbf{Set}$ which sends an $(X,M)$-graph $G$ to the set $G(A)/\sim$ where $\sim$ is an equivalence relation generated by $\alpha\sim \beta$ if $\exists m\in M,\ \alpha.m=\beta$.  Define the functor $C^A_{(-)}\colon \mathbf{Set} \to \widehat{\DD G}_{(X,M)}$ on a set $\Sigma$ by $C^A_\Sigma (V)\defeq 1$ (the terminal set) and $C^V_\Sigma (A)\defeq \Sigma$ with right-actions fixed for each $m\in M$, i.e., $\alpha.m=\alpha$ for each $\alpha\in \Sigma$ and $m\in M$. For a set map $f\colon \Sigma\to \Sigma'$, the $(X,M)$-graph morphism $C_f^{A}\colon C_\Sigma^A\to C_{\Sigma'}^A$ is the identity on the singleton vertex and takes $\sigma\in \Sigma$ to $f(\sigma)\in \Sigma'$ on arcs. It is straightforward to show that $\widetilde{(-)_A}\dashv C^{A}_{(-)}\colon \mathbf{Set} \to \widehat{\DD G}_{(X,M)}$.

\begin{definition} \label{D:ArcLabel}
	Let $\Sigma$ be a set of labels. The category of arc-labeled $(X,M)$-graphs on a label set $\Sigma$ is the slice category $\widehat{\DD G}_{(X,M)}\ls C^A_\Sigma $. 
\end{definition}
\newpage
\begin{example}\label{E:Labeled}\mbox{}
\begin{enumerate}
	\item (Functional $X$-Gaphs)  Consider the category of $(X,1)$-directed symmetric hypergraphs and let \\$\Sigma\defeq\DD R\x \DD R^{\DD R^X}\x \DD R^{\DD R^{2}}$, i.e., an element in $\Sigma$ consists of a real number $w$, and set maps $\psi\colon \DD R^{X}\to \DD R$ and $f\colon \DD R^2\to \DD R$. Then an arc-labeled $(X,1)$-directed symmetric hypergraph is the $(X,M)$-graph analogue to a functional hypergraph given in \cite{gA} (Definition 1).
\ignore{		\item \label{E:GameTheory} A game $G$ consists of a finite set of players $X$, a finite set of strategies $S$, and a payoff function $p\colon S^X\to \DD R^X$ which associates to every strategy profile $s\colon X\to S$ a payoff $p_s\colon X\to \DD R$ such that $p_s(x)\in \DD R$ is the payoff for the individual player $x$ given a strategy profile $s$. 
		
		We define a tournament of games as an $S$-vertex labeled oriented $X$-graph $\varphi\colon T\to C_S^V$. The arcs of $T$ should be interpreted as a set of games and the vertices as available strategies. Then for an arc $\alpha$, the strategy employed by player $x$ is the label of the $x$-incidence $\varphi_V(\alpha.x)$. The payoff function can be given as an oriented $X$-graph morphism $p\colon C^V_S\to C^A_{\DD R^{X}}$ and the payoff for player $x\in X$ in game $\alpha\in T(A)$ is given by the evaluation $p_A\circ \varphi_A(\alpha)\colon X\to \DD R$ at $x$. If we want to restrict to zero-sum games, we define the subobject $G_0$ of $C^V_S$ to have the same vertex set $G_0(V)=S$ and arc set $G_0(A)\defeq \setm{f\colon X\to S}{\sum_{x\in X}p_A(f)(x)=0}$ and define a tournament of zero-sum games as an object in $\widehat{\oG}_X\ls G_0$. More generally, given a set of payoff values $P\subseteq \DD R$, we set $G_P(V)=S$ and $G_P(A)\defeq \setm{f\colon X\to \DD R}{\sum_{x\in X}p_A(f)(x)\in P}$ and define a tournament of $P$-sum games as an object in $\widehat{\oG}_X\ls G_P$.
}
	\item \label{E:Oriented}  (Oriented Hypergraph) Consider the symmetric $X$-graph $\C O$ consisting of a single vertex and with arc set $\C O(A)\defeq 2^X$ with right-action given by $f.m=f\circ m$ for each $f\in \C O(A)$, $m\in \sX$. Then the slice category $\widehat{\sG}_X\ls \C O$ is the $X$-graph analogue of oriented hypergraphs given in $\cite{lR}$ where $2$ can be interpreted as the two-element set $\{\text{in},\text{out}\}$. Notice that in this case we did not identify $m$-partners in the arc labeling as we did in Definition \ref{D:ArcLabel} above. 
\end{enumerate}
\end{example}

We can simultaneously label vertices and arcs since a vertex labeling $\varphi \colon G\to  C^V_\Gamma$ and an arc labeling $\psi\colon G\to C^V_\Sigma$ induces a unique morphism $G\to C^V_\Gamma\x C^A_\Sigma$. 

\begin{definition}
	Let $\Gamma$ and $\Sigma$ be sets of labels and set $C_{(\Gamma ,\Sigma)}\defeq C^V_\Gamma \x C^A_{\Sigma}$. The category of labeled $(X,M)$-graphs on label sets $\Gamma$ and $\Sigma$ is the slice category $\widehat{\DD G}_{(X,M)}\ls C_{(\Gamma ,\Sigma)}$. 
\end{definition}

Recall that categories of presheaves are closed under taking slices, i.e., for a small category $\C C$ we have $\widehat{\C C}\ls F$ is equivalent to $[\int F^{op}, \mathbf{Set}]$ where $\int F$ is the category of elements of the presheaf $F$ (\cite{eR}, Section 2.4). Thus the following result holds.

\begin{corollary}
	Let $\Gamma$ and $\Sigma$ be sets. The category of labeled $(X,M)$-graphs on label sets $\Gamma$ and $\Sigma$ is a category of presheaves.
\end{corollary}

This means categories of labeled $(X,M)$-graphs are also presheaf toposes. Moreover, a relabeling of the sets $\Gamma$ and $\Sigma$ given by set maps $f\colon \Gamma \to \Gamma'$ and $g\colon \Sigma \to \Sigma'$ induce essential geometric morphisms (\cite{pJ}, Corollary 1.5.3) \[\Sigma_{(f,g)}\dashv (f,g)^*\dashv \Pi_{(f,g)}\colon \colon \widehat{\DD G}_{(X,M)}\ls C_{(\Gamma,\Sigma)}\to \widehat{\DD G}_{(X,M)}\ls C_{(\Gamma^*,\Sigma^*)}.\]

\begin{example} \label{E:Labeled2}\mbox{}
If $\Gamma =1$ is a singleton set, then $C^V_\Gamma$ and $C^A_\Gamma$ are the terminal $(X,M)$-graph with one vertex and one arc. Therefore, labeled $(X,M)$-graphs generalize both vertex-labeled (e.g. take $\Sigma=1$) and arc-labeled (e.g. take $\Gamma =1$) $(X,M)$-graphs as well as $(X,M)$-graphs (take $\Gamma =\Sigma=1$). 
\end{example}

\chapter{Hybrid and Mixed Structures}\label{S:HyM}

Next, we describe hybrid and mixed structures of $(X,M)$-graphs. 

\begin{definition}
	Let $\C F=(\DD G_{(X_j,M_j)})_{j\in J}$ be a family of theories for $(X,M)$-graphs. The theory for mixed $\C F$-graphs is the category $\DD G_{\C F}$ with one vertex object $V$ and an arc object $A_j$ for each $j\in J$ and morphism set the union of morphisms in each $\DD G_{(X_j,M_j)}$.
	Composition is defined as in each of the theories in $\C F$.   The category of mixed $\C F$-graphs is defined to be the category of presheaves $\widehat{\DD G}_{\C F}$.
\end{definition}

We depict a theory for mixed $\C F$-graphs as $\left(\xymatrix{ V \ar[r]|-{X_j} & A_j \ar@(ur,dr)^{M_j}  }\right)_J$.

\begin{example}\mbox{}
	\begin{enumerate}
		\item (Comparison with Classical Definition) When the family $\C F$ contains two theories, a mixed $\C F$-graph is the $(X,M)$-graph analogue of mixed hypergraphs given in \cite{vV} (Definition 10.3.1).
		\item (Oriented and Symmetric Structures on Same Set of Vertices) Consider the family $\C F\defeq (\DD G_i)_{i=0,1}$ where $\DD G_0=\sG_2$ and $\DD G_1=\oG_2$. Then a mixed $\C F$-graphs $G$ consists of a set of vertices $G(V)$ a set of $0$-arcs $G(A_0)$ and a set of $1$-arcs $G(A_1)$.  For example, the $\C F$-graph $G$ can be depicted as follows.
		\end{enumerate}
		\begin{multicols}{2}
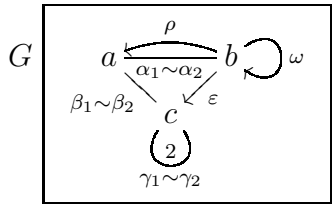
			\begin{center}$G\ $\framebox{ 
					$\xymatrix@C=1em@R=.75em{ a \ar@{-}[dr]_{\beta_1\sim \beta_2} \ar@{-}[rr]_{\alpha_1\sim \alpha_2}  && b \ar@(ur,dr)[]^{\omega} \ar[dl]^{\varepsilon} \ar@/_/[ll]_{\rho}\\ & c \ar@{-}@(dl,dr)[]_{\gamma_1\sim \gamma_2}^{2} &}$
				}
			\end{center}
			\begin{align*}
			& G(V)= \{a,\ b,\ c\}\\
			& G(A_0)=\{\alpha_1\sim \alpha_2, \beta_1\sim \beta_2, \gamma_1\sim \gamma_2 \},\\
			& G(A_1)=\{\rho, \omega, \varepsilon \}\\
			\end{align*}
		\end{multicols}
\begin{enumerate}[resume]
		\item \label{E:RegularGraphs} (Regular $(X,M)$-graphs) Mixed bouquets can be used to characterize/define regular $(X,M)$-graphs. Let $S$ be a non-empty subset of $X$ and consider the family of theories $(\mathbb{G}_{(1,1)})_{x\in S}$. There is an obvious functor $B\colon \mathbb{G}_{\mathcal{F}}\to \mathbb{G}_{(X,M)}$ which takes $V$ to $V$ and $A_x\to A$. This induces an essential geometric morphism $B_!\dashv B^*\dashv B_*\colon \widehat{\mathbb{G}}_{\mathcal F}\to \widehat{\mathbb{G}}_{(X,M)}$. Then we say an $(X,M)$-graph $G$ is $S$-regular of order $k$ provided $B^*(G)\iso \sqcup_{v\in G}\neg\neg v'$ for any vertex $v'\in B^*(G)$ where $k=\#\neg\neg v(A)$. In the case of simple loopless symmetric $X$-graphs when $X=\{s,t\}$ and $S=\{s\}$ (or $\{t\}$ by symmetry) we regain the classical definition. Notice that this approach counts the multiplicities in the incidence of an arc. For example, the 2-loop symmetric $X$-graph where $X=\{s,t\}$ would be considered $s$-regular of order 2. 
	\end{enumerate}
\end{example}

The following result shows that arc-labeled $(X,M)$-graphs are the same as certain mixed $\C F$-graphs.

\begin{proposition}\label{P:LabelsHybrid}
	Let $M$ be a monoid, $X$ an $M$-set, and $\C F\defeq (\DD G_{(X,M)})_{j\in J}$ for some set $J$.  There is an adjoint equivalence $M\dashv K\colon \widehat{\DD G}_{\C F}\to \widehat{\DD G}_{(X,M)}\ls C_J^A$.
\end{proposition}
\begin{proof}
	It is enough to show that there is an equivalence between the category of elements $\int C^A_J$ and the mixed $\C F$-graph theory $\DD G_{\C F}$. Let $\mu\colon \DD G_{\C F}\to \int C^A_J$ be the functor which takes $V$ to $(1,V)$, each $A_j$ to $(j,A)$, the morphism $x\colon V\to A$ to $x\colon (1,V)\to (a_j,A)$ and $m\colon A_j\to A_j$ to $m\colon (a_j,A)\to (a_j,A)$. It is clear that this is an isomorphism of categories and hence induces an isomorphism of categories $M\dashv K\colon \widehat{\DD G}_{\C F}\to \widehat{\DD G}_{(X,M)}\ls C_J^A$.
 
\end{proof}

\begin{example}\mbox{} \label{E:CG}
	\begin{enumerate}
		\item (Associated $X$-Graph Functor) In Example \ref{E:OSG}(\ref{E:Assoc}), we were able to given the associated symmetric $X'$-graph to a symmetric $X$-graph via the restriction functor induced by an injective map $j\colon X'\hookrightarrow X$. For the oriented case, this is not possible since each injection will give a different restriction functor. However, a construction is possible by using the above adjunction. Let $X'\subseteq X$ and set $J$ be the set of injective maps from $X'$ to $X$. Define $\C F_{X}\defeq (\oG_X)_{j\in J}$. There is a functor  $\kappa\colon \DD G_{\C F_{X}}\to \oG_X$ which takes $V$ to $V$ and each $A_j$ to $A$. Next, define $\C F_{X'}\defeq (\oG_{X'})_{j\in J}$.  There is also a functor $\lambda\colon \DD G_{\C F_{X'}}\to \DD G_{\C F_{X}}$ which takes $V$ to $V$, $A_j$ to $A_j$, and a morphism $x'\colon V\to A_j$ to $j(x')\colon V\to A_j$. Then the functor $\mu\defeq \kappa\circ \lambda$ induces an essential geometric morphism $\mu_! \dashv \mu^* \dashv \mu_*\colon \widehat{\DD G}_{\C F_{X'}} \to \widehat{\oG}_X.$
					
		The restriction functor $\mu^*$ takes an oriented $X$-graph $G$ to the $\C F$-family of $X'$-graphs with same underlying vertex set as $G$ and $\mu^*G(A_j)=G(A)$ for each $j\in J$ with the $x'$-incidence of a $j$-arc $\alpha\in \mu^*G(A_j)$ equal to $\alpha.j(x')$ in $G$. Since $\widehat{\oG}_X$ has pullbacks, the terminal morphism $C^A_J\to 1$ induces an adjunction 
		$\Sigma_{C^A_J}\dashv {C^{A}_J}^*\colon  \widehat{\oG}_{X'}\to \widehat{\oG}_{X'}\ls C^A_{J}$ where $\Sigma_{C^A_J}$ acts by postcomposition and thus has the effect of forgetting the labeling.\footnote{We are using the $(X,M)$-graph $C^A_J$ introduced in Chapter \ref{S:Colored}.}  Then there is an adjunction  
		\[ \text{Gr}_{X'}\dashv \mu_*M{C^A_J}^*\colon \widehat{\oG}_{X'}\to \widehat{\oG}_{X}.
		\] 
		where we call $\text{Gr}_{X'}\defeq \Sigma_{C^A_J}K\mu^*\colon \widehat{\oG}_{X}\to \widehat{\oG}_{X'}$ the associated oriented $X'$-graph functor.
					
		For instance, let $X=\{a,b,c\}$ and $X'=\{s,t\}$. Then $J$ is the set of three inclusions $ab, ac, bc\colon X'\hookrightarrow X$ (e.g., $ab$ is the map which sends $s$ to $a$ and $t$ to $b$). Let $G$ be the $X$-graph with vertex set $G(V)\defeq \{u,v,w,x,y,z \}$ and arc set $G(A)\defeq \{\alpha,\beta,\gamma \}$ with incidence $\alpha.a=u$, $\alpha.b=\alpha.c=v$, $\beta.a=w$, $\beta.b=x$, $\beta.c=y$ and $\gamma.a=\gamma.b=\gamma.c=z$. The following composition is the associated oriented $X'$-graph $\text{Gr}_{X'}(G)$.
	\end{enumerate}
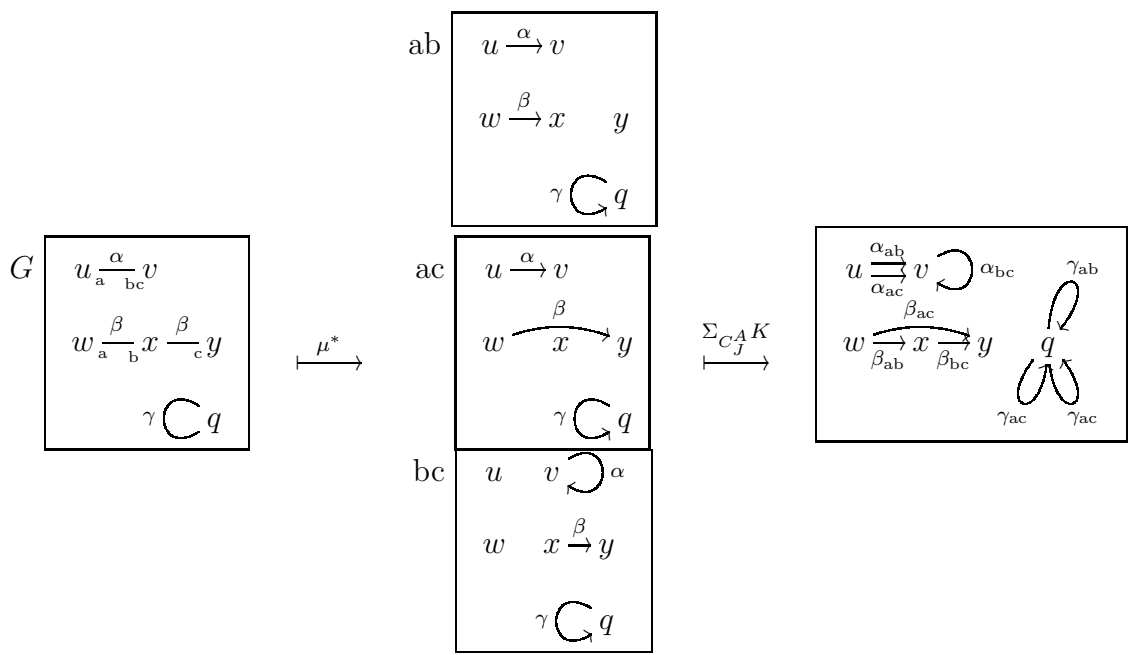
					\begin{center}
						\begin{tabular}{ c c  c c c } 
							&& ab \framebox{ $\xymatrix@R=1.3em@C=1em{  u \ar[r]^{\alpha}   & v  & \\ w \ar[r]^{\beta}  & x  &  y \\
									& & q \ar@(ul,dl)[]_{\gamma}}$ } && \\		
							$G\ $\framebox{ 
								$\xymatrix@R=1.3em@C=1em{  u \ar@{-}[r]^{\alpha} \ar@{}[r]_<{\text{\tiny a}} \ar@{}[r]_>{\text{\tiny bc}} & v & \\ w \ar@{-}[r]_<{\text{\tiny a}}^{\beta} \ar@{}[r]_>{\text{\tiny b}} & x \ar@{-}[r]_>{\text{\tiny c}}^{\beta} &  y \\
									& & q \ar@{-}@(ul,dl)[]_{\gamma} }$
							} &	$ \xymatrix{ \ar@{|->}@<-3em>[r]^{\mu^*} & }$ &	ac \framebox{ 
							$\xymatrix@R=1.3em@C=1em{  u \ar[r]^{\alpha} & v & \\ w \ar@/^/[rr]^{\beta} & x &  y \\
								& & q \ar@(ul,dl)[]_{\gamma} }$}  & $\xymatrix{ \ar@<-3em>@{|->}[r]^{\Sigma_{C^A_{J}}K} & }$ 
						& \framebox{ $\xymatrix@R=1.3em@C=1em{   u \ar@<+.2em>[r]^{\alpha_{\text{\tiny ab}}} \ar@<-.2em>[r]_{\alpha_{\text{\tiny ac}}}  & v \ar@(ur,dr)[]^{\alpha_{\text{\tiny bc}}} & \\  w \ar@/^/[rr]^{\beta_{\text{\tiny ac}}} \ar[r]_{\beta_\text{\tiny ab}}  & x \ar[r]_{\beta_{\text{\tiny bc}}} &  y & q \ar@(d,dr)[]_{\gamma_{\text{\tiny ac}}} \ar@(dl,d)[]_{\gamma_{\text{\tiny ac}}} \ar@(u,ur)[]^{\gamma_{\text{\tiny ab}}} }$ } 
						\\			
						&&	bc \framebox{ $\xymatrix@R=1.3em@C=.7em{   u & v \ar@(ur,dr)[]^{\alpha} & \\ w  & x \ar[r]^{\beta} &  y \\
								& & q \ar@(ul,dl)[]_{\gamma} }$ } &&
						\\[1ex] 
					\end{tabular}
				\end{center}
\begin{enumerate}[resume]
		\item[]		Notice that each arc in an oriented $X$-graphs get taken to three arcs in its associated oriented $X'$-graph. 
			\item \label{E:CayleyGraph}(Cayley Graph) Let $G$ be a monoid and $S\hookrightarrow G$ an injective set map. Define the family of theories $\C F\defeq (\oG_2)_{g\in S}$ ($2\defeq \{s,t\}$). There is a functor $\gamma\colon \DD G_{\C F}\to \DD G_{(|G|,G)}$ which takes $V$ to $V$, $A_g$ to $A$, $s\colon V\to A_g$ to $e\colon V\to A$ (the unit of the monoid $G$) and $t\colon V\to A_g$ to $g\colon V\to A$ for each $g\in S$. This induces the essential geometric morphism $\gamma_!\dashv \gamma^*\dashv \gamma_*\colon \widehat{\DD G}_{\C F}\to \widehat{\DD G}_{(|G|,G)}$. By composition with the equivalence above, the $\widehat{\oG}_2\ls C^A_S$-object $K\gamma^*(\underline A)$ is the Cayley graph $\Cay(G,S)$ determined by the set $S$. The construction shows that $\Cay(G,S)$ is an $S$ arc-labeled oriented graph. Notice that $S$ is a generating set iff the underlying $\oG_2$-graph of $\Cay(G,S)$ is connected. When $G$ is a group, this coincides with the classical definition. 
			\item \label{E:SchreierGraph} (Schreier Graph) We can generalize the construction of Cayley Graphs by replacing $\mathbb G_{(|G|,G)}$ above with $\mathbb G_{(|G/H|,G)}$ for a submonoid $H$ of $G$. The functor $\gamma\colon \mathbb{G}_{\mathcal F}\to \mathbb G_{(|G/H|,G)}$ takes $V$ to $V$, $A_g$ to $A$, $s\colon V\to A$ to $He\colon V\to A$ and $t\colon V\to A_g$ to $Hg\colon V\to A$.  Then the arc-labeled oriented graph $K\gamma^*(\underline A)$ is the Schreier graph Sch($G,H,S$). When $H=1$, we have Sch($G,1,S$)$=\Cay(G,S)$. Another way to see this is to take the morphism of group actions $p\colon (|G|,G)\to (|G/H|,G)$ which induces the morphism between theories $p\colon \mathbb{G}_{(|G|,G)}\to \mathbb{G}_{(|G/H|,G)}$ and thus gives us an essential geometric morphism between the toposes $p_!\dashv p^*\dashv p_*\colon \widehat{\mathbb G}_{(|G|,G)}\to \widehat{\mathbb G}_{(|G/H|,G)}$. Then by essential uniqueness of adjoints, $K\gamma^*p^*(\underline A)$ is the Schreier $(X,M)$-graph Sch$(G,H,S)$. When $G$ is a group, this definition coincides with the classical definition.
			\end{enumerate}
\end{example}

\begin{definition}\label{D:Hybrid}
	Let $M$ be a monoid and $\C F\defeq (X_j)_{j\in J}$ a family of right $M$-sets. The theory for a hybrid $\C F$-graphs is the category $\DD G_{(\C F, M)}$ with one arc object $A$ and a vertex object $V_j$ for each $j\in J$ and a morphism set equal to the union of $\bigcup_JX_j$ and $M$ with inherited composition. 
	The category of hybrid $\C F$-graphs is defined to be the category of presheaves $\widehat{\DD G}_{(\C F,M)}$.
\end{definition}

We depict a theory for hybrid $\C F$-graphs as $\left(\xymatrix{ V_j \ar[r]|-{X_j} & A \ar@(ur,dr)^{M}  }\right)_J$.

\begin{example}\label{E:Hybrid}\mbox{}
	\begin{enumerate}
		\item (Collective Hybrid $(X,M)$-Graphs) Let $(\DD G_{(X_j,M_j)})_{j\in J}$ be a family of $(X,M)$-graph theories. Let $M\defeq \prod_{j\in J}M_j$ the the product of monoids. Then by restriction of scalars along the projections, $\C F\defeq (X_j)_{j\in J}$ is a family of right $M$-sets and $\DD G_{(\C F,M)}$ is the theory for hybrid $\C F$-graphs. Thus any collection of $(X,M)$-theories admit a hybrid structure. 
		\item \label{E:Bipart} (The Category of $k$-partite $X$-Graphs) Let $M$ be the trivial monoid and $\C F\defeq (1)_{x\in X}$. Then $\widehat{\DD G}_{(\C F,1)}$ is the category of $k$-partite $X$-graphs where $k$ is the cardinality of $X$. In particular when $X=\{s,t\}$ then $\DD G_{(\C F,1)}$ is the category$\xymatrix{V_1 \ar[r]^{s}& A & V_2 \ar[l]_{t}}$ and thus $\widehat{\DD G}_{(\C F,1)}$ is the category of bipartite graphs, $\C B$. By Proposition \ref{P:FoliatedGraphs} above, we have $\C B$ is equivalent to both $\widehat{\oG}_2\ls \underline A$ and $\widehat{\sG}_2\ls \underline A$. 
		
	\item \label{E:BipartH} (Hypergraphs vs. Bipartite Graphs)	A hypergraph $H=(H(V), H(E), \varphi)$ (see Chapter \ref{S:PGraphs}) has a bipartite graph representation $i(H)$ where the set of $V_1$-vertices is $G(V)$, the set of $V_2$-vertices is the set $G(E)$, the set of arcs is given by $\setm{(v,e)\in H(V)\x H(E)}{v\in \varphi(e)}$ with right-actions $(v,e).s=v$ and $(v,e).t=e$. \footnote{This is equivalent to the definition of Hypergraphs given in \cite{hE}, Fact 4.17 } Given a morphism $f\colon H\to H'$ of hypergraphs, there is a bipartite graph morphism $i(f)\colon i(H)\to i(H')$ such that $i(f)_{V_1}\defeq f_V$, $i(f)_{V_2}\defeq f_E$ and  $i(f)_A(v,e)\defeq (f_{V}(v),f_E(e))$ for each $(v,e)\in H(A)$. This assignment defines a faithful (non-full) functor $i\colon \C H\to \C B$ (cf., \cite{wD}, Proposition 3.2).  
		
		For a bipartite graph $G$ there is a natural way to associate it to a hypergraph. We define $r(G)\defeq (G(V_1),G(V_2),\varphi_G)$ where $\varphi_G\colon G(V_2)\to \C P(G(V_1))$ takes a $V_2$-vertex $e$ to the set 
		\[\setm{v\in G(V_1)}{\exists a\in G(A),\ \sigma(a)=v\text{ and } \tau(a)=e}.\]
		
		However, the definition of $r$ does not extend to functor. Consider the hypergraphs $L$ and $P$ where $L$ has one $V_1$-vertex $v$, one $V_2$-vertex $e$ and one arc $a$ connecting them, and $P$ has two $V_1$-vertices $v_1$ and $v_2$, one $V_2$-vertex $e$ and two arcs $a_s$ and $a_t$ such that source incidents of $a_s$ and $a_t$ are $v_1$ and $v_2$ respectively. Let $g\colon L\to P$ be the bipartite graphs morphism with takes the arc $a$ to $a_s$, i.e., it is given by the following commuting diagram
		\[
		\xymatrix{ L \ar[d]_{g}  & \{v\} \ar[d]_{\named{v_1}}  & \{a\} \ar[r] \ar[l] \ar[d]_{\named{a_s}} & \{e\}  \ar[d]^{\Id_e} \\ P & \{v_1,v_2\} & \{a_s,a_t\} \ar[r] \ar[l] & \{e\} }
		\]
		The associated hypergraph of $L$ is the loop hypergraph $r(L)$ with one vertex $v$, one edge $e$ and the incidence $\named{v}\colon \{e\}\to \C P(\{v\})$ which sends $e$ to $\{v\}$. The associated hypergraph of $P$ is the edge hypergraph with two vertices $v_1$ and $v_2$, one edge $e$ and incidence $\named{v_1,v_2}\colon \{e\} \to \C P(\{v_1,v_2\})$ which sends $e$ to the set $\{v_1,v_2\}$.
		
		Consider the diagram
		\begin{align}\label{D:Ob1}
		\xymatrix{ \{e\} \ar@{}[dr]|-{\text{\tiny{not com.}}} \ar[r]^{\Id_e} \ar[d]_{\named v} & \{e\} \ar[d]^{\named{v_1,v_2}} \\ \C P(\{v\}) \ar[r]^-{\C P(\named{v_1})} & \C P(\{v_1,v_2\})}
		\end{align}
		Since $\{v_1\}\subsetneq \{v_1,v_2\}$ is a strict inclusion, the diagram does not commute. Thus, the assignment of $r$ does not extend to the morphism $g$. 

		The problem is that morphisms of hypergraphs preserve incidence of edges on the nose and morphisms of bipartite graphs allow subset inclusions. We will address this issue again in Chapter \ref{S:Simple}.
		
		\item (Hyperedges Between Arcs) Let $M$ be any of the submonoids of endomaps described for $X$-graphs and $(S,T)$-directed $X$-graphs, let $\DD G_X$ be the corresponding theory, and let $\C F$ be the family of $M$-sets consisting of $X$ and the singleton set $1$. Then the  hybrid $\C F$-graph can be interpreted as an $X$-graph $G$ and a collection of hyperedges between arcs in $G$ (including empty hyperedges) in the same way that a hypergraph can be interpreted as a bipartite graph. Note that the hyperedges must respect the $M$-actions, e.g., if a hyperedge $e$ is incident to an arc $\alpha$, then it is also incident to $\alpha.m$ for each $m\in M$. 
		\item \label{E:DualBi} (Hypergraph Dualization Functor) Let $\C F=(X)_{j\in J}$ be a constant family for some set $X$ and let $\sigma\colon J\to J$ be an automap. Then there is an isomorphism of categories $\overline \sigma\colon \DD G_{(\C F,M)}\to \DD G_{(\C F,M)}$, which takes $V_j$ to $V_{\sigma(j)}$ and induces an equivalence $D_\sigma\colon \widehat{\DD G}_{(\C F, M)}\to \widehat{\DD G}_{(\C F, M)}$. When $\sigma$ is an involution, we call $D_{\sigma}$ a dualization of structure  (cf, Example \ref{E:OSG}(\ref{StructDual})). In the case of bipartite graphs, the nontrivial automap $\sigma\colon 2\to 2$ induces the dualization equivalence $D_\sigma\colon \C B\to \C B$ which restricts along the faithful functor $i\colon \C H \to \C B$ given above in Example \ref{E:Hybrid}(\ref{E:Bipart})
		\[
			\xymatrix{ \C H \ar[d]_{D_\sigma} \ar[r]^i & \C B \ar[d]^{D_\sigma} \\ \C H \ar[r]^i & \C B. }	
		\]
		The functor $D_\sigma\colon \C H\to \C H$ is the well-known dualization functor for hypergraphs. 
		\item \label{E:Int} (Intersection Graphs) Let $\C F'$ be the family of $\Aut(2)$-sets consisting of the singleton $1$ and two-element set $2$. Then the theory for hybrid $\C F'$-graphs $\DD G_{(\C F',\Aut(2))}$ is the category generated by the graph 
		\[ 
			\xymatrix{V_1 \ar[r]^s & A \ar@(ul,ur)[]^{i} & V_2 \ar@<-.3em>[l]_{\tau} \ar@<+.3em>[l]^\sigma }
		\]
		where $i\circ \tau=\sigma$ for the non-trivial automap $i\colon 2\to 2$. Let $\C F=(X)_{i\in \{1,2\}}$ be the family  given above in Example \ref{E:Hybrid}(\ref{E:DualBi}) above. Then there is a functor from $\lambda\colon \DD G_{(\C F', \Aut(2))}\to \DD G_{(\C F,1)}$ which preserves vertex objects $V_1$ an $V_2$ respectively as well as the arc object $A$. This induces an essential geometric morphism $\lambda_!\dashv \lambda^*\dashv \lambda_*\colon \widehat{\DD G}_{(\C F',\Aut(2))}\to \DD G_{(\C F, 1)}$. There is also an embedding $\theta\colon \sG_2\to \DD G_{(\C F',\Aut(2))}$ which takes $V$ to $V_2$ preserves the arc object which induces another essential geometric morphism $\theta_!\dashv \theta^*\dashv \theta_*\colon \widehat{\sG}_{2}\to \widehat{\DD G}_{(\C F',\Aut(2))}$. Let $i\colon \C H\to \C B=\widehat{\DD G}_{(\C F, 1)}$ be the faithful functor described Example \ref{E:Hybrid}(\ref{E:Bipart}) above. Then the intersection graph construction (also called the line graph) of a hypergraph is given by the composition $\Int_2\defeq\lambda^*\theta^* i\colon \C H\to \widehat{\sG}_2$. Note that this differs slightly from the intersection graph construction given in  \cite{wD} (Proposition 3.4) since the construction there gave a simple loopless graph while our construction contains loops and multiple edges between vertices.
		
		More generally, for any set $X$ we let $\C F'$ be the family of $\Aut(X)$-sets consisting of the singleton and the set $X$. Then the construction above generalizes to a functor $\Int_X\colon \C H\to \widehat{\sG}_X$ which takes a hypergraph $H$ to the symmetric $X$-graph with vertex set  $\Int_X(H)(V)=H(E)$ and arc set 
		\[\Int_X(H)(A)=\setm{f\in H(E)^X}{\exists v\in H(V), \forall x\in X, v\in \varphi(f(x))}.\] 
		Taking the maximal subobject of $\Int_X(H)$ which is a $k$-uniform hypergraph where $k$ is the cardinality of $X$ and simplifying it (see Section \ref{S:Simple}) gives us the usual construction. However, note that restricting to this maximal subobject is not functorial. 
	\end{enumerate}
\end{example}

\begin{proposition}\label{P:FoliatedGraphs}
	Let $G$ be a group, $X$ a right $G$-set, and $\C F\defeq (1)_{x\in X}$ be the family consisting of the singleton set considered as the right $M$-sets for the trivial monoid $M=1$. Then $\widehat{\DD G}_{(X,G)}\ls \underline A$ is equivalent to the category of $k$-partite $X$-graphs $\widehat{\DD G}_{(\C F,1)}$ where $k$ is the cardinality of $X$.
\end{proposition}
\begin{proof}
	By \cite{eR} (Example 2.4.6) $\widehat{\DD G}_{(X,G)}\ls \underline A$ is equivalent the category of presheaves on the category of elements $\int \underline A$. There is an equivalence $L\colon \DD G_{(\C F, 1)} \to \int \underline A$ given by taking the arc object $A$ to $(a_1,A)$ and each $V_x$ to $(v_x,V)$. On morphisms, it takes $V_x\to A$ to $x\colon (v_x,V)\to (a_1,A)$. Thus $L$ is full and faithful. It is also essentially surjective since for each object $(a_\sigma,A)$ for $\sigma\in G$, there is an isomorphism $\underline \sigma\colon (a_\sigma,A)\to (a_1,A)$. Then since $\int \underline A$ is equivalent to $\DD G_{(\C F,1)}$, there is an equivalence between $\widehat{\DD G}_{(X,G)}\ls \underline A$ and $\widehat{\DD G}_{(\C F,1)}$.
	
\end{proof}

\begin{corollary}
	For each set $X$, the slice category $\widehat{\oG}_X\ls \underline A$ is equivalent to $\widehat{\sG}_X\ls \underline A$ with equivalence given by taking $\varphi\colon G\to \underline A$ to the symmetric extension $p_!(\varphi)\colon p_!(G) \to p_!(\underline A)\iso \underline A$ (see Example \ref{E:OSG}(\ref{E:OSGraph})).
\end{corollary}

Let $G$ be a group and $H$ an $(X,G)$-graph. The product projection $\pi_{\underline A}\colon \underline A\x H\to \underline A$ realizes $\underline A\x H$ as a $X$-partite $(X,G)$-graph. The other projection $\pi_{H}\colon \underline A\x H\to H$ in the category of $(X,G)$-graphs is the canonical $k$-partite $X$-cover and generalizes the canonical bipartite double cover discussed in \cite{dW} and \cite{dA}.

\chapter{The Topos Structure and Properties}\label{S:ToposStructure}

The category of (reflexive) $(X,M)$-graphs is a topos and hence is locally cartesian closed and has a subobject classifier. The limits and colimits are computed pointwise in the category of sets. We will construct exponentials and the subobject classifier below. The symbols and notation follow \cite{pJ}.

\section{The Yoneda Embedding}
In the category of $(X,M)$-graphs the representable $\underline V\defeq \DD G_{(X,M)}(V,-)$ consists of one vertex corresponding to the identity morphism and an empty arc set. In the reflexive case, $\underline V\defeq \rG_{(X,M)}(V,-)$ also has one distinguished loop corresponding to the morphism $\ell\colon A\to V$.  The representables $\underline A\defeq \DD G_{(X,M)}(A,-)$ and $\underline A\defeq \rG_{(X,M)}(A,-)$ each have vertex set equal to $X$ corresponding to each morphism $x\colon V\to A$ and arc set equal to $M$. The right-actions are given by Yoneda, e.g., $\underline \sigma=Y(\sigma)\colon \underline A\to \underline A$. Observe that each representable has no unfixed loops.

\begin{example} 
	 Let $X=\{s,t\}$. The Yoneda embedding gives the following diagrams, 
	
	\begin{center}
		$\text{$\widehat{\oG}_X:$}\qquad \underline V\ $\framebox{ 
			$\xymatrix{ v_1 }$} $\xymatrix{ \ar@<+.3em>[rr]^{\underline s} \ar@<-.3em>[rr]_{\underline t} &&}$
		\framebox{ 
			$\xymatrix{ v_s  \ar[rr]^{a_{1}} &&  v_t }$} $\xymatrix{ \underline A} \qquad $
	\end{center}
	\begin{center}
		$\text{$\widehat{\roG}_X:$}\qquad \underline V\ $\framebox{ 
			$\xymatrix{ v_1 \ar@{..>}@(ul,ur)[]^{a_\ell}}$} $\xymatrix{ \ar@<+1.5em>[rr]^{\underline s} \ar[rr]^{\underline t} \ar@{<-}@<-.5em>[rr]_{\underline \ell} &&}$
		\framebox{ 
			$\xymatrix{ v_s \ar@{..>}@(ul,ur)[]^{a_{s\ell}} \ar[rr]^{a_{1}} &&  v_t \ar@{..>}@(ul,ur)[]^{a_{t\ell}}}$} $\xymatrix{ \underline A \ar@(ul,ur)[]^{\underline{s}\circ \underline \ell} \ar@(dl,dr)[]_{\underline{t}\circ \underline \ell }}\qquad$
	\end{center}
	\begin{center}
		$\text{$\widehat{\sG}_X:$}\qquad \underline V\ $\framebox{ 
			$\xymatrix{ v_1 }$} $\xymatrix{ \ar@<+.3em>[rr]^{\underline s} \ar@<-.3em>[rr]_{\underline t} &&}$
		\framebox{ 
			$\xymatrix{ v_s  \ar@{-}[rr]^{a_{1}\sim a_{i}} &&  v_t }$} $\xymatrix{ \underline A \ar@(ur,dr)[]^{\underline i}}$
	\end{center}
	\begin{center}
		$\text{$\widehat{\srG}_X:$}\qquad \underline V\ $\framebox{ 
			$\xymatrix{ v_1 \ar@{..}@(ul,ur)[]^{a_\ell}}$} $\xymatrix{ \ar@<+1.5em>[rr]^{\underline s} \ar[rr]^{\underline t} \ar@{<-}@<-.5em>[rr]_{\underline \ell} &&}$
		\framebox{ 
			$\xymatrix{ v_s \ar@{..}@(ul,ur)[]^{a_{s\ell}} \ar@{-}[rr]^{a_{1}\sim a_{i}} &&  v_t \ar@{..}@(ul,ur)[]^{a_{t\ell}}}$} $\xymatrix{ \underline A \ar@(ur,dr)[]^{\underline i} \ar@(ul,ur)[]^{\underline{s}\circ \underline \ell} \ar@(dl,dr)[]_{\underline{t}\circ \underline \ell }}$
	\end{center}
	where $i\colon \{s,t\}\to \{s,t\}$ is the non-trivial automapping, $\underline s, \underline t$ are the symmetric $X$-graph morphisms which pick out $v_s$ and $v_t$ respectively, and $\underline i$ is the symmetric $X$-graph morphism which swaps $v_s$ with $v_t$ and $a_1$ with $a_i$. In the reflexive case, $\underline \ell$ is the terminal morphism, $\underline s\circ \ell, \underline t\circ \ell$ takes each arc to $a_{s\ell}$ and $a_{t\ell}$ respectively, and $\underline i$ swaps loops $a_{s\ell}$ with $a_{t\ell}$ and $a_1$ with $a_i$.
\end{example}

\section{Exponentials}\label{S:Exponential}
Let $G$ and $H$ be (reflexive) $(X,M)$-graphs. By Yoneda and the exponential adjunction, 
\begin{align*}
G^H(V)&=\widehat{\DD G}_{(X,M)}(\underline V,G^H)\iso \widehat{\DD G}_{(X,M)}(\underline V\x H, G)\\
G^H(A)&=\widehat{\DD G}_{(X,M)}(\underline A, G^H)\iso \widehat{\DD G}_{(X,M)}(\underline A\x H, G)
\end{align*}
with right-actions being defined by precomposition. For example, given an arc in $G^H$ represented by the morphism $f\colon \underline A\x H\to G$, for each $x\in X$, $f.x\defeq f\circ (\underline x\x H)\colon \underline V\x H\to G$. The evaluation morphism is defined on components 
\begin{align*}
\ev_V&\colon \widehat{\DD G}_{(X,M)}(\underline V\x H, G)\x H(V) \to G(V), \qquad (\gamma,v)\mapsto \gamma_{V}(\Id_V,v),\\
\ev_A&\colon \widehat{\DD G}_{(X,M)}(\underline A\x H,G)\x H(A)\to G(A),\ \qquad (\delta,a)\mapsto \delta_{A}(\Id_A,a).
\end{align*}
Thus for $(X,M)$-graphs, the vertex set $G^H$ is given by $G(V)^{H(V)}$ since $\underline V$ has just a single vertex with no arcs. For reflexive $(X,M)$-graphs, since $\underline V$ is the terminal object, $\underline V\x H\iso H$, the vertex set is given by the homset $G^H(V)=\widehat{\DD G}_{(X,M)}(H,G)$.

To give a description of the arc set of the (reflexive) $(X,M)$-graph $G^H$. We define a set map analogous to taking a homset of a category
\[
\overline G\colon G(V)^X\to 2^{G(A)},\quad (v_x)_{x\in X}\mapsto \setm{\beta\in G(A)}{\forall x\in X,\ \beta.x=v_x}.
\] 
We recall that the graph $\underline A\x H$ has a parametrized incidence operator  
 
 \[
    \partial\colon \underline A\x H(A)\to ((\underline A \x H)(V))^X.
 \]  For each set map $f\colon X\x H(V)=(\underline A\x H)(V)\to G(V)$ we compose to obtain the following diagram.
\[
\xymatrix{ \underline A\x H(A) \ar[d]_{\partial} \ar[r]^-{G_f\defeq\overline Gf^X\partial} & 2^{G(A)} \\ (X\x H(V))^X \ar[r]^-{f^X} & G(V)^X \ar[u]_{\overline G} } 
\]
We see that $G_f\defeq \overline Gf^X\partial(a_\sigma, \alpha)$ is the set of arcs in $G$ with the same set of incident vertices determined by the value of $f$ on the incident vertices of the arc $(a_\sigma, \alpha)$ in $\underline A\x H(A)$. Observe that a morphism $g\colon \underline A\x H\to G$ is determined on the arcs of $H(A)$, i.e., given an arc $(a_\sigma, \alpha)\in \underline A\x H(A)$ we have $g_A(a_\sigma,\alpha)=g_A(a_1,\alpha).\sigma$.

The general formula for the arc set of exponentials of  non-reflexive $(X,M)$-graphs is as follows
\begin{align*}
G^H(A)& \textstyle = \bigsqcup_{f\in (G^H(V))^X}\prod_{\alpha\in H(A)}G_f(a_1,\alpha)
\end{align*}
Thus an arc in $G^H$ is given by a pair $(f=(f_x)_{x\in X},g)$ where \\$\left(f_x\colon H(V)\to G(V)\right)_{x\in X}$ is a family of set maps and $g\colon H(A)\to G(A)$ is an element in the product $\prod_{\alpha\in H(A)}G_f(a_1,\alpha)$. Note that $((f_x)_{x\in X},g)$ is an arc in $G^H$ implies $f\colon X\x H(V)\to G(V)$ has at least one extension to a morphism $\underline A\x H\to G$. 

Given a family of set maps $\left(f_x\colon H(V)\to G(V)\right)_{x\in X}$, we define $\overline f\colon H(V)^X\to G(V)^X$ where $\overline f(h)(x)\defeq f_x(h(x))$ for each $h\in H(V)^X$ and $x\in X$. Then the set of arcs has an equivalent description 
\[
	G^H(A)=\setm{((f_x)_{x\in X},g)\in (G^H(V))^X\x G(A)^{H(A)}}{\overline f\circ \partial_H=\partial_G\circ g}
\]
i.e., it is the set of pairs $((f_x)_{x\in X},g)$ such that $g(\alpha).x=f_x(\alpha.x)$ for each $\alpha\in H(A)$ and $x\in X$. In diagram form we require that the following commute
\[
	\xymatrix{ H(A) \ar[d]_{\partial_H} \ar[r]^g & G(A) \ar[d]^{\partial_G} \\ H(V)^X \ar[r]^{\overline f} & G(V)^X.}
\]
The right-actions are given by $((f_x)_{x\in X},g).x=f_x$ for each $x\in X$ and \\ $((f_x)_{x\in X},g).\sigma=((f_{\sigma(x)})_{x\in X},g.\sigma)$ for each $\sigma\in M$ where $g.\sigma\colon H(A)\to G(A)$ takes $\alpha$ to $g(\alpha.\sigma)$. In other words, the following commute for each $x\in X$ and $\sigma\in M$.
\[
	\xymatrix@C=5em{ & G(V) \ar@{>->}[dr] & \\ H(V) \ar[ur]^{((f_x)_{x\in X},g).x} \ar@{>->}[r]^-{\iota_x} & \underline A\x H \ar[r]^{((f_x)_{x\in X},g)} & G}
\]	
\[	\xymatrix@C=5em{ \\ \underline A\x H \ar@/^3em/[rr]^-{((f_x)_{x\in X},g).\sigma} \ar[r]^-{\underline \sigma\x 1} & \underline A \x H \ar[r]^{((f_x)_{x\in X},g)} & G }
\]
where $\iota_x\colon H(V)\to \underline A\x H$ sends vertex $v$ to $(x,v)$.

In the reflexive case, given a family of morphisms $(f_x\colon H\to G)_{x\in X}$, we define $f\colon X\x H(V)\to G(V)$, $(x,v)\mapsto f_x(v)$.   Then the formula above hold for the reflexive case as well. We have 
\begin{align*}
	G^H(V)&=\widehat{\rG}_{(X,M)}(H,G)\\
	G^H(A)&=\textstyle \bigsqcup_{f\in (G^H(V))^X}\prod_{\alpha\in H(A)}G_f(a_1,\alpha).
\end{align*}
Alternatively, 
\[
	G^H(A)=\setm{((f_x)_{x\in X},g)\in (G^H(V))^X\x G(A)^{H(A)} }{\overline f\circ \partial_H=\partial_G\circ g}
\] as above. Then an arc in $G^H$ is given by a pair $((f_x)_{x\in X},g)$ where $(f_x)_{x\in X}$ is a family of graph morphisms $f_x\colon H\to G$ and $g\colon H(A)\to G(A)$ is an element in the product $\prod_{\alpha\in H(A)}G_f(a_1,\alpha)$. Then for each $x\in X$, $((f_x)_{x\in X},g).x=f_x$. Given a morphism $k\colon H\to G$ (i.e., a vertex in $G^H$), $k.\ell=((k)_{x\in X}, k_A)$ where $k_A\colon H(A)\to G(A)$ is the evaluation of $k$ at the arc component. For each $\sigma\in M$,  $((f_x)_{x\in X},g).\sigma=((f_{x.\sigma})_{x\in X}, g.\sigma)$ where $g.\sigma\colon H(A)\to G(A)$ takes $\alpha$ to $g(\alpha.\sigma)$.

The evaluation morphism $\ev\colon G^H\x H\to G$ for (reflexive) $(X,M)$-graphs is given as 
\begin{align*}
&\ev_V\colon G^H(V)\x H(V) \to G(V),\qquad (h, v) \mapsto h(v), \\
&\ev_A\colon G^H(A)\x H(A)\to G(A),\qquad (((f_x)_{x\in X},g),\alpha)\mapsto g(\alpha)
\end{align*}

\begin{example}\label{E:Exponentials}\mbox{}
	\begin{enumerate}
		\item (Characterization of Simple Graphs via the $k$-partite $X$-cover) The canonical $k$-partite $X$-cover $\pi_G\colon \underline A\x G\to G$ of a $(X,M)$-graph $G$ is a fixed loop in the endomorphism $(X,M)$-graph $G^G$ with incident vertex the identity map $\Id_{G(V)}\colon G(V)\to G(V)$. By the adjoint relation, fixed loops in an exponential object $G^H$ correspond to the homset $\widehat{\DD G}_{(X,M)}(1,G^H)\iso \widehat{\DD G}_{(X,M)}(H,G)$. It is easily checked that the loop $\pi_G$ corresponds to the identity morphism $\Id_G\colon G\to G$ in $\widehat{\DD G}_{(X,M)}$. 
		An $(X,M)$-graph $G$ is simple (see Chapter \ref{S:Simple}) if and only the $k$-partite $X$-cover $\pi_G$ is the only fixed loop at the vertex $\Id_{G(V)}$ in the endomorphism object $G^G$.
		\item \label{E:HypergraphFail} (Taking $\underline V$ as the Exponent) Let $X$ be a nonempty $M$-set. Then the exponential of $\underline V^{\underline V}$ in $\widehat{\DD G}_{(X,M)}$ is the terminal object $1$, which has one vertex and one fixed loop. This is an example of a creation of an arc from two $(X,M)$-graphs with no arcs. 
		
		More generally, for an arbitrary $(X,M)$-graph $G$, $G^{\underline V}$ is the $(X,M)$-graph with vertex set $G^{\underline V}=G(V)$ and arc set $G^{\underline V}(A)=G(V)^X$ with right-actions $f.x=f(x)$, $f.\sigma=f\circ \sigma$ for each $x\in X$, $f\in G^{\underline V}(A)$ and $\sigma\in M$, i.e., it is a $\neg\neg$-sheaf in the category of $(X,M)$-graphs (cf, Chapter \ref{S:Simple}).  
		
		\item  (Labeling Products)  Consider some category of $(X,M)$-graphs and let $C$ be one of the objects described in Chapter \ref{S:Colored} which classifies a type of labeling of vertices and/or arcs. Then a product $G\x H$ admits this type labeling if and only if $G$ admits a morphism to $C^H$ by the adjunction $-\x H\dashv (-)^H$ (cf, \cite{mE}). 
		
		\item \label{E:XLoops} (Creation of 2-Loops in Exponentials of Symmetric $X$-Graphs) Let $X$ be a set with cardinality greater than 1 and consider the symmetric $X$-graph $L$ such that $L(V)\defeq \{v\}$ and $L(A)\defeq \{0, 1\}$ where $0.\sigma=0$ and $1.\sigma=1$ for each $\sigma\in \sX$. The vertex set for $L^{\underline A}$ is a singleton $\{v \}$ since $L(V)$ is a singleton. The set $\underline A_A=\setm{(a_1,a_\sigma)}{\sigma\in \sX}\iso \sX$ and thus the set of arcs is  $L^{\underline A}(A)\iso \mathbf{Set}(\sX, \{0,1\})$. We show that $L^{\underline A}$ contains a unfixed loop. Consider a loop given by a set map $g\colon \sX\to \{0,1\}$ such that $\Id_X\mapsto 0$ and $\sigma\mapsto 1$ for the permutation $\sigma\colon X\to X$ which swaps distinct elements $x$ and $x'$ and leaves the rest fixed. Then $g.{\sigma}(\sigma)=g(
		\sigma\circ \sigma)=g(\Id_X)=0$ and thus $g.\sigma\neq g$.  Therefore $L^{\underline A}$ contains a unfixed loop. 
		
		For example, when $X=\{s,t\}$, the exponential $L^{\underline A}$ has the following undirected representation.
\end{enumerate}
\newpage	
\noindent		\begin{multicols}{2}
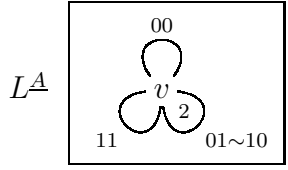
			\begin{center}
				$L^{\underline A}\ \ $\framebox{ 
					$\xymatrix{ v \ar@{-}@(ul,ur)[]^{00} \ar@{-}@(l,d)[]_{11} \ar@{-}@(d,r)[]_{01\sim 10}^2}$}
			\end{center}
			\begin{align*}
			& \text{ \footnotesize $L^{\underline A}(A)=\mathbf{Set}(s(2),\{0,1\})=\{00,11,01\sim 10\}$}\\
			& \text{ \footnotesize $L^{\underline A}(V)=\{v \}$}\\
			& \text{ \footnotesize $00.i=00, \ 11.i=11, 01.i=10$}
			\end{align*}
		\end{multicols}
\begin{enumerate}[resume]
		\item[] where $i\colon 2\to 2$ is the non-identity automorphism and $xy\colon s(2)\to s(2)$ is the set map $xy(\Id_X)=x,\ xy(i)=y$ for $x,y\in \{0,1\}$. Evaluation on arcs is given by projection, e.g., $\ev_A(xy,a_1)=x$. 
		\item \label{E:1Loops} (Creation of 2-Loops in Exponentials of Reflexive Symmetric $X$-Graphs) Let $X$ be a set of cardinality greater than $1$ and consider the reflexive symmetric $X$-graph $L$ such that $L(V)\defeq \{v\}$ and $L(A)\defeq \{0,1 \}$. For each $\sigma\in \sX$, we set $i.\sigma=i$ for $i=0,1$. We also set $v.\ell=0$. The vertex set of the exponential $L^{\underline A}$ is $\widehat{\rG}_{(X,\srX)}(\underline A, L)=L(A)=\{0,1 \}$ by Yoneda. Using the construction above we obtain the arrow set
		\[
		L^{\underline A}(A)=L(A)^X\x \mathbf{Set}(\srX,\{0,1\}).
		\]
		We show that $L^{\underline A}$ contains a unfixed loop. Consider the loop \\$((1)_{x\in X},g\colon \srX\to \{0,1\})$ such that $g(\Id_X)=0$ and $g(\sigma)=1$ for the automorphism $\sigma$ which exchanges two elements in $X$ and thus $g.{\sigma}(\Id_X)=g(\sigma)=1$ and $g.{\sigma}(\sigma)=g(\Id_X)=0$. Then $((1)_{x\in X},g).\sigma\neq ((1)_{x\in X}, g)$ showing $g$ is a unfixed loop in $L^{\underline A}$.
		
		For example, when $X=\{s,t\}$, the exponential $L^{\underline A}$ has arc set equal to $2^2\x \mathbf{Set}(2^2,2)$, i.e., it has $2^6=64$ elements. Each arc can be represented by a 6-digit binary number. The exponential object $L^{\underline A}$ is given as follows.
\end{enumerate}
		\begin{multicols}{2}
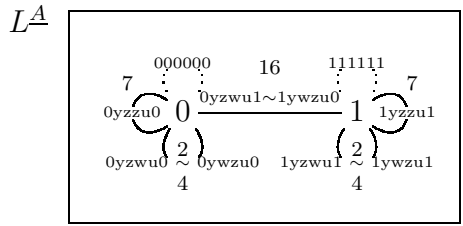
			\begin{center}
				$L^{\underline A}\ \ $\framebox{ 
					$\xymatrix{\\ \ar@<+.5em>@{}[r]^<<<<{7} & 0 \ar@{}[d]|->>4 \ar@{}@<+1em>[rr]^-{16} \ar@{-}[rr]^{\text{\tiny 0yzwu1$\sim$1ywzu0}} \ar@{..}@(ul,ur)[]|-{\text{\tiny 000000}} \ar@{-}@(ul,dl)[]|-{\text{\tiny 0yzzu0}} \ar@{-}@(dl,dr)[]|-{\text{\tiny 0yzwu0 $\sim$ 0ywzu0}}^2  && 1\ar@{..}@(ul,ur)[]|-{\text{\tiny 111111}} \ar@{-}@(ur,dr)[]|-{\text{\tiny 1yzzu1}}  \ar@{-}@(dl,dr)[]|-{\text{\tiny 1yzwu1 $\sim$ 1ywzu1}}^2 \ar@<.5em>@{}[r]^>>>>7 \ar@{}[d]|->>4 & \\ &  && }$}
			\end{center}
			\begin{align*}
			& \text{\footnotesize $L^{\underline A}(A)=\setm{(xyzwuv)}{x,y,z,w,u,v\in \{0,1\}}$}\\
			& \text{\footnotesize $L^{\underline A}(V)=\{0,1 \}$}\\
			& \text{\footnotesize $(xyzwuv).s=x,\ (xyzwuv).t=v,$}\\ 
			& \text{\footnotesize $(xyzwuv).i=(vywzux)$}.
			\end{align*}
		\end{multicols}
\begin{enumerate}[resume]
	\item[]	where $i\colon 2\to 2$ is the non-identity automap. We see that $L^{\underline A}$ has 16 fixed loops (with 7 non-distinguished fixed loops at each vertex), 8 non-fixed loops (4 at each vertex) and 16 edges between vertices.  It is helpful to keep track of the edges associated to the digits $(\underset{s}{x}\underset{\ell_s}{y}\underset{a_1}{z}\underset{a_i}{w}\underset{\ell_t}{u}\underset{t}{v})$. Then evaluation $\ev_A\colon L^{\underline A}(A)\x \underline A\to \{0,1\}$ is given by projection to the corresponding digit, e.g., $\ev_A( (\underset{s}{x}\underset{\ell_s}{y}\underset{a_1}{z}\underset{a_i}{w}\underset{\ell_t}{u}\underset{t}{v}),\ell_s)=y$.\footnote{In \cite{dP}[Proposition 2.3.1], it is proven that the category of conceptual graphs does not have exponentials by attempting to construct the corresponding exponential $L^{\underline A}$. We have given a constructive reason why it failed. Namely, the objects in the category of conceptual graphs lack 2-loops.}
		\item \label{E:BPDC}($X$-Fold Isomorphisms) Let $G$ and $H$ be $(X,M)$-graphs. The $(X,M)$-graph projection $\underline A\x G\to G$ is the $(X,M)$-graph generalization of the bipartite graph double cover \cite{dW}. By Proposition \ref{P:FoliatedGraphs} above, when $M$ is a group (e.g., in the case of oriented and symmetric $X$-graphs) then the projection $\underline A\x G\to \underline A$ is an object in $\widehat{\DD G}_{(\C F, 1)}$ which in the case $X=2$, is the category of bipartite graphs. Consider the subobject $\Iso^{\text{XF}}(H,G)$ of the exponential $G^H$ with vertex and arc sets given as follows
		\begin{align*}
			 \Iso^{\text{XF}}(H,G)(V) &\defeq \setm{f\in G^H(V)}{f\text{ is an automap of sets}}\\
			 \Iso^{\text{XF}}(H,G)(A) &\defeq \setm{((f_x)_{x\in X},g)\in G^H(A)}{\forall x\in X, f_x\text{ is an automap}}.
		\end{align*} 

 In other words, $\Iso^{\text{XF}}(H,G)$ is the induced sub-$(X,M)$-graph of $G^H$ consisting of vertices which are automaps $f\colon H(V)\to G(V)$. The $(X,M)$-graph $\Iso^{\text{XF}}(H,G)$ is the called the object of $X$-fold isomorphisms. When $X=2$, the object $\sigma(\Iso^{\text{XF}}(H,G))$ (where $\sigma$ is the simplification functor in Chapter \ref{S:Simple}) has arc set equal to the set of two-fold isomorphisms as given in \cite{jL}. Therefore $\Iso^{\text{XF}}(H,G)$ is the $(X,M)$-graph generalization of the set of two-fold isomorphisms. In the case $G=H$, it is the object of $X$-fold automorphisms $\Aut^{XF}(G)$ which contains the object $\Aut(G)$ as a sub-$(X,M)$-graph.
	\end{enumerate}
\end{example}
In a cartesian closed category, given exponentials $G^H$ and $H^K$, there is a unique morphism $c_{\text{\tiny KHG}}\colon  G^H\x H^K\to G^K$ induced by the universal mapping property of exponentials
\[
\xymatrix@C=3em{ G^H \x H^K \x K \ar[r]^-{1\x \ev_{K,H} } \ar@{..>}[d]_{c_{\text{\tiny KHG}}\x K}& G^H\x H \ar[d]^{\ev_{H,G}} \\ G^K\x K \ar[r]^-{\ev_{K,G}} & G}
\]
which internalizes composition, i.e., if $f\colon K\to H$ and $g\colon H\to G$ are morphisms, then the following commutes
\[
\xymatrix{ 1 \ar[r]^-{(\named g,\named f)} \ar[dr]_{\named{g\circ f}} & G^H \x H^K \ar[d]^{c} \\ & G^K }
\]
where $\named f\colon 1\to K^H$, $\named g\colon 1\to H^G$, and $\named{g\circ f}\colon 1\to G^K$ are the exponential adjoints to $f, g,$ and $g\circ f$ respectively. 
Thus given (reflective) $(X,M)$-graphs $G$, $H$, and $K$ composition has the following form

\begin{align*}
c_{\text{\tiny KHG},V}\colon (G^H\x H^K)(V) &\to G^K(V),\\ (f',f)&\mapsto f' f\\\\
c_{\text{\tiny KHG},A}\colon (G^H\x H^K)(A) &\to G^K(A),\\  (((f'_x)_{x\in X},g'),((f_x)_{x\in X},g))&\mapsto ((f'_xf_x)_{x\in X},g'g)
\end{align*}
This allows us to define endomorphism and automorphism (reflexive) $(X,M)$-graphs (cf, \cite{jS}). 

\begin{example}\mbox{}
 Let $G$ be an $(X,M)$-graph. Consider the automorphism $(X,M)$-graph $\Aut(G)$ and take the points of its simplification $\Gamma(\sigma(\Aut(G))$ (see Chapter \ref{S:PointsPieces} and Chapter \ref{S:Simple}) . This is a subgroup of the group of automaps $\Aut(G(V))$ in the category of sets. An $(X,M)$-graph $G$ is said to be vertex transitive provided $\Gamma(\sigma(\Aut(G)))$ acts transitively on the vertices of $G$. 
		
		Similarly, the points of the automorphism $(X,M)$-graph $\Gamma(\Aut(G))$ is a subgroup of automaps $\Aut(G(A))$ in the category of sets. An $(X,M)$-graph $G$ is said to be arc-transitive provided $\Gamma(\Aut(G))$ acts transitively on the set of arcs of $G$. 

\ignore{
		\item Let $G$ be a reflexive $(X,M)$-graph. Since there is a chaotic functor $B\colon \mathbf{Set} \to \widehat{\rG}_{(X,M)}$ which is right adjoint to the points functor \tbf{FINISH THIS}
		\item \tbf{INTERNAL GROUP and FUNCTORS}
		\item Consider 
	\end{enumerate} 
	}
\end{example}

\section{The Subobject Classifier}
The subobject classifier $\Omega$ for (reflexive) $(X,M)$-graphs is constructed via Yoneda,
\begin{align*}
\Omega(V)&=\widehat{\DD G}_{(X,M)}(\underline V, \Omega)=\Sub(\underline V)=\{\empset, \underline V\}\eqdef \{v_\bot, v_\top\}\\
\Omega(A)&=\widehat{\DD G}_{(X,M)}(\underline A, \Omega)=\Sub(\underline A)\eqdef \setm{a_E}{E\text{ is a sub-$(X,M)$-graph of }\underline A},
\end{align*} where right-actions are given by pullback. For example, 
\[{a_E.x=\begin{cases}v_\top & x\in E(V)\\ v_\bot & x\nin E(V) \end{cases}}\] for $a_E\in \Omega(A)$ and $x\in X$. For each $\sigma\in M$,  $a_E.\sigma=a_{\sigma^*(E)}$ where $\sigma^*(E)$ is given by the pullback of $E\hookrightarrow \underline A$ and $\underline \sigma\colon \underline A\to \underline A$. In the case of $(X,M)$-graphs this is just the restriction of $\sigma\colon X\to X$ to $E$. In the case of reflexive $(X,M)$-graphs, the vertex and arc sets are the same as the non-reflexive case. The distinguished loops are given as $v_\bot.\ell=a_\empset\eqdef a_\bot$ and $v_\top.\ell=a_{\underline A}\eqdef a_\top$. 
The universal subobject $\top\colon 1\to \Omega$ in both cases is the morphism which picks out the loop $a_\top$.

A sub-$(X,M)$-graph $G'\hookrightarrow G$ is associated to a characteristic morphism $\chi_{G'}\colon G\to \Omega$ such that for each $v\in G(V)$ and $\alpha\in G(A)$
\begin{align*}
\chi_{G',V}(v)&=\begin{cases}v_\top & v\in G'(V)\\ v_\bot & v\nin G'(V) \end{cases}, \\
\chi_{G',A}(\alpha)&=\begin{cases}a_\top & \alpha\in G'(A)\\ a_{E} & \alpha\nin G'(A)\text{ and }\forall x\in X,\ \alpha.x\in G'(V) \text{ iff } x\in E.  \end{cases}
\end{align*}
Then it is straightforward to show that the following is a pullback diagram
\[
\xymatrix@!=.5em{ G' \ar@{>->}[r] \ar[d] \ar@{}[dr]|-<\pb & G \ar[d]^{\chi_{G'}} \\ 1 \ar[r]^{\top} & \Omega.}
\]

\begin{example}\label{E:SubobjectClassifier}\mbox{}
	\begin{enumerate}
		\item  ((Reflexive) Symmetric $X$-Graphs) The subobject classifier for symmetric $X$-graphs  and reflexive symmetric $X$-graphs is given as follows.
		\begin{center}
			$\Omega \ \ $\framebox{ 
				$
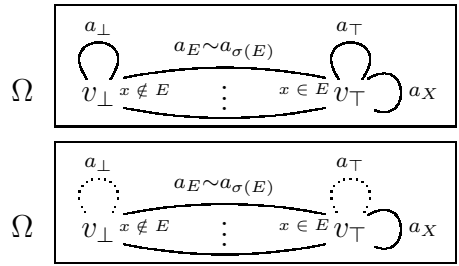
\xymatrix@C=3em{ v_\bot \ar@{-}@(ul,ur)[]^{a_\bot}\ar@/^/@<+.3em>@{-}[rr]^{a_E\sim a_{\sigma(E)}}_<<{\text{\tiny{$x\nin E$}}} \ar@{-}@/^/@<+.3em>[rr]_>>{\text{\tiny{$x\in E$}}} \ar@/_/@<-.3em>@{-}[rr]^{\vdots} && v_\top \ar@{-}@(ul,ur)[]^{a_\top} \ar@{-}@(ur,dr)[]^{a_{X}}}$}$\qquad \qquad$

		$\Omega \ \ $\framebox{ 
				$\xymatrix@C=3em{ v_\bot \ar@{..}@(ul,ur)[]^{a_\bot}\ar@/^/@<+.3em>@{-}[rr]^{a_E\sim a_{\sigma(E)}}_<<{\text{\tiny{$x\nin E$}}} \ar@{-}@/^/@<+.3em>[rr]_>>{\text{\tiny{$x\in E$}}} \ar@/_/@<-.3em>@{-}[rr]^{\vdots} && v_\top \ar@{..}@(ul,ur)[]^{a_\top} \ar@{-}@(ur,dr)[]^{a_{X}}}$}$\qquad\qquad$
		\end{center}
		where $a_E$ corresponds to a proper subset $E$ of $X$ and is associated to $a_{\sigma(E)}$ for each automap $\sigma\colon X\to X$, $a_\bot$ corresponds to the empty set, $a_X$ the set $X$, and $a_\top$ to $\underline A$. In the reflexive case, the loops corresponding to false $a_\bot$ and true $a_\top$ are distinguished.
			
		\item \label{E:DirSubC} ($(S,T)$-Directed Symmetric $X$-Graphs) Let $S$ and $T$ be sets and $X\defeq S\sqcup T$. The subobject classifier for $(S,T)$-directed symmetric $X$-graphs is given by
		\begin{center}
			$\Omega \ \ $\framebox{ 
				$\xymatrix@C=4em{ v_\bot \ar@(ul,ur)[]^{a_\bot}\ar@/^/@<+.3em>@{-}[rr]^{a_{(S',T')}\sim a_{(\sigma(S'),\tau(T')}} \ar@/_/@<-.3em>@{-}[rr]^{\vdots} && v_\top \ar@(ul,ur)[]^{a_\top} \ar@(ur,dr)[]^{a_{(S,T)}}}$}
		\end{center}
		
		where $S'$ and $T'$ are subsets of $S$ and $T$ respectively and $\sigma\colon S\to S$ and $\tau\colon T\to T$ are automaps. For $s\in S$, the $s$-incidence is given by $a_{(S',T')}.s=\begin{cases} a_\top & s\in S'\\ a_\bot & s\nin S' \end{cases}$ and similarly for eacy $t$-incidence, $t\in T$. For instance, if $S=\{a,b,c\}$ and $T=\{0,1\}$, the arc $a_{\{a,c\},\{0\}}$ and one of its $\dsX$-equivalent partners $a_{\{a,b\},\{1\}}$ are  depicted as
		\begin{center}
			$\xymatrix@R=.1em@C=4em{ v_\top \ar@/_/[dr] \ar@{..>}[dd] & & v_\top \ar@/_/[dr] \ar@{..>}[dd] & \\ & v_\top \ar@{..>}[dd] && v_\bot \ar@{..>}[dd] \\ v\bot \ar@{}[r]|-{a_{\{a,c\},\{0\}}} \ar@{..>}[dd] &  \ar@{}[r]|-{\sim}  & v_\top \ar@{}[r]|-{a_{\{a,b\},\{1\}}} \ar@{..>}[dd] & \\ & v_\bot  && v_\top \\ v_\top \ar@/^/[ur] & & v_\bot \ar@/^/[ur] }$
		\end{center}

		\item \label{E:BipartSOC} (Bipartite Graphs) Let $\widehat{\DD G}_{(\C F,1)}$ be the category of bipartite graphs, i.e., the category of presheaves on \\$\xymatrix{V_1 \ar[r]^{x_1} & A & V_2 \ar[l]_{x_2}}$ (see Example \ref{E:Hybrid}.\ref{E:Bipart}). The Yoneda embedding is given by 
		\begin{center}
			$\underline V_1 \ \ $\framebox{ 
				$\xymatrix@C=3em{v_1 \ar@{}[r] & }$}$\xymatrix{ \ar[r]^{\underline{x_1}} &}$
			$\underline A \ \ $\framebox{ 
				$\xymatrix@C=3em{v_{x_1} \ar[r]^{a_1} & v_{x_2}}$}$\xymatrix{ \ar@{<-}[r]^{\underline{x_2}} & }$
			\framebox{ 
				$\xymatrix@C=3em{ \ar@{}[r] & v_2 }$}$\ \ \underline V_2$
		\end{center}
		Therefore, the subobject classifier is the following bipartite graph
		\begin{multicols}{2}
			\begin{center}	
				$\Omega\ \ $\framebox{ 
					$\xymatrix@R=3em@C=5em{ v_{1\bot} \ar[dr]|-<<<<{a_{x_2}} \ar[r]^{a_\bot} & v_{2\bot} \\  v_{1\top} \ar[ur]|->>>>{\ a_{x_1}} \ar@<+.1em>[r]^{a_{12}} \ar@<-.5em>[r]_{a_\top}  & v_{2\top}  }$}
			\end{center}
			\begin{align*}
			&\Omega(V_1)=\{\empset,\underline V_1\}\eqdef \{v_{1\bot},v_{1\top}\},\\
			&\Omega(V_2)=\{\empset, \underline V_2\}\eqdef \{v_{2\bot},v_{2\top}\},\\ &\Omega(A)=\{\empset,\underline V_1, \underline V_2, \{v_1,v_2 \}, \underline A \} \\
			&\qquad \eqdef \{a_\bot, a_{x_1}, a_{x_2}, a_{12},a_1 \}.
			\end{align*}
		\end{multicols}
		
	\end{enumerate}
\end{example}

\section{The Internal Logic of $(X,M)$-Graphs}
In this section, we show how to use the internal logic of presheaf outlined in \cite{mR} (Ch. 9) to understand relations between substructures of $(X,M)$-graphs. Since the category of $(X,M)$-graphs is a category of presheaves, the subobject preorders $\Sub(G)$ have both a frame and coframe structure for each $(X,M)$-graph $G$  and thus are in particular bi-Heyting algebras (\cite{mR} Proposition 9.1.1, 9.1.11). 

The following examples show how we can use the internal logic of negation and subtraction to define common notions used in graph and hypergraph theory. Recall that the negation of a subobject $G'\hookrightarrow G$ is the largest subobject $\neg G'\hookrightarrow G$ such that $\neg G'\mm G'$ is the initial object and the subtraction of $G'\hookrightarrow G$ is the smallest subobject $\sim G'\hookrightarrow G$ such that $\sim G'\jj G'$ is $G$. Note that only negation is natural, i.e., is respected by the morphisms. 

\begin{example}\label{E:TNN} \mbox{}
	\begin{enumerate}
		\item (Strong Vertex Deletion)	Let $G$ be a (reflexive) $(X,M)$-graph and $v$ a vertex in $G$. A strong deletion of $v$ from $G$ is defined to be the subgraph $\neg v$ (cf. \cite{vV}, p 146). 
		\item \label{E:DR}	 (Decks and Reconstruction) Let $G$ be a (reflexive) $(X,M)$-graph. Then the deck of $G$ is the $(X,M)$-graph $\text{Deck}(G)\defeq \bigsqcup_{v\in G(V)}\neg v$.  The classical graph reconstruction conjecture is as follows:
		\begin{quote}
			\textbf{Vertex Reconstruction Theorem}: Let $G$ and $H$ be simple symmetric $2$-graphs with at least 3 vertices. If $\text{Deck}(G)\iso \text{Deck}(H)$, then $G\iso H$. 	
		\end{quote}
		Since monomorphisms of $(X,M)$-graphs restrict to subobjects, we have the functor $\text{Deck}\colon \text{Mono}(\widehat{\DD G}_{(X,M)})\to \text{Mono}(\widehat{\DD G}_{(X,M)})$. A slightly stronger verson of the reconstruction theorem is that $\text{Deck}$ is conservative.
		
		Similarly, to obtain the deck of edge deleted sub-$(X,M)$-graphs of a graph $G$, we use  $\text{Deck}_e(G)\defeq \bigsqcup_{[a]\in \frac{G(A)}{M}} \sim a$ where $\frac{G(A)}{M}$ is the set of arcs modulo the equivalence generated by the $M$-action. 
		\item (Stars, Transversals) An $(X,M)$-graph $G$ is called a star provided it is connected and there is a vertex $v$ in $G$ such that the subgraph $\neg v$ of $G$ has no arcs (i.e. $\neg v$ is discrete). A subset $S$ of vertices of a $(X,M)$-graph $G$ is called a transversal provided $\neg S$ has no arcs (i.e., $\neg S$ is discrete). 
		\item (Intersecting Families) An $(X,M)$-graph $G$ is called an intersecting family provided for each subobject $\named \alpha\colon \underline A\hookrightarrow G$, the subobject $\neg \named \alpha$ has no arcs (cf, \cite{vV}, p 155). 
		\item (Vertex Cover) We say a sub-$(X,M)$-graph $S\hookrightarrow G$ is a vertex cover provided $\neg S$ is the empty $(X,M)$-graph. 
		\item (Induced Sub-$(X,M)$-Graph) Given a sub-$(X,M)$-graph $G'\hookrightarrow G$, we call $\neg\neg G'$ the induced sub-$(X,M)$-graph of $G'$ ((cf, \cite{yL}, p 12)). 
		\item (Cuts and Minimal Cuts) A cut for an $(X,M)$-graph $G$ is defined to be pairs of disjoint sub-$(X,M)$-graphs $(H,K)$ such that $H(V)\cup K(V)=G(V)$ and $\alpha\in G(A)\backslash (H(A)\cup K(A))$ iff $\alpha$ is incident to a vertex in $H$ and a vertex in $K$. Given any sub-$(X,M)$-graph $S\hookrightarrow G$, the pair $(\neg\neg S,\neg S)$ is a cut. In fact, all cuts are of this form. Moreover, the set of arcs between $\neg\neg S$ and $\neg S$ is the arc set in the $(X,M)$-subgraph ${\sim(\neg\neg S\jj \neg S)}$. A nonempty set of vertices $S$ is said to be minimal provided $\sim(\neg\neg S \jj \neg S)$ has an arc set with a minimal cardinality with respect to all other nonempty subsets of vertices. 
		\item (Independent Sets) A subset $S$ of vertices of a $(X,M)$-graph $G$ is called an independent set provided $\neg\neg S=S$, i.e., there is no arc in $G$ with incidence contained in $S$. The independence cardinality of a hypergraph $G$ is the maximum cardinality of an independent set in $G$ (cf. \cite{vV}, p 151-152). 
		\item (Neighbor Operator) We define the neighbor operator 
		\[ \text{nbr}\defeq \sim \neg\colon \Sub(G) \to \Sub(G).\] For example, if $v$ is a vertex in $G$, then $\text{nbr}(v)$ is the subgraph of $G$ such that 
		\begin{align*} 
		\text{nbr}(v)(V)&=\setm{w\in G(V)}{\exists \alpha\in G(A), \exists x, x'\in X,\ \alpha.x=v \text{ and } \alpha.x'=w}\\
		\text{nbr}(v)(A)&=\setm{\alpha\in G(A)}{\exists x,x'\in X,\ \alpha.x=v \text{ and } \alpha.x'=w}.
		\end{align*} 
		In other words $\text{nbr}(v)$ is the subgraph containing $v$ and all arcs which contains $v$.\footnote{This is sometimes called the closed neighborhood of $v$ (cf. \cite{vV}. Definition 2.1.3)} Then $\text{nbr}^n(v)$ is the subgraph of $G$ which contains $v$ and all arcs which are in a sequence of arcs of length $n$ from $v$. The degree of $v$ is the cardinality of the arc set $\text{nbr}(v)$ (modulo the maximal subgroup of invertible elements of $M$). 
		\item \label{E:Kneser} (Kneser Hypergraph) Let $H$ be a hypergraph and $\Int_X(H)$ the intersection graph construction for a set $X$ (see Example \ref{E:Hybrid}(\ref{E:Int})). Let $\Int_X(H)\hookrightarrow \widetilde{\Int_X(H)}$ be the injective hull of $\Int_X(H)$ (see Chapter \ref{S:InjProj}). Then $\sim \Int_X(H)$ as a subobject of $\widetilde{\Int_X(H)}$ is the Kneser hypergraph $\text{KG}^k(H)$ (where $k$ is the cardinality of $X$) as defined in \cite{hH}. To obtain the classical definition of the Kneser hypergraph, take the maximal $k$-uniform hypergraph contained in $\sim \Int_X(H)$ where $k$ is the cardinality of $X$.
	\end{enumerate} 
\end{example}

Note that each of the definitions/characterizations above which are defined by negation is preserved under geometric morphisms which includes all essential geometric morphisms. However, subtraction is not preserved and so more care is needed when discussing transfer of structures under functors. 

\section{The Topos Properties}\label{S:PointsPieces}

Since each small category $\DD T$ has a unique functor to the terminal category $\mbf 1$, there is an essential geometric morphism 
\[
\Pi\dashv \Delta\dashv \Gamma\colon \widehat{\DD T} \to \mathbf{Set}
\]
where the functors are called the pieces, discrete, and points functor respectively (see \cite{mR} Chapter 11). 

The pieces functor $\Pi$ on a presheaf $G$ gives the set of connected pieces of $G$ and gives us a functorial definition of connectivity.\footnote{A $(X,M)$-graph $G$ is connected iff $\Pi(G)\iso 1$ the terminal set.} Recall that a category of presheaves is a connected topos provided the pieces functor preserves the terminal object. Since the terminal object in a category of (reflexive) $(X,M)$-graphs ($X\neq \empset$) is connected, the categories of $(X,M)$-graphs are connected toposes.  However, if $X$ has cardinality greater than $1$, $\Pi$ does not preserve finite products and hence is not strongly connected (\cite{pJT}, p 134). 

The discrete functor $\Delta$ on a set $Y$ gives a presheaf where each level is the set $Y$ and each right-action is the identity. The points functor $\Gamma$ on a presheaf $G$ gives the set of morphisms from the terminal set $1$ to $G$. In the reflexive case, this is just the evaluation of the presheaf on vertices, i.e., $\Gamma(G)=G(V)$. When the points functor $\Gamma$ admits a right adjoint $B\colon \mathbf{Set}\to \widehat{\DD T}$, this functor is called the chaotic functor. The following result is due to Lawvere \cite{wL89}.
\begin{proposition}
	The points functor $\Gamma\colon \widehat{\DD T}\to \mathbf{Set}$ admits a right adjoint $B\colon \mathbf{Set} \to \widehat{\DD T}$ if and only if at least one representable has a point. 	
\end{proposition}

In the reflexive case, $\underline V$ is the terminal object, hence by the proposition categories of reflexive $(X,M)$-graphs admit a chaotic functor. We have the functor of points is equivalent to the evaluation functor on the vertex set, i.e., $\Gamma=(-)_V\colon \widehat{\DD G}_{(X,M)}\to \mathbf{Set}$ and so the chaotic functor $B$ on a set $S$ gives the reflexive $(X,M)$-graph with vertex set $B(S)(V)\defeq S$ and arc set $B(S)(A)\defeq S^X$ with $s.\ell$ being the constant function on $s$ for vertex $s\in S$, $f.x\defeq f(x)$ and $f.m=f\circ m$ for arc $f\colon X\to S$.   

However, in the case of non-reflexive $(X,M)$-graph $\underline V$ does not admit a point since it is loopless, while $\underline A$ admits a morphism from $1$ only if there is a $m\in M$ and $y\in X$ such that for each $x\in X$ we have $x.m=y$. Hence, there is no chaotic functor for categories of non-reflexive $(X,M)$-graphs when $M$ is a group. For example, the categories of hereditary $X$-graphs is the only category of non-reflexive $X$-graphs which admits a chaotic functor.  

Also recall that a topos $\C E$ is \'etendue provided there is an object $E$ such that $\C E\ls E$ is generated by the subobjects of its terminal object. It is shown in \cite{kR} (Theorem 1.5) that a category of presheaves $\widehat{\C C}$ is \'etendue if and only if every morphism in $\C C$ is a monomorphism. When $X$ has cardinality greater than 1, the category of reflexive $(X,M)$-graphs is not an \'etendue topos because $\ell\colon A\to V$ in $\rG_{(X,M)}$ is not a monomorphism.  However, in the non-reflexive case we have the following.

\begin{proposition}
	The category of $(X,M)$-graphs is \'etendue if and only if $M$ is a right cancellative monoid. 
\end{proposition}
\begin{proof}
	The condition is equivalent to every morphism in $\DD G_{(X,M)}$ being a monomorphism.  
	
\end{proof}
\noindent In particular, when $M$ is a group the category of $(X,M)$-graphs is \'etendue, e.g., the categories $\widehat{\oG}_X$ or $\widehat{\sG}_X$.

\section{Aside: The Various Products of Graphs} \label{S:PPP}

In this section we show how the various products of graphs satisfy certain universal properties thereby making these constructions categorical. We restrict our attention to the category of simple loopless symmetric $2$-graphs which we will call "graph" in this section. However, each of these products could be easily generalized to symmetric $X$-graphs.  We will be frequently be using the functors involved in essential geometric morphism $r_!\dashv r^*\dashv r_*\colon \colon \widehat{\sG}_2\to \widehat{\srG}_2$ induced by the obvious morphism of theories $\sG_2\to \srG_2$.

\textbf{The Strong Product}: Given graphs $G$ and $H$, the strong product is given by $G\boxtimes G\defeq r_!(G)\x r_!(H)$. Therefore, the universal property occurs in the category of reflexive symmetric $2$-graphs where distinguished loops are viewed as vertex proxies. 

\textbf{The Cartesian Product}:   Given graphs $G$ and $H$, the cartesian product is given by $G\square H\defeq \sim(r_!(G\times H))$ as a subobject of $G\boxtimes H$. Therefore, it is the smallest subobject $K$ of $G\boxtimes H$ such that the union $K\cup r_!(G\times H)=G\boxtimes H$. Then given any monomorphism $K\to G\boxtimes H$ such that $K\sqcup r_!(G\x H)\to G\boxtimes H$ is an epimorphism, there exists a unique monomorphism $f\colon G\square H\to H$ such that the following commutes
\[
	\xymatrix{ (G\square H)\sqcup r_!(G\x H) \ar@{-->}[d]_{f\sqcup \Id} \ar@{->>}[dr] \\ K\sqcup r_!(G\x H) \ar@{->>}[r] & G\boxtimes H.  }
\]

\textbf{The Co-normal Product}: Given graphs $G$ and $H$, the co-normal product is given by  $G*H\defeq (G\times K_{H(V)})\cup (K_{G(V)}\times H)$ as a subobject of $K_{G(V)}\times K_{H(V)}$ where $K_n$ is the complete graph (including loops) on $n$ vertices.  

\textbf{The Lexicographical Product}: Given graphs $G$ and $H$, the lexicographical product is given by $G\cdot H\defeq (G\times K_{H(V)})\cup( r^*r_!(G)\x H)$ as a subobject of $K_{G(V)}\x K_{H(V)}$. 

\textbf{The Modular Product}: Given graphs $G$ and $H$, the modular product is given by $(G\times H)\cup ((\sim G)\times (\sim H))$ as a subgraph of $K_{G(V)}\times K_{H(V)}$ where $\sim G$ is a subgraph of $K_{G(V)}^{ll}$ and $\sim H$ is a subgraph of $K_{H(V)}^{ll}$ (and $K_n^{ll}$ is the loopless complete graph on $n$ vertices).

\chapter{Simple $(X,M)$-Graphs}\label{S:Simple}\label{C:Simple}
The symbols and notation in this section follow from \cite{pJ}, Section C2.

\begin{definition}
	A (reflexive) $(X,M)$-graph $G$ is simple provided for any collection of vertices $(v_x)_{x\in X}$ (with multiplicities allowed), there is at most one edge $e$ incident with $(v_x)_{x\in X}$. 
\end{definition}
If the $(X,M)$-graph has each edge with incidence of distinct vertices, i.e., no multiplicities, then this definition corresponds to the usual one. However, a simple $(X,M)$-graph may also contain multiple edges that contain the same underlying vertices but each with different multiplicities. For example, if $X=3$ and we are given a symmetric $X$-graph $G$ such that $G(V)=\{0,1\}$ and $G(A)=\{001,010,100,011,101,110\}$ such that $001\sim 010\sim 100$ and $011\sim 101\sim 110$ under $s(3)$, then $G$ is simple even though it has two distinct edges with the same underlying set of vertices. However, notice that all arcs are fixed. In particular, simple $(X,M)$-graphs have no unfixed loops.

We recall that a Lawvere-Tierney topology on a topos with subobject classifier $\top\colon 1\to \Omega$ is a morphism $j\colon \Omega \to \Omega$ such that $j\circ j=j$, $j\circ \top=\top$, and $\mm\circ (j\x j)=j\circ \mm$ where $\mm\colon \Omega\x \Omega\to \Omega$ is the internal meet operator. In every topos, the negation morphism $\neg\colon \Omega\to \Omega$ which classifies false $\bot\colon \Omega \to \Omega$ induces a double negation topology $\neg\neg\colon \Omega\to \Omega$. For a presheaf topos, the subobject classified by $\neg\neg$ is a Grothendieck topology $J_{\neg\neg}\hookrightarrow \Omega$.

The double negation operator is described on vertices by $\neg\neg_V(v_\bot)=v_\bot,\ \neg\neg_V(v_\top)=v_\top$. On arcs, $\neg\neg_A(a_\top)=a_\top$ and for a subgraph with a proper subset of vertices $E\hookrightarrow \underline A$, $\neg\neg_A(a_E)=a_E$ and for the sub-$(X,M)$-graph $X$, $\neg\neg_A(a_X)=a_\top$ (since the complement of the subgraph which contains all vertices of $\underline A$ is empty). Thus in the category of (reflexive) $(X,M)$-graphs, the $(X,M)$-graph (resp. reflexive $(X,M)$-graph) the subobject of $\Omega$ classified by $\neg\neg\colon \Omega\to \Omega$ is given by 
\begin{center}
	$J_{\neg\neg} \ \ $\framebox{ 
		$\xymatrix{  v_\top \ar@{-}@(ur,dr)[]^{a_\top} \ar@{-}@(ul,dl)[]_{a_X}}$} $\qquad \text{resp.}\qquad 
	J_{\neg\neg} \ \ $\framebox{ 
		$\xymatrix{  v_\top \ar@{..}@(ur,dr)[]^{a_\top} \ar@{-}@(ul,dl)[]_{a_X}}$} .
\end{center}
Therefore, the only non-trivial sieve in $J_{\neg\neg}$ is the one corresponding to the subgraph $R_X\hookrightarrow \underline A$, where $R_X$ is the proper subgraph containing all the vertices of $\underline A$. Thus a $(X,M)$-graph $G$ is a $\neg\neg$-separated presheaf iff for each collection of vertices $(v_x)_{x\in X}$ (multiplicities allowed) in $G$ there is at most one arc incident with $(v_x)_{x\in X}$. Therefore, we have shown the following.

\begin{proposition}\label{P:SimpleNegNeg}
	The category of simple (reflexive) $(X,M)$-graphs is the category of $\neg\neg$-separated (reflexive) $(X,M)$-graphs.
\end{proposition} 
By \cite{pJ} [Proposition C2.2.13], each category of simple (reflexive) $(X,M)$-graphs is a Grothendieck quasi-topos. Namely, it is a locally cartesian closed category with a regular subobject classifier. Moreover, the adjunction $\sigma\dashv i\colon \Sep(\DD G_{(X,M)}, J_{\neg\neg})\to \widehat{\DD G}_{(X,M)}$ (resp. $\sigma\dashv i\colon \Sep(\rG_{(X,M)}, J_{\neg\neg})\to \widehat{\rG}_{(X,M)}$) is a reflective subcategory such that the inclusion preserves exponentials.  The regular subobject classifier is given by the image of $\Omega$ under the reflector $\sigma$. Thus the regular subobject classifier in the category of is the same as the subobject classifier but with the omission of the loop $a_X$.

\begin{example}\mbox{}
	\begin{enumerate}
	\item  (Categories of $X$-Relations as Categories of Simple $(X,M)$-Graphs) The $\neg\neg$-separated presheaves corresponding to the $X$-graph theories $\oG_{X}$, $\rG_{X}$, $\sG_{X}$, and $\srG_{X}$ are the $X$-ary systems of relations $\mbf{Rel}_X$, reflexive relations $\mbf{rRel}_X$, symmetric relations $\mbf{sRel}_X$, and reflexive symmetric relations $\mbf{rsRel}_X$ respectively. By composition of $\sigma\dashv i$ with the extensions and restrictions of the essential geometric morphisms (Chapter \ref{S:ComparisonFunctors}) gives us the corresponding adjunctions between systems of relations
	\[
		\xymatrix@!=1em{ & \mbf{rRel}_X \ar@<+.3em>[dr] \ar@{<-}@<-.3em>[dr] \\ 
			 \mbf{Rel}_X \ar@<+.3em>[ur]  \ar@{<-<}@<-.3em>[ur]  \ar@{<-<}@<+.3em>[dr] \ar@<-.3em>[dr] && \mbf{rsRel}_X \\ 
			 & \mbf{sRel}_X \ar@{<-<}@<+.3em>[ur] \ar@<-.3em>[ur]}
	\]
	where the inclusions are coreflective. When $X=\{s,t\}$, we can also obtain the transitive closure of equivalence relations $\mbf{Equiv}$ of $\mbf{rsRel}_X$ by composition with the free groupoid construction in Example \ref{E:NRAdj}(\ref{E:Groupoid}) below.
	
	\item (Hypergraph Morphisms vs. Bipartite Graph Morphisms) We address the obstruction related to $r$ extending to a functor in Example \ref{E:Hybrid}(\ref{E:Bipart}) by modifying the definition of hypergraph morphisms. Let $H$ and $H'$ be hypergraphs. We define a lax morphism of hypergraphs $f\colon H\to H'$ to consist of set maps $f_V\colon H(V)\to H'(V)$ and $f_E\colon H(E)\to H'(E)$ such that for each edge $e$ in $H$, $\C P(f_V)\circ \varphi(e)\subseteq \varphi'\circ f_E$. The diagram for a lax morphism is given as follows.
	\[
	\xymatrix{ H(E) \ar[r]^{f_E} \ar[d]_\varphi \ar@{}[dr]|-{\leq} & H'(E) \ar[d]^{\varphi'} \\ \C P(H(V)) \ar[r]^{\C P(f_V)} & \C P(H'(V))}
	\]
	Verification of the associativity and identity laws are straightforward, giving us a category of lax hypergraphs, $\C H_{\text{\normalfont lax}}$. 
	
	To see that the assignment of $r$ to lax hypergraphs lifts to a functor $\rho\colon \C H_{\text{\normalfont lax}} \to \C B$, observe that the obstruction depicted in Example \ref{E:Hybrid}(\ref{E:Bipart}) above disappears since it satisfies the requirement to be a lax morphisms 
	\begin{align}
	\xymatrix{ \{e\} \ar@{}[dr]|-{\leq} \ar[r]^{\Id_e} \ar[d]_{\named v} & \{e\} \ar[d]^{\named{v_1,v_2}} \\ \C P(\{v\}) \ar[r]^-{\C P(\named{v_1})} & \C P(\{v_1,v_2\})}
	\end{align}
	Given a bipartite morphism $g\colon G\to G'$, we define the lax hypergraph morphism $\rho(g)\colon \rho(G)\to \rho(G')$ as set maps $\rho(g)_{V}\defeq g_{V_1}$ and $\rho(g)_{E}\defeq g_{V_2}$. Then given an edge $e$ in $\rho(G)$, $\C P(g_{V_1})(\varphi(e))\subseteq \varphi' (g_{V_2}(e))$. Indeed, for $w\in \C P(g_{V_1})(\varphi(e))$, there exists a $V_1$-vertex $v$ in $G$ such that $v\in \varphi(e)$ and $g_{V_1}(v)=w$, and there exists an arc $a$ in $G$ such that $\sigma(a)=g_{V_1}(v)=w$ and $\tau(a)=w$. Since $g$ is a bipartite graph morphism, $g_A(a)$ is an arc in $G'$ such that $\sigma'(g_A(a))=w$ and $\tau'(g_A(a))=g_{V_2}(e)$ and hence $w\in \varphi'(g_{V_2}(e))$. Composition and identity laws are easily verified.
	
	The category of hypergraphs is a wide subcategory of lax hypergraphs.\footnote{Recall a wide subcategory is given by a faithful functor which is bijective on objects (\cite{eR}).} Moreover, the functor $i\colon \C H\to \C B$ factors through $\mu\colon \C H_{\text{\normalfont lax}} \to \C B$.
	\[
	\xymatrix@!=.1em{ &  \C H_{\text{\normalfont lax}} \ar[dr]^{\mu} & \\ 
		\C H \ar[rr]^{i} \ar@{>->}[ur]^{\text{\tiny wide}} && \C B}
	\] 
	The functor $\mu$ is full and faithful. Observe that for the bipartite graph $\mu(H)$ there is at most one arc connected a $V_1$-vertex $v$ to a $V_2$-vertex $e$ since this is the elementhood relation $v\in \varphi(e)$. If $f, f'\colon H\to H'$ are lax hypergraph morphisms such that $\mu(f)=\mu(f')$, then $\mu(f)_{V_1}=\mu(f')_{V_1}$ and $\mu(f)_{V_2}=\mu(f')_{V_2}$ and thus it is equal on arcs $\mu(f)_A=\mu(f')_A$. Similarly, any morphism $g\colon \mu(H)\to \mu(H')$ is determined by where it sends its vertices and thus is equal to $\mu(f)$ for $f=(g_{V_1},g_{V_2})$. We thus have that the category of lax hypergraphs is equivalent to the category of simple bipartite graphs $\Sep(\DD G_{(\C F,1)},J_{\neg\neg})$. Therefore $\C H_{\text{\tiny lax}}$ is a Grothendieck quasi-topos and thus a reflective subcategory of bipartite graphs. 
	
	Moreover, the functor $\C H\hookrightarrow \C H_{\text{\tiny lax}}$ admits a left adjoint $\Sigma\colon \C H_{\text{\tiny lax}}\to \C H$ given on objects by
	  \begin{align*}
	  \Sigma(H)(V) &\defeq H(V)\\
	  \Sigma(H)(E) &\defeq \setm{(e, S)}{e\in H(E),\ S\subseteq \varphi(e)}.
	  \end{align*}
	  The incidence operator is defined $\varphi\colon \Sigma (H)(E)\to \C P(\Sigma (H)(V))$, $(e,S)\mapsto S$. For a morphism $f\colon H\to H'$ between lax hypergraphs, we define the hypergraph morphism $\Sigma(f)\colon \Sigma(H)\to \Sigma(H')$ such that 
	  \begin{align*}
	  \Sigma(f)_{V}& \defeq f_{V},\\ 
	  \Sigma(f)_{E} &\colon (e,S) \mapsto (f_E(e), f_{V!}(S))
	  \end{align*}
	  where $f_{V!}(S)$ is the image of $S$ under $f_V$. It is straightforward to verify the adjoint relations. Note that $\Sigma\colon \C H_{\text{\tiny lax}}\to \C H$ factors through the full subcategory of hereditary hypergraphs as defined in \cite{dM}.
	  \end{enumerate}
\end{example}

An $(X,M)$-graph $G$ is a $\neg\neg$-sheaf provided for each collection of vertices $(v_x)_{x\in X}$ (multiplicities allowed) there is exactly one arc $e\in G(A)$ incident with $(v_x)_{x\in X}$. It is easy to see that the $\neg\neg$-sheaves are precisely the complete $(X,M)$-graphs. In this case, the evaluation functor $(-)_V\colon \Sh(\DD G_{(X,M)}, J_{\neg\neg})\to \mathbf{Set}$ is an equivalence of categories.

\begin{example}\mbox{}
 \label{D:Automata} Recall that a non-deterministic automaton $M=(Q,\sigma, \delta,\tau)$ on a fixed set of input symbols $\Sigma$, consists of a set of states, $Q$, a set of initial states $\sigma$, a set map $\delta\colon Q\x \Sigma\to \C P(Q)$ called the transition function and a predicate $\tau$ is a set of terminal states.\footnote{Note that we drop the finite condition and fix the set of input symbols.}. A morphism $f\colon M\to M'=(Q',\sigma',\delta',\tau')$ is a set map $f\colon Q\to Q'$ such that $f(\sigma)\subseteq \sigma'$, $f(\tau)\subseteq \tau'$, and $\delta'\circ (f\x \Sigma)=f\circ \delta$ as in the following diagram.
\[
\xymatrix{ Q\x \Sigma \ar[r]^{\delta} \ar[d]_{f\x \Sigma} & Q \ar[d]^f \\ Q'\x \Sigma \ar[r]^{\delta'} & Q'}
\]
It is straightforward to verify this defines a category. This leads us to the definition of an $(S,T)$-automata.

Let $(S,T)$ be a partition of $X$, $B=\{s,\bot,\top\}$ and $\Sigma$ be a nonempty set. The category of non-deterministic $(S,T)$-automata on $\Sigma$ is defined to be the category of $\neg\neg$-separated objects in $\widehat{ds\DD G}_{(S,T)}\ls C_{(B,\Sigma)}$.

We think of the vertex labellings of an $(S,T)$-automata as neutral states $s$, initial states $\bot$, and terminal states $\top$. A directed arc is labeled by an element in the alphabet $\Sigma$ and represents a transition from a collection of  $S$ states to a collection of $T$ states (multiplicities allowed). A string $\omega$ in $\Gamma$ is said to be accepted by an $(S,T)$-automata $G$ provided there exists a sequence of arcs which agrees in labeling with $\omega$ where an initial state vertex is connected to a final state vertex.  In particular, the category of non-deterministic $(1,1)$-automata on $\Sigma$ is equivalent to the category of non-deterministic automata on $\Sigma$. The following corollary generalizes \cite{sV}(Theorem 5). Thus the category of non-deterministic $(S,T)$-automata is a Grothendieck quasi-topos.
\end{example}
\ignore{
\textbf{GRAPH MINORS!!!} Built from of quotients of $\sim$ in a reflexive category

\[
	\xymatrix{ E \ar@{>->}[r]  & G \ar@{<-<}[r] & \sim E \ar@{->>}[r] & M}
\]

When is $\sim$ functorial???

\chapter{Internal Group $(X,M)$-Graphs}

\begin{definition}
	Let $\widehat{\DD T}$ be a (reflexive) $(X,M)$-graph theory. The category of groups internal to $\widehat{\DD T}$ is the category of presheaves with values taken in the category of groups $\Group(\widehat{\DD T})\defeq [\DD T^{op},\Group]$. 
\end{definition}

\begin{proposition}
	Let $G$ be an internal $(X,M)$-graph group. 
	\begin{ienum}
		\item The discrete subgraph of vertices $G(V)$ is a normal subgroup of $G$.
	\end{ienum}
\end{proposition}

\begin{proposition}
	Let $M$ be a group, $X$ a transitive right $M$-set, and $G$ be an internal $(X,M)$-graph group.
	\begin{ienum}
		\item The connected component $\langle e_V\rangle$ of the unit vertex $e_V\in G(V)$ is a normal subgroup of $G$.
		\item For each family of vertices $(v_x)_{x\in X}$ in $G$, either $\overline G((v_x)_X)$ is empty or $\overline G((v_x)_X)\iso \overline G((e_V)_X)$.
	\end{ienum}
\end{proposition}

\tbf{REFLEXIVE GRAPHS!!!!!}

\section{Automorphism $(X,M)$-Graphs}

By \cite{}(\tbf{KELLY}) the endomorphisms object $G^G$ has the structure of an internal monoid where the monoid unit $e\colon 1\to G^G$ is the exponential adjoint to the identity $\Id_G\colon G\to G$ and the monoid operation given by internal composition $c_{\text{\tiny GGG}}\colon G^G\x G^G\to G$. Since categories of $(X,M)$-graphs are categories of presheaves, an object $H$ is internal monoid if and only if each component is monoid interpreted in the category of sets. For the endomorphism object $G^G$, we see that on vertices $G^G(V)=G(V)^{G(V)}$ is the monoid of endomaps on maps between vertices. On arcs, the monoid $G^G(A)$ has unit $((\Id_{G(V)})_{x\in X},\Id_{G(A)})$ and operation $\star\colon G^G(A)\x G^G(A)\to G^G(A)$ which takes $(((f'_x)_{x\in X},g'),((f_x)_{x\in X},g))$ to $((f_x'f_x)_{x\in X},g'g)$ as given above. 

The maximal submonoid object of $G^G$ which is a group is called the automorphism object of $G$ and denoted $\Aut(G)$. For the vertex set, $\Aut(G)$ consists of the set of automaps $G(V)\to G(V)$. The arc set of $\Aut(G)$ consists of those arcs $((f_x)_{x\in X},g)$ in $G^G$ such that $f(x,-)\colon G(V)\to G(V)$ for each $x\in X$ and $g\colon G(A)\to G(A)$ are bijective. 

\begin{example}\mbox{}
	\begin{enumerate}
		\item \tbf{FINISH!!!!!!!} Let $G$ be a group ...$\gamma^*$ preserved products and thus group objects. Therefore $\gamma^*(\Aut(\underline A))$ 
	\end{enumerate}
	
\end{example}	
}

\chapter{Injective Hulls and Projective Covers}\label{S:InjProj}

 Recall the morphism of (reflexive) $(X,M)$-graph toposes $\iota_!\dashv \iota^*\dashv \iota_*$ in Examples \ref{E:FFC}(\ref{E:FreeForget}, \ref{E:FreeForget2}).   We set $\Proj\defeq \iota_!\iota^*$ and $\Inj\defeq \iota_*\iota^*$ where $\iota\colon \DD G_{(\empset,X)}\to \DD G_{(X,M)}$ is the functor given above. Then since adjunctions are closed under composition, we have
\[
\Proj \dashv \Inj\colon \widehat{\DD G}_{(X,M)}\to \widehat{\DD G}_{(X,M)}.
\]
We will show in this section that the natural transformations $\varepsilon\colon \Proj \Rightarrow \Id$ and $\eta\colon \Id \Rightarrow \Inj$ can be thought of as the functorial projective and injective refinements for non-initial  $(X,M)$-graphs.

We first characterize the class of injective and projective objects. Recall that an object $Q$ in a category $\C C$ is injective provided for each monomorphism $m\colon A\to B$ and morphism $f\colon A\to Q$ there exists a morphism (not necessarily unique) $k\colon B\to Q$ such that $f=km$. Dually, an object $P$ in $\C C$ is (regular) projective\footnote{Note that since regular epimorphisms are equivalent to epimorphisms in categories of presheaves, a regular projective object is equivalent to a projective object.} provided for each (regular) epimorphism $e\colon B\to A$ and each morphism $f\colon P\to A$ there is a morphism $k\colon P\to B$ such that $f=ek$.\footnote{The results in this section generalize the results of \cite{wG} and \cite{kW}.}
\begin{proposition}\label{P:Injective}
	A  (reflexive) $(X,M)$-graph $Q$ is injective if and only if $Q$ is non-initial and for each set map $f\colon X\to Q(V)$, there is an arc $\alpha\in Q(A)$ such that the incidence map $\partial_Q(\alpha)$ is equal to $f$. 
\end{proposition}
\begin{proof}
	Suppose $Q$ is injective and consider the set map $f\colon X\to Q(V)$. This is equivalent to giving an $(X,M)$-graph morphism $\overline f\colon \bigsqcup_{x\in X}\underline V\to Q$. Consider the inclusion $m\colon \bigsqcup_{x\in X}\underline V\to \underline A$ induced by the morphisms $\underline x\colon \underline V\to \underline A$. Since $Q$ is injective, there is a morphism $\alpha\colon \underline A\to I$ such that  $\alpha m=\overline f$. By Yoneda, this is equivalent to an arc $\alpha\in I(A)$ with incidence map $\partial_I(\alpha)=f$.
	
	Conversely, let $f\colon G\hookrightarrow H$ be a monomorphism and $g\colon G\to Q$ a morphism of  $(X,M)$-graphs. Since $Q$ is non-initial, there is a vertex $v\in Q(V)$. Each arc $\alpha$ in $H$ has incidence $\partial_H(\alpha)\colon X\to H(V)\iso f_V(G(V))\sqcup H(V)\backslash f_V(G(V)$ where $f_V(G(V))$ is the image of the vertices in $G$ under $f$. For each arc $\alpha$ in $H$ not in the image of $f_A$, let $j_\alpha\colon X\to Q(V)$ be the set map $[g_V,!]\circ \partial_G(\alpha)$ given by universal property of the disjoint union
	\[
	\xymatrix{ && f_V(G(V))\iso \inv f_V(G(V)) \ar@{>->}[d] \ar[rr]^-{g_V} && g_V(G(V)) \ar@{>->}[d] \\ X \ar[rr]^-{\partial_G(\alpha)} && f_V(G(V))\sqcup H(V)\backslash f_V(G(V)) \ar[rr]^-{[g_V,!]} && Q(V) \\ 
		&& H(V)\backslash f_V(G(V)) \ar@{>->}[u] \ar[rr]^-{!} && \{v\} \ar@{>->}[u]}
	\] 
	Thus by assumption, we may choose an arc $[\alpha]\in Q(A)$ with incidence equal to $j_\alpha$.  We define the following maps $h_V\colon H(V)\to Q(V)$ and $h_A\colon H(A)\to Q(A)$
	\begin{align*}
	h_V(w)&\defeq \begin{cases} g_V(u) & \text{if }\exists u\in G(V), f_V(u)=w \\ v & \text{if }\forall u\in G(V), f_V(u)\neq w\end{cases}\\
	h_A(\alpha)&\defeq \begin{cases} g_A(\beta) & \text{if }\exists \beta\in G(A), f_A(\beta)=\alpha \\ [\alpha] & \text{if }\forall \beta\in G(A), f_A(\beta)\neq \alpha \end{cases}
	\end{align*}
	By construction this defines a morphism $h\colon H\to Q$ such that $h\circ f=g$. Therefore $Q$ is injective.
	
\end{proof}

\begin{corollary}
	The class of injective objects in $\widehat{\DD G}_{(X,M)}$  ($\widehat{\rG}_{(X,M)}$) is precisely the class of non-initial split subobjects\footnote{An object $H$ is a split subobject of $G$ provided it admits a split monomorphism $s\colon H\to G$.} of objects in the essential image of the functor $\Inj$.
\end{corollary}
\begin{proof}
	Let $Q$ be an injective object in $\widehat{\DD G}_{(X,M)}$.  Hence $Q$ is non-initial and thus by the previous lemma, we have $\Inj(Q)$ is an injective object. Then since $\eta_Q\colon Q \to \Inj(Q)$ is a monomorphism, there must be a split epimorphism $r\colon \Inj(Q)\to Q$ such that $r\eta_Q=\Id$ by the property of $Q$ being injective.  
	
\end{proof}

We also have the dual argument that the class of projective objects in the category of  $(X,M)$-graphs is precisely the split quotients of objects in the essential image of $\Proj$. 

\begin{proposition}
	A (reflexive) $(X,M)$-graph $P$ is projective if and only if it is a coproduct of representables. 
\end{proposition}
\begin{proof}
	Suppose $P$ is projective. Since $\varepsilon_P\colon \Proj(P)\to P$ is an epimorphism, there must exist a section $s\colon P\to \Proj(P)$ such that $\varepsilon_Ps=\Id$ by the property that $P$ is projective. Since $\Proj(P)$ is a coproduct of representables and $s$ is a split monomorphism, $P$ must also be a coproduct of representables. Conversely, representables $\underline V$ and $\underline A$ in a category of presheaves are always projective. Since projective objects are closed under coproducts, the reverse condition is also true.
	
\end{proof}
\begin{corollary}
	The class of projective objects in $\widehat{\DD G}_{(X,M)}$ (resp. $\widehat{\rG}_{(X,M)}$) is the object class of the essential image of the functor $\Proj$. 
\end{corollary}
\begin{proof}
	Given a coproduct of representables $\bigsqcup_S \underline V\sqcup \bigsqcup_T \underline A$, let $H$ be the $(X,M)$-graph with vertex set $H(V)\defeq S$ and arc set $H(A)=T$. Take some $s\in S$ and define right-actions $t.x\defeq s$ and $t.m=t$ for each $t\in T$, $x\in X$ and $m\in M$. Then $\Proj(H)\iso \bigsqcup_S \underline V\sqcup \bigsqcup_{T}\underline A$. 
\end{proof}

Next, we construct injective hulls and projective covers for  $(X,M)$-graphs. Recall that a monomorphism $i\colon G\to \widetilde G$ is essential provided for each morphism $h\colon G\to H$ such that $hi$ is a monomorphism implies $h$ is a monomorphism. An injective hull of an object $G$ is an essential monomorphism $i\colon G\to \widetilde G$ where $\widetilde G$ is injective. Dually, an epimorphism $e\colon \overline G\to G$ is essential provided for each morphism $h\colon H\to G$ such that $eh$ is an epimorphism implies $h$ is an epimorphism. A projective cover of an object $G$ is an essential epimorphism $e\colon \overline G\to G$ where $\overline G$ is projective.

In the case of the initial  $(X,M)$-graph $0$, it is straightforward to verify the terminal morphism $0\to 1$ is the injective hull. For a non-initial  $(X,M)$-graph $G$ we define $\widetilde G$ to be the $(X,M)$-graph with vertex set $\widetilde G(V)\defeq G(V)$ and arcs set $\widetilde G(A)\defeq G(A)\sqcup \setm{f\colon X\to G(V)}{\forall \alpha\in G(A),\ \partial_G(\alpha)\neq f}$ with the obvious right-action. Then $\widetilde G$ is an injective object and there is an obvious inclusion $i\colon G\to \widetilde G$. To show that it is essential, let $h\colon \widetilde G\to H$ be a morphism such that $hi$ is a monomorphism. Then since $i$ is bijective on vertices, $h_V$ must be injective. On arcs, it is enough to show that $h_A$ is injective on $\widetilde G(A)\backslash G(A)$. However, this is trivial since there is only one arc $f\in \widetilde G(A)\backslash G(A)$ with incidence $f\colon X\to G(V)$. Hence $h$ is a monomorphism.

For the projective cover, we define $\overline G\defeq \bigsqcup_S \underline V \sqcup \bigsqcup_T \underline A$ where \\$S\defeq \setm{v\in G(V)}{\forall \alpha\in G(A),\forall x\in X,\ \alpha.x\neq v}$, i.e., $S$ is the set of isolated vertices in $G$, and $T$ is a generating subset of $G(A)$ for the right $M$-action $G(A)\x M\to G(A)$ of minimal cardinality, i.e., for each $\alpha\in G(A)$ there exists a $\beta\in T$ and an element $m\in M$ such that $\beta.m=\alpha$.  Since $T$ generates $G(A)$ under the right-action of $M$ and $S$ is the set of vertices of $G$ which are not incident to an arc, the restriction of $\varepsilon_{G|\overline G}\colon \overline G\to G$ is an epimorphism. It is clear that if $h\colon H\to \overline G$ is a morphism such that $he$ is an epimorphism, then $h$ must be an epimorphism since $\overline G$ is a coproduct of representables of minimal size. 
\[
\xymatrix@!=.3em{ & \widetilde G \ar@{>->}[dr] & \\ G \ar@{>->}[ur] \ar@{>->}[rr]^-{\eta_G} && \Inj(G)} \qquad \qquad
\xymatrix@!=.3em{ & \overline G \ar@{>->}[dl] \ar@{->>}[dr] & \\ \Proj(G) \ar@{->>}[rr]^-{\varepsilon_G} && G }  
\]
Note that these assignments $\widetilde G$ and $\overline G$ do not extend to functors since there is choice involved. However, by construction we see that both the injective hull $\widetilde G$ and projective cover $\overline G$ embed into $\Inj(G)$ and $\Proj(G)$ which are functorial constructions.

\part{Models of $(X,M)$-Graphs}\label{S:MainResults}

\chapter{The Nerve-Realization Adjunction}\label{S:Category}
We use functorial semantics to model (reflexive) $(X,M)$-graphs in various categories. In particular, we focus on connecting the ideas of this paper to categories of hypergraphs, $k$-uniform hypergraphs, and directed and undirected graphs.

The symbols and notation in this section follow from \cite{eR}. 

Let $I\colon \DD T \to \C M$ be functor from a small category $\DD T$ to a cocomplete category $\C M$. Since the Yoneda embedding $y\colon \C T\to \widehat{\DD T}$ is the free cocompletion of a small category there is a essentially unique adjunction $R\dashv N\colon \C M\to \widehat{\DD T}$, called the nerve realization adjunction, such that $Ry\iso I$. 
\[
\xymatrix{ \DD T \ar[dr]_I \ar[r]^{y} & \widehat{\DD T} \ar@<-.4em>[d]_{R}^{\dashv} \ar@{<-}@<+.5em>[d]^{N}\\ & \C M }
\]
The nerve and realization functors are given on objects by 
\begin{align*}
	N(m)&=\C M(I(-), m),\\ R(G)&=\colim_{(c,\varphi)\in\int G} I(c)
\end{align*} respectively,
where $\int G$ is the category of elements of $G$ (\cite{hA}, Section 2, pp 124-126).\footnote{In \cite{hA}, the nerve functor is called the singular functor.} 

We call a functor $I\colon \DD T \to \C M$ from a small category to a cocomplete category an interpretation functor. The category $\DD T$ is called the theory for $I$ and $\C M$ the modeling category for $I$. An interpretation $I\colon \DD T\to \C M$ is dense, i.e., for each $\C M$-object $m$ is isomorphic to the colimit of the diagram
$I\ls m \to \C M,\ (c,\varphi)\mapsto I(c),$
if and only if the nerve $N\colon \C M\to \widehat{\DD T}$ is full and faithful (\cite{sM}, Section X.6, p 245). When the right adjoint (resp. left adjoint) is full and faithful we call the adjunction reflective (resp. coreflective).\footnote{since it implies $\C M$ is equivalent to a reflective (resp. coreflective) subcategory of $\widehat{\DD T}$}

\begin{example}\label{E:NRAdj}\mbox{}
	\begin{enumerate}
		\item (Free Category Functor) Let $\Cat$ be the category of small categories, $X=\{s,t\}$ and consider the interpretation $I\colon \rG_X\to \Cat$ where $I(V)$ is the terminal category $\mbf 1$, $I(A)$ is the walking arrow category $\mbf{2}$ with two objects and one non-identity morphism, $I(s)$ and $I(t)$ are the separate inclusions and $I(\ell)$ is the terminal functor. Then the nerve realization adjunction $R_{C}\dashv N_{C}\colon \mbf{Cat}\to \widehat{\rG}_X$ is the free forgetful adjunction given in \cite{fB} (I, Chapter 5). In other words, $R_C$ takes a reflexive $X$-graph $G$ to the category with object set $G(V)$, morphism set the finite paths of arcs in $G(A)$, e.g., a path $(\alpha_1,\dots, \alpha_n)$ is a morphism in $R_C(G)$ with domain $\alpha_1.s$ and codomain $\alpha_n.t$, the identity morphisms the distinguished loops, and composition given by concatenation of paths. 
		
		\item \label{E:Groupoid}  (Free Groupoid Functor) Let $\mbf{Grpd}$ be the category of small groupoids, $X=\{s,t\}$ and consider the interpretation $I\colon \srG_X\to \mbf{Grpd}$ where $I(V)$ is the terminal groupoid $\mbf 1$, $I(A)$ is the smallest connected groupoid on two objects $\mbf 2$, $I(s)$ and $I(t)$ are the inclusion of the objects in $\mbf 2$ and $I(\ell)$ is the terminal functor. Then $I$ induces the nerve realization adjunction $R_G\dashv N_G\colon \mbf{Grpd}\to \widehat{\srG}_X$ such that $R_G$ is the path construction as given above and $N_G$ is the underlying reflexive symmetric $X$-graph of a groupoid. 
		\item (Fundamental Groupoid Functor) Consider the category $\DD T$ generated by the graph 
		\[
			\xymatrix{ V \ar@<+.4em>[r]^s \ar[r]|-t \ar@{<-}@<-.4em>[r]_\ell & A_1 \ar@(ul,ur)[]^i \ar@<+.3em>[r]^\sigma \ar@<-.3em>[r]_{\tau} & A_2 }
		\]
		and relations $\sigma\circ s=\tau\circ s$, $\sigma\circ t=\tau\circ t$, $\ell\circ s=\Id_V$, $\ell\circ t=\Id_V$, $i\circ s=t$, $i\circ i=\Id_{A_1}$. Let $\mbf{Sp}$ be the category of topological spaces. Then there is an interpretation functor $I\colon \DD T\to \mbf{Sp}$ given by $I(V)=1$ (the terminal space), $I(A_1)=[0,1]$ (the unit interval), $I(A_2)\defeq D^2\defeq \setm{r\in \DD C}{|r|\leq 1}$ (the 2-dimensional disk) on objects. On morphisms we define $I(s)\defeq \named 0$ and $I(t)\defeq \named 1$ (inclusion of the endpoints), $I(\ell)$ the terminal morphism, $I(\sigma)\defeq e^{\pi i(1-x)}\colon [0,1]\to D^2$ and $I(\tau)\defeq e^{\pi i(1+x)}\colon [0,1]\to D^2$ (the inclusion of the top and bottom of the disc) and $I(i)\colon [0,1]\to [0,1]$, $x\mapsto 1-x$. This induces a nerve realization adjunction $R_{Sp}\dashv N_{Sp}\colon \mbf{Sp}\to \widehat{\DD T}$ such that for a topological space $Z$, $N(Z)$ has vertex set equal to the underlying set of $Z$, $A_1$-arc set equal to paths in $Z$ and $A_2$-arc set equal to morphisms of the disc $D^2$. Consider the functor $\beta\colon \DD T\to \srG_X$ ($X=\{s,t\}$) which takes $V$ to $V$, $A_1$ to $A$, $s,t,\ell,i$ to $s,t,\ell, i$, $A_2$ to $A_1$ and all morphisms from $A_1$ to $A_2$ to the identity on $A_1$. Then $\beta$ induces an essential geometric morphism $\beta_!\dashv \beta^*\dashv \beta_*\colon \widehat{\DD T}\to \widehat{\srG}_X$.  The $\beta$-extension $\beta_!$ takes a $\widehat{\DD T}$-object $G$ to the reflexive symmetric $X$-graph with vertex set $\beta_!G(V)=G(V)$ and arc set $\beta_!G(A)=\frac{G(A)}{\sim}$ where $\alpha_1\sim \alpha_2$ if there exists $a\in G(A_2)$ such that $\alpha_1=a.\sigma$ and $\alpha_2=a.\tau$. The $\beta$-restriction $\beta^*$ takes a reflexive symmetric $X$-graph $H$ to the $\widehat{\DD T}$-object with $\beta^*H(V)=H(V)$, $\beta^*H(A_1)=H(A)$ and $\beta^*H(A_2)=H(A)$ such that for each $A_2$-arc $\alpha$ we have $\alpha.\sigma=\alpha.\tau=\alpha$. The $\beta$-coextension $\beta_*$ takes the $\widehat{\DD T}$-object $G$ and forgets the $A_2$-arcs. 

		The reflexive symmetric $X$-graph $\beta_!N_{Sp}(Z)$ captures the set of homotopy classes of paths in $Z$, i.e., for any two points $z_1,z_2\in Z$, the set $\setm{\alpha\in \beta_!N_{Sp}(Z)}{\alpha.s=z_1, \alpha.t=z_2}$ is the set of paths up to homotopy between $z_1$ and $z_2$. The reflexive symmetric $X$-graph $\beta_!N_{Sp}(Z)$ naturally has the structure of a groupoid, where composition is given by the usual concatenation of paths (up to homotopy) and reparametrization of the interval. Thus $\beta_!N_{Sp}(Z)$ is the underlying reflexive symmetric $X$-graph of the fundamental groupoid \cite{rBTG} (Chapter 6). Similarly, the reflexive symmetric $X$-graph $\beta_*N_{Sp}(Z)$ is the underlying reflexive symmetric $X$-graph of the path groupoid. In other words, the following diagrams commute
		\[
			\xymatrix{ & \mbf{Grpd} \ar[dr]^{N_G} & \\ \mbf{Sp} \ar[ur]^{P} \ar[rr]^-{\beta_*N_{Sp}} && \widehat{\srG}_X }\qquad
			\xymatrix{ & \mbf{Grpd} \ar[dr]^{N_G} & \\ \mbf{Sp} \ar[ur]^{\Pi_1} \ar[rr]^-{\beta_!N_{Sp}} && \widehat{\srG}_X }
		\]
		where $\Pi_1$ is the fundamental groupoid functor and $P$ is the path groupoid functor. 
	Therefore, the proof that $\Pi_1$ preserves products and coproducts reduces to showing $\beta_!$ preserves products and coproducts which is left to the reader. 

		\item (Quotient Graph Construction) Let $I\colon \DD G_{(X,M)}\to \widehat{\DD G}_{(X,M)}$ be the interpretation functor which on objects takes $I(V)\defeq \underline A$ and $I(A)\defeq \bigsqcup_{x\in X}\underline A$. On morphisms, we define $I(x)$ to be the coproduct inclusion at index $x\in X$ and $I(m)$ to be induced by the coproduct
		\[\textstyle
			\xymatrix{ \underline A \ar[d]_{I(x)} \ar[r]^{\Id} & \underline A \ar[d]^{I(x.m)}\\ \bigsqcup_{x\in X} \underline A \ar[r]^-{I(m)} & \bigsqcup_{x\in X} \underline A}
		\]
		for each $m\in M$. Thus $I$ induces a nerve realization adjunction $R\dashv N\colon \widehat{\DD G}_{(X,M)}\to \widehat{\DD G}_{(X,M)}$ such that for each $(X,M)$-graph $G$, $N(G)(V)=G(A)$, $N(G)(A)=\prod_{x\in X}G(A)$ and $\vec \alpha.x=\alpha_x$ and $\vec\alpha.m=(\alpha_{x.m})_{x\in X}$ for each $x\in X$, $m\in M$  and arc $\vec \alpha\defeq (\alpha_x)_{x\in X}$. In the case $X=\{s,t\}$, $R\dashv N$ is adjunction of the arc graph and quotient graph construction in \cite{jF1}. 
		\item \label{E:Pultr} (Pultr Functors) If $I\colon \DD G_{(X,M)}\to \widehat{\DD G}_{(X,M)}$ is an interpretation functor, then there are $(X,M)$-graphs $G=I(V)$ and $H=I(A)$ and a family of morphism $(I(x)=f_x\colon G\to H)_{x\in X}$ and endomorphisms $(I(m)=f_m\colon H\to H)_{m\in M}$ such that $f_m\circ f_x=f_{x.m}$ for each $x\in X$ and $m\in M$. Then $I$ induces a nerve realization adjunction $R\dashv N\colon \widehat{\DD G}_{(X,M)}\to \widehat{\DD G}_{(X,M)}$. 
	
		Let $X=\{s,t\}$ and $f,g\colon G\to H$ be $\widehat{\oG}_X$-morphisms. Then there is an interpretation functor $I\colon \oG_X\to \widehat{\oG}_X$ given by $I(V)=G$, $I(A)=H$, $I(s)=f$ and $I(t)=g$ which induces the adjunction $R\dashv N\colon \widehat{\oG}_X\to \widehat{\oG}_X$. Recall that there is a functor $\tau\colon \widehat{\oG}_X\to \Sk(\widehat{\oG}_X)$ where $\Sk(\widehat{\oG}_X)$ is the associated skeletal (a.k.a. thin) category of $\widehat{\oG}_X$. The functor $\tau$ respects the adjunction $\Sk(R)\dashv \Sk(N)\colon \Sk(\widehat{\oG}_X)\to \Sk(\widehat{\oG}_X)$ where $\Sk(N)\defeq N\circ \tau$ and $\Sk(R)\defeq R\circ \tau$. These are the left and central Pultr functors for the Pultr template $(G,H, f, g)$ given in \cite{aP}\cite{jF}\cite{jFcT}. Indeed, given an oriented $X$-graph $K$, \begin{align*}
		N(K)(V)&=\widehat{\oG}_X(I(V),K)=\widehat{\oG}_X(G,K),\\
		N(K)(A)&=\widehat{\oG}_X(I(A),K)=\widehat{\oG}_X(H,K)
		\end{align*}
		where for each $h\in N(K)(A)$, $h.s=h\circ f$ and $h.t=h\circ g$. 
	\end{enumerate}
\end{example}

In the subsequent, we are interested in when the nerve also preserves any exponentials which exist. For the purpose of this paper, we show that if an interpretation is dense, full and faithful, then the nerve not only preserves limits, but also any exponentials which exist.

\begin{lemma}\label{L:FFInt}
	An interpretation functor $I\colon \DD T \to \C M$ is full and faithful iff $\underline c\defeq y(c)$ is a $NR$-closed object for each $\DD T$-object $c$, i.e., the unit $\eta_{\underline c}\colon \underline c\to NR(\underline c)$ at component $\underline c$ is an isomorphism.
\end{lemma}
\begin{proof}
	The unit of the adjunction $\eta_G$ is defined as the following composition
	\[
	\xymatrix{ G \ar[r]^-\varphi_-{\iso} & \widehat{\DD T}(y(-), G) \ar[r]^-{R_{(y,G)}} & \C M(Ry(-),R(G)) \ar[r]^-\psi_-{\iso} & \C M(I(-), R(G))=NR(G)}, 
	\]
	where $\varphi$ is given by Yoneda, $R_{(y,G)}$ is the map of homsets given by application of $R$, and $\psi$ is precomposition by the isomorphism $I\iso Ry$. For a representable, $\underline c$, there is an isomorphism $\rho\colon \C M(I(-), R(\underline c))\to \C M(I(-),I(c))$ by postcomposition by the isomorphism $I\iso Ry$. Thus $\rho\circ \psi\circ R_{(y,G)}$ evaluated at $\DD T$-object $c'$ takes a $\DD T$-morphism $f\colon c'\to c$ to $I(f)\colon I(c')\to I(c)$. Thus $I$ is full and faithful iff $\eta_{\underline c}$ is an isomorphism. 
\end{proof}

\begin{proposition}\label{P:Exponential}
	If an interpretation functor $I\colon \DD T\to \C M$ is dense, full and faithful, then  $R\dashv N$ is reflective and $N$ preserves any exponentials that exist in $\C M$. 
\end{proposition}
\begin{proof}
	Suppose $G$ and $H$ are $\C M$-objects such that the exponential $G^H$ exists in $\C M$. Since $I$ is assumed to be full and faithful, by Lemma \ref{L:FFInt} above,  $\underline c\iso NR(\underline c)$ for each $\DD T$-object. Thus we have the following string of natural isomorphism:
	\begin{align*}
	N(G^H)(c) &\iso \C M(R(\underline c)\x H, G) && \text{(Yoneda, $R\dashv N$, exponential adj.})\\
	&\iso\widehat{\DD T}(NR(\underline c)\x N(H), N(G)) && \text{($N$ is full, faithful, preserves limits)}\\
	&\iso \widehat{\DD T}(\underline c\x N(H), N(G)) &&\text{($\underline c$ is $NR$-closed)}\\
	&\iso N(G)^{N(H)}(c) && \text{(Exponential adj. and Yoneda)}.
	\end{align*}
	Since the right-action structures are determined by Yoneda,  $N(G^H)\iso N(G)^{N(H)}$ in $\widehat{\DD T}$.
\end{proof}

The nerve of a dense, full and faithful interpretation is a presheaf topos completion of a modeling category as the following result shows. 

\begin{proposition}\label{P:PresheafCompletion}\text{\normalfont (The Presheaf Completion of a Modeling Category).}\\
	Let $I\colon \DD T\to \C M$ be a dense, full and faithful interpretation functor with nerve-realization adjunction $R\dashv N\colon \C M \to \widehat{\DD T}$. Then for each small category $\DD S$ and functor $F\colon \C M\to \widehat{\DD S}$, there is an essentially unique essential geometric morphism $k_!\dashv k^*\dashv k_*\colon \widehat{\DD T}\to \widehat{\DD S}$ such that $F\iso k_!N$. 
	\[
	\xymatrix{ \C M \ar[rr]^{N} \ar[drr]_{F} && \widehat{\DD T} \ar@<-1em>[d]_{k_!}^\dashv \ar@{<-}[d]|->>>>{k^*}^\dashv \ar@<1em>[d]^{k_*}
		\\ && \widehat{\DD S}.}
	\]
	Moreover, if $F$ is full and faithful, $k_!\dashv k^*$ is a coreflective adjunction and $k^*\dashv k_*$ is a reflective adjunction.
\end{proposition}
\begin{proof}
	Define $k_!\defeq FR\colon \widehat{\DD T}\to \widehat{\DD S}$. Then $k_!y\colon \DD T\to \widehat{\DD S}$ is a NR-context inducing a left adjoint $k^*\colon \widehat{\DD S}\to \widehat{\DD T}$ as the nerve to $k_!y$. This in turn induces $k_*\colon \widehat{\DD T}\to \widehat{\DD S}$ as the left adjoint to $k^*$ by the NR-context $k^*y\colon \DD S \to \widehat{\DD T}$. Therefore, there exists an essential geometric morphism $k_!\dashv k^*\dashv k_*\colon \widehat{\DD T}\to \widehat{\DD S}$. For uniqueness, suppose $\overline k_!\dashv \overline k^*\dashv \overline k_*$ was another essential geometric morphism such that $F\iso \overline k_!N$. Then $\overline k_!y \iso \overline k_! NRy 
	\iso FRy 
	= k_!y$ since $N$ is full and faithful. Then by the essential uniqueness of the realization functor, $k_!=\overline k_!$. 
	
	Next, assume $F$ is full and faithful.  Since $N$ is full and faithful, by the isomorphism $k_!N\iso F$, $k_!$ is full and faithful. To prove $k_*$ is also full and faithful, it is enough to show that there is an isomorphism $k^*k_*(G)\to G$ for each $\widehat{\DD T}$-object $G$. Indeed,  
	\begin{align*} 
	k^*k_*(G)&=\widehat{\DD T}(k_!y(-),k_*(G)) && \text{(Definitions of $k^*$ and $N_{k_!y}$)}\\
	& \iso \widehat{\DD T}(k^*k_!y(-), G), && \text{($k^* \dashv k_*$ adjunction)} \\
	&\iso \widehat{\DD T}(y(-), G), &&\text{($k_!$ is full and faithful)} \\
	&\iso G 	&& \text{(Yoneda)}
	\end{align*}
	Therefore, $k_!\dashv k^*$ is coreflective and $k^*\dashv k_*$ is reflective. 	
\end{proof}

Notice that we did not use the assumption that the interpretation functor was full and faithful in the proof above. However, we believe that for a functor to be called a "presheaf completion", it should preserve  exponentials. Thus we included it in the assumption.\footnote{Also we should be careful not to call this the (Grothendieck) topos completion of a category since the nerve may properly factor through a category of sheaves on a site.}

\ignore{
We also recall the following results concerning injective and projective objects.

\begin{proposition}
	Let $\C A$ be a cartesian closed category. 
	\begin{ienum}
		\item If $Q$ is an injective object in $\C A$, then for each object $G$, $Q^G$ is an injective object.
		\item If $P$ is a projective object in $\C A$, then for each object $G$ in $\C A$, $P\x G$ is projective.
	\end{ienum}
\end{proposition}
\begin{proof}
	
\end{proof}

\begin{proposition}
	Let $R\dashv N\colon \C A\to \C B$ be a reflective adjunction.
	\begin{ienum}
		\item If $R$ preserves monomorphisms and $Q$ is injective in $\C A$, then $Q$ is isomorphic to $R(Q')$ for some injective $Q'$ in $\C B$.
		\item If $N$ preserves epimorphisms, if $P$ is projective in $\C B$, then $P$ isomorphic to $N(P')$ for some projective $P'$ in $\C A$.
	\end{ienum}
\end{proposition}
\begin{proof}
	
\end{proof}

}

\chapter{Categories of F-Graphs}\label{C:FGraphs}
In this section, we model $(X,M)$-graphs in frameworks of more conventional categories of graphs and hypergraphs. The first is given by \cite{cJ} as the category of $F$-graphs where vertices are distinct parts of a hypergraph. In the case when vertices are degenerate edges, we introduce the category of reflexive $F$-graphs, which extends the definition of the category of conceptual graphs given in \cite{dP} to a reflexive counterpart to $F$-graphs.

\section{Interpretations in $F$-graphs} \label{S:Interp}

We follow the definition given in \cite{cJ}. 

\begin{definition}
	Let $F\colon \mathbf{Set}\to \mathbf{Set}$ be an endofunctor. The category of $F$-graphs $\C G_F$ is defined to be the comma category $\C G_F\defeq \mathbf{Set}\ls F$.
\end{definition}
In other words, an $F$-graph $G=(G(E),G(V),\partial_G)$ consists of a set of edges $G(E)$, a set of vertices $G(V)$ and an incidence map $\partial_G\colon G(E)\to F(G(V))$. A morphism \[(f_E,f_V)\colon (G(E),G(V),\partial_G)\to (H(E),H(V),\partial_H)\] is a pair of set maps $f_E\colon G(E) \to H(E)$ and $f_V\colon G(V)\to H(V)$ such that the following square commutes
\[
\xymatrix{ G(E) \ar[d]_{\partial_G} \ar[r]^{f_E} & H(E) \ar[d]^{\partial_H} \\ F(G(V)) \ar[r]^{F(f_V)} & F(H(V)). }
\]
It is well-known that the category of $F$-graphs is cocomplete with the forgetful functor $U\colon \C G_F\to \mathbf{Set}\x \mathbf{Set}$ creating colimits \cite{cJ}.

Let $\DD G_{(X,M)}$ be a theory for $(X,M)$-graphs and $q$ an element in $F(X)$ such that $F(m)(q)=q$ for each $m\in M$ where $m\colon X\to X$ is the right-action map. We define $I(V)\defeq (\empset,\ 1,\ !_1),$ and $I(A)\defeq (1,\ X,\ \named q)$ where $!_1\colon \empset \to 1$ is the initial map and $\named x\colon 1\to X$ the set map with evaluation at $x\in X$. On morphisms, we set
\begin{align*} 
(x\colon V\to A) \quad &\mapsto \quad  I(x)\, \defeq \ (!_1,\named x)\colon (\empset,\ 1,\ !_1)\to (1,\ X,\ \named q),\\
(m\colon A\to A) \quad &\mapsto \quad I(m)\defeq (\Id_1, F(m))\colon (1,X,\named q)\to (1,X,\named q).
\end{align*}
Verification that $I\colon \DD G_{(X,M)}\to \C G_F$ is a well-defined interpretation functor is straightforward. 

\section{Interpretations in Reflexive $F$-Graphs}\label{S:Interpr}

For categories of graphs with vertices as degenerate edges, we generalize the definition of conceptual graphs in \cite{dP} (Definition 2.1.1, p 16).

\begin{definition}
	Let $F\colon \mathbf{Set} \to \mathbf{Set}$ be an functor and $\eta\colon \Id_{\mathbf{Set}}\Rightarrow F$ a natural transformation. The category of reflexive $F$-graphs $\crG_F$
	has objects $G=(G(P),G(V),\partial_G)$ where $G(P)$ is a set, $G(V)\subseteq G(P)$ is a subset and $\partial_G\colon G(P)\to F(G(V))$ is a set map. An $F$-graph morphism $f\colon G\to H$ consists of a set map $f_P\colon G(P)\to H(P)$ such that the following commutes
	\[
	\xymatrix@R=1.5em@C=1.5em{ &  G(V) \ar@{>->}[dl] \ar[dd]|-\hole^<<<<{\eta}  \ar[r]^{f_V} & H(V) \ar[dd]|-\hole_<<<<{\eta}  \ar@{>->}[dr] & \\ G(P) \ar[dr]_{\partial_G} \ar[rrr]^{f_P} &&& H(P) \ar[dl]^{\partial_H} \\ & F(G(V)) \ar[r]^{F(f_V)} & F(H(V))}
	\]
	where $f_V$ is the set map $f_P$ restricted to $G(V)$.\footnote{By naturality $\eta\colon \Id_{\mathbf{Set}}\Rightarrow F$ the middle square always commutes.}
\end{definition}
In other words, a reflexive $F$-graph $G$ consists of parts $G(P)$ with a subset of vertices $G(V)$ and an incidence operation $\partial_G\colon G(P)\to F(G(V))$ which considers a vertex $v$ to be a degenerate edge in the sense that $\partial_{G|G(V)}=\eta$. A reflexive $F$-graph morphism $f\colon G\to H$ that maps an edge to a vertex is one where $e\in G(P)\backslash G(V)$ has $f_P(e)\in H(V)$. 

The category of reflexive $F$-graphs is cocomplete. Indeed, the empty $F$-graph is the initial object. Given a family of $F$-graphs $(G_i)_{i\in I}$ the coproduct is given by taking the disjoint union of parts with incidence operator induced by the universal property of the coproduct on the cocone \[(\xymatrix{G_i(P)\ar[r]^-{\partial_{G_i}} &   F(G_i(V)) \ar[r]^-{F(s_i)} & F(\bigsqcup_I G_i(V)) })_{i\in I}\] where $s_i\colon G_i(V)\to \bigsqcup_I G_i(V)$ is the coproduct inclusion. Given a pair of morphisms $f,g\colon G\to H$, the coequalizer $\coeq(f,g)$ has part set equal to $H(P)/\sim$ where $\sim$ is the equivalence generated by the relation 
\[ 
 	\setm{(f(a),g(a))\in H(P)\x H(P)\ }{\ a\in G(P)}
 \] 
 and vertex set equal to the image of $H(V)\to H(P)/\sim$. The incidence $\partial_{\coeq(f,g)}\colon \coeq(f,g)\to F(\coeq(f,g)(V))$ is induced by the universal property of coequalizer. 
\[
\xymatrix{ G(P) \ar@<+.3em>[r]^f \ar@<-.3em>[r]_g & H(P) \ar[d]_{\partial_H} \ar[r] & \coeq(f,g) \ar@{..>}[dl] \ar[d]^{\partial_{\coeq(f,g)}}  \\ & F(H(V)) \ar[r] & F(\coeq(f,g)(V))}
\]
It is straightforward to verify these are well-defined reflexive $F$-graphs which enjoy universal properties.

Let $\rG_{(X,M)}$ be a theory for reflexive $(X,M)$-graphs. Define the set $M_A\defeq \frac{M}{\sim}$ where $\sim$ is the equivalence relation such that $m\sim m'$ iff there exists an invertible $n\in M$ such that $mn=m'$. This makes $M_A$ a right $M$-set with the obvious action. Let $q\colon M_A\to F(X)$ be a set map such that for each $m\in M$ we have $F(m)\circ q=q\circ m$ 
\[
	\xymatrix{ M_A \ar[r]^{m} \ar[d]_q & M_A \ar[d]^{q} \\ F(X) \ar[r]^{F(m)} & F(X) }
\]
where $m\colon M_A\to M_A$ is the right-action map.
 Define $I(V)(P)=1$ (and thus has a single vertex with no edges) and $I(A)(P)=M_A$, with vertex set $I(A)(V)=X$ and inclusion $I(A)(V)\hookrightarrow I(A)(P)$\footnote{Recall $X=\Fix(M)$.} with incidence defined by $\delta_{I(A)}\defeq q$. For morphisms we assign for each $x_m\in X$ and $m\in M$ 
\begin{align*}
(x_{m'}\colon V\to A)\quad \, &\mapsto \quad  I(x_{m'})_P\defeq \named{m'}\colon 1\to M_A,\\
(m\colon A\to A)\quad &\mapsto \quad  I(m)_P\defeq m\colon M_A\to M_A ,\\
(\ell\colon A\to V) \quad \, &\mapsto \quad  I(\ell)_P\defeq !_{M_A}\colon M_A\to 1 \ \text{(the terminal set map)}
\end{align*}
which is readily verified to define an interpretation functor $I\colon \rG_{(X,M)}\to \crG_F$.

In the following, we will consider the properties of the nerve realization adjunction $R\dashv N$ induced by $I$ as well as the restriction to an adjoint equivalence between fixed points.\footnote{Recall that the fixed points of an adjunction $F\dashv G\colon \C A\to \C B$ are the full subcategories $\C A'$ and $\C B'$ of $\C A$ and $\C B$ consisting of objects such that the counit and unit of the adjunction are isomorphisms. This in particular implies that $\C A'$ is equivalent to $\C B'$. } 

\section{The Category of Hypergraphs}\label{S:PGraphs}

We recall that a hypergraph $H=(H(V),H(E),\varphi)$ consists of a set of vertices $H(V)$, a set of edges $H(E)$ and an incidence map $\varphi\colon H(E)\to \C P(H(V))$ where $\C P\colon \mathbf{Set}\to \mathbf{Set}$ is the covariant power-set functor. In other words, we allow infinite vertex and edge sets, multiple edges, loops, empty edges and empty vertices.\footnote{An empty vertex is a vertex  not incident to any edge in $H(E)$. An empty edge is an edge $e$ such that $\varphi(e)=\empset$.} In other words the category of hypergraphs $\C H$ is the category of $\C P$-graphs.

Let $X$ be a set and apply the definition for the interpretation given above in \ref{S:Interp} for $\sG_X$ with $q\defeq X$ in $\C P(X)$. Note that for each automap $\sigma\colon X\to X$, $\C P(\sigma)$ is the identity map. Thus the interpretation $I\colon \sG_X\to \C H$ defined in Chapter \ref{S:Interp} is a well-defined functor. 

The nerve $N\colon \C H\to \widehat{\sG}_{X}$ induced by $I$ takes a hypergraph $H=(H(E),H(V),\varphi)$ to the symmetric $X$-graph $N(H)$ with vertex and arc set given by
\begin{align*}
N(H)(V)&=\C H(I(V),H)=H(V),\\
N(H)(A)&=\C H(I(A),H)=\setm{(\beta,f)\in H(E)\x H(V)^X}{\C P(f)=\varphi(\beta)}
\end{align*} Notice that in the case a hyperedge $e$ has less than $\#X$ incidence vertices the nerve creates multiple edges and if a hyperedge has more than $\#X$ incidence vertices there is no arc in the correponding symmetric $X$-graph given by the nerve (see Example \ref{E:HyperNR} below).

The realization $R\colon \widehat{\sG}_{X}\to \C H$ sends a symmetric $X$-graph $G$ to the hypergraph $R(G)=(R(G)(E),R(G)(V),\psi)$ with vertex, edge sets and incidence map given by
\begin{align*}
& R(G)(V)=G(V),\\  & R(G)(E)=G(A)/\sim, \quad \text{($\sim$ induced by $\sX$)},\\
&\psi\colon R(G)(E)\to \C P(R(G)(V)), \quad [\gamma]\mapsto \setm{v\in G(V)}{\exists x\in X, \gamma.x=v}
\end{align*} 
For a symmetric $X$-graph morphism $f\colon G\to G'$, the hypergraph morphism $R(f)\colon R(G)\to R(G')$ has $R(f)_V\defeq f_V$ and $R(f)_E\defeq [f_A]$ where $[f_A]\colon \frac{G(A)}{\sim} \to \frac{G'(A)}{\sim}$ is induced by the quotient.

\begin{example}\label{E:HyperNR} \mbox{}
	\begin{enumerate}
		\item \label{E:Creates} Let $X=\{a,b,c\}$ and consider the hypergraph $H$ with two vertices $0$ and $1$ and one hyperedge $\alpha$ between them. Then the nerve $N(H)$ has two vertices $0$ and $1$ and arc set $N(H)(A)=\{001, 010, 100, 011, 101, 110 \}$ with a pair of $\sX$-partners $001\sim 010\sim 100$ and $011\sim 101\sim 110$. The realization identifies the $\sX$-partners and thus $RN(H)$ has two vertices $0$ and $1$ and two edges $[001]$ and $[011]$ between them. The counit $\varepsilon_H\colon RN(H)\to H$ is bijective on vertex set and sends $[001]$ and $[011]$ to $\alpha$.
		\item Let $X=\{a,b,c\}$ and consider the hypergraph $H$ with four vertices and one hyperedge $\alpha$ connecting them. The nerve $N(H)$ has vertex set equal to $H(V)$ but empty arc set $N(H)(A)=\empset$.  The counit $\varepsilon_H\colon RN(H)\to H$ is the inclusion of vertices.
	\end{enumerate}
\end{example}

Let $k$ be a cardinal number. Recall that a hypergraph $H=(H(E),H(V),\varphi)$ is $k$-uniform provided for each edge $e\in H(E)$, the set $\varphi(e)$ has cardinality $k$.

\begin{proposition} \label{P:HyperGSG}
	Let $k$ be the cardinality of $X$ and $I\colon \sG_X\to \C H$ be the interpretation above. The fixed points of the nerve realization adjunction $R\dashv N\colon \C H\to \sG_X$ is equivalent to the category of $k$-uniform hypergraphs, $k\C H$.  Moreover, the inclusion $i\colon k\C H\to \widehat{\sG}_X$ preserves limits and any exponential objects which exist in $k\C H$. 
\end{proposition}
\begin{proof}
	It is clear that the fixed points is the category of $k$-uniform hypergraphs and that the product (respectively, equalizer) of $k$-uniform hypergraphs in $\widehat{\sG}_X$ is $k$-uniform. Thus the inclusion $i\colon k\C H\to \widehat{\sG}_X$ preserves limits. To show that $N$ must preserve any exponentials that exist, suppose $G^H$ is an exponential object in $k\C H$. We have the following natural isomorphisms:
	\begin{align*}
		N(G^H)(V)&=k\C H(I(V),G^H)\cong k\C H(I(V)\x H, G)\\
		 &\cong \widehat{\sG}_X(NI(V)\x N(H),N(G))\\
		 &\cong \widehat{\sG}_X(\underline V,N(G)^{N(H)})\cong N(G)^{N(H)}(V),\\\\
		 N(G^H)(A)&=k\C H(I(A),G^H)\cong k\C H(I(A)\x H, G)\\
		 &\cong \widehat{\sG}_X(NI(A)\x N(H),N(G))\\
		 &\cong \widehat{\sG}_X(\underline A,N(G)^{N(H)})\cong N(G)^{N(H)}(A).
	\end{align*}
Therefore, $N$ preserves any exponentials which exist in $k\C H$. 
\end{proof}
\begin{corollary} \label{C:HyperExp}
If $k$ is a cardinal number greater than $1$, the category of $k$-uniform hypergraphs does not have exponentials. 	
\end{corollary}
\begin{proof}
	Example \ref{E:Exponentials}(\ref{E:HypergraphFail}) provides us with a counterexample. 
\end{proof}

The previous results shows that since $k\C H$ lacks connected colimits and exponentials and the inclusion $i\colon k\C H\hookrightarrow \C H$ does not preserve limits, we may continuously embed $k\C H$ in $\widehat{\sG}_X$ giving us the factorization of the inclusion $j\colon k\C H\hookrightarrow \C H$
\[
\xymatrix@!=.3em{ & \widehat{\sG}_{X} \ar[dr]^{R} & \\ \kH \ar@{>->}[ur]^{i} \ar@{>->}[rr]^j && \C H.}
\]
and work in the improved categorical environment $\widehat{\sG}_X$.

We are also able to use the adjunction above to classify the projective objects in the category of hypergraphs. Recall that a right adjoint functor is faithful if and only if the counit is an epimorphism.

\begin{lemma}
	Let $X$ be a set with cardinality $\kappa$ greater than $1$. Then the nerve $N\colon \C H\to \widehat{\sG}_X$ of the interpretation $I\colon \sG_X\to \C H$ is faithful on the full subcategory $\C H_\kappa$ consisting of hypergraphs $H$ with $\max_{e\in H(E)}\varphi(e)$ at most of cardinality $\kappa$.
\end{lemma}
\begin{proof}
	Given hypergraph $H$ in the subcategory $\C H_{\kappa}$ the counit $\varepsilon_H\colon RN(H)\to H$ of the adjunction $R\dashv N\colon \C H\to \widehat{\sG}_X$ at component $H$ is an epimorphism since for each hyperedge $e\in H(E)$, there is a symmetric $X$-graph morphism $f\colon I(A)\to H$ such that $f_E$ takes the lone hyperedge in $I(A)$ to $e$. 
	
\end{proof}

Let $I(V)$ be the hypergraph with one vertex and no hyperedges. For each cardinal number $k$, let $E_k\defeq I(A)$ be the hypergraph where $I\colon \sG_k\to \C H$ is the interpretation functor above. 

\begin{lemma}
	The proper class of objects consisting of the vertex object $I(V)$ and hyperedge objects $(E_k)_{k\in \tbf{Card}}$ is a family of separators for the category of hypergraphs.  
\end{lemma}
\begin{proof}
	Let $f, g\colon H\to H'$ be distinct hypergraph morphisms. Let $\kappa$ be the maximum of the cardinalities such that $H$ and $H'$ are in the subcategory $\C H_{\kappa}$. By the lemma above, the nerve $N\colon \C H\to \widehat{\sG}_{X}$ is faithful on $\C H_\kappa$. Thus we have $N(f)\neq N(g)$. Therefore, either $I(V)$ separates $f$ and $g$ or $I(\underline A)=E_k$ separates $f$ and $g$ by definition of the nerve functor. 
	
\end{proof}

\begin{proposition}
	A hypergraph $H$ is (regular) projective if and only if it has no hyperedges.  
\end{proposition}
\begin{proof}
	If a hypergraph has no edges it is projective since it is a coproduct of the vertex object $I(V)$ which is clearly projective. Conversely, by the lemma above it is enough to show that each hyperedge object $E_k$ is not projective. Let $1$ be the terminal hypergraph with one vertex and one hyperedge. Then every morphism $E_k\to 1$ is a (regular) epimorphism. Let $E_r$ be the hyperedge object with $r$ vertices where $r$ is of cardinality strictly greater than $k$. Since there is no morphism from $E_k$ to $E_r$ there is no factorization of $E_k\to 1$ through $E_r$ showing $E_k$ is not (regular) projective.
	
\end{proof}

For each set $X$, the interpretation functor $I\colon \sG_X\to \C H$ factors through the full subcategory $\C H_{k}$ of hypergraphs consisting of hypergraphs $H$ such that the incidence of each edge is of cardinality less than or equal to the cardinality $k$ of $X$. In other words, $\C H_k$ is the slice category $\mathbf{Set}\ls \C P_k$ where $\C P_k$ is the covariant $k$-power set functor which takes a set to the set of all subsets with cardinality less than or equal to $k$. The inclusion functor $i\colon \C H_k\hookrightarrow \C H$ admits a coreflector $r\colon \C H\to \C H_k$ which takes a hypergraph $H$ to the hypergraph $r(H)$ with vertex set $r(H)(V)=H(V)$ and edge set $r(H)(E)\defeq \setm{\alpha\in H(E)}{\#\varphi(\alpha)\leq k}$. Therefore the nerve realization for the interpretation $I_k\defeq rRy\colon \sG_X\to \C H_k$ is $R_k\dashv N_k\colon \C H_k\to \widehat{\sG}_X$ where $R_k=rR$ and $N_k=Ni$. Moreover, by restricting to the subcategory $\C H_k$ the counit $\varepsilon_H\colon rRNi(H)\Rightarrow H$ is now an epimorphism.

\begin{proposition}
	For a cardinal number $k$, the class of projective objects in $\C H_k$ are precisely the coproducts of $I_k(V)$ and $I_k(A)$ where $I_k\colon \sG_X\to \C H_k$ is the interpretation functor described above. 
\end{proposition}
\begin{proof}
	It is clear that $I_k(V)$ is projective. Since a hypergraph $H$ in $\C H_k$ has an edge if and only if it admits a morphism from $I_k(A)$ to it and since epimorphisms in $\C H_k$ are those morphisms surjective on vertex and edge sets, it is clear $I_k(A)$ is projective. Therefore the coproducts of $I_k(V)$ and $I_k(A)$ are projective. Conversely, consider the following composition 
	\[
		\xymatrix{ R_k(\Proj(N(H))) \ar@{->>}[r] & R_kN_k(H) \ar@{->>}[r]^-{\varepsilon_H} & H}
	\]
	where the morphism $\xymatrix{R_k(\Proj(N(H))) \ar@{->>}[r] & R_kN_k(H)}$ is the application of $R_k$ on the epimorphism $\Proj(N(H))\to N(H)$ described in Chapter \ref{S:InjProj}. Note that since $R_k$ is a left adjoint, it preserves epimorphisms. Moreover, $R_k$ preserves colimits, therefore $R_k(\Proj(N(H)))=\bigsqcup_{N(H)(V)}R_k\underline V\sqcup \bigsqcup_{N(H)(A)}R_k\underline A$. Then since $R_ky=I_k$ where $y$ is the Yoneda embedding, we have $R_k(\Proj(N(H)))=\bigsqcup_{N(H)(V)}R_k\underline I_k(V)\sqcup \bigsqcup_{N(H)(A)}I_k(A)$. Therefore, every object in $\C H_k$ admits an epimorphism from a projective object (i.e., $\C H_k$ has enough projectives). Thus every projective in $\C H_k$ is a split subobject of the essential image of the functor $R_k\Proj\colon \widehat{\sG}_X\to \C H_k$. However it is clear that the only split subobjects are coproducts of $I_k(V)$ and $I_k(A)$. 
	
\end{proof}

\begin{proposition}
	A hypergraph $Q$ is injective if and only if $Q$ is non-initial and for each subset of $S\subseteq Q(V)$, there is an edge $\alpha\in Q(E)$ with $S=\varphi(\alpha)$.
\end{proposition}
\begin{proof}
	Suppose $Q$ is injective. For each subset $S\subseteq Q(V)$, let $E_S$ be the hypergraph with one edge $e$ with incidence equal to $\varphi(e)=S$. Let $f\colon \coprod_S I(V)\hookrightarrow E_S$ and $g\colon \coprod_S I(V)\hookrightarrow Q$ be the inclusions of vertices. Since $Q$ is injective, there is a morphism $h\colon E_S\to Q$ which necessarily is a monomorphism. Hence $Q$ must have an edge $q$ with incidence $\varphi(q)=S$. 
	
	Conversely, let $f\colon G\to H$ be a monomorphism and $g\colon G\to Q$ be a morphism in the category of hypergraphs. Since $Q$ is non-initial, there is a vertex $v\in Q(V)$. We define the morphism $h\colon H\to Q$ on vertices
	\begin{align*}
		h_V(w)\defeq \begin{cases} g_V(u) & \text{if } \exists u\in G(V),\ g_V(u)=w\\
		v & \text{if } \forall u\in G(V),\ g_V(u)\neq w.\end{cases}
	\end{align*}
Each edge $e$ in $H$ not in the image of $f_A$ has incidence a subset $S\subseteq H(V)$ which can be decomposed $S\cong S_0\sqcup S_1$ such that $S_0$ is in the image of $f_V$ and $S_1$ is disjoint to the image of $f_V$. Then for such an edge $e$ in $H(E)$, choose an edge $[e]\in Q(E)$ with incidence $g(f_V^{-1}(S_0))\cup \{v\}$ where $g_V(f_V^{-1}(S_0))$ is in the image of $f_V^{-1}(S_0)$ under $g_V$. Then we define 
\[
	h_E(e)\defeq \begin{cases}g_E(b) & \text{if } \exists b\in G(E),\ g_E(b)=e\\
	[e] & \text{if } \forall b\in G(E),\ g_E(b)\neq e. \end{cases}
\]
Then $h_V$ and $h_E$ describe a morphism of hypergraphs $h\colon H\to Q$ such that $h\circ f=g$. Therefore $Q$ is injective.
\end{proof}
\begin{corollary}
	Let $Q$ be a hypergraph and $X\defeq Q(V)$. Then $Q$  is injective if and only if $N(Q)$ is injective as a symmetric $X$-graph where $N$ is the nerve of the interpretation $I\colon \sG_X\to \C H$ defined above.
\end{corollary}
\begin{proof}
	Let $Q$ be an injective hypergraph. By Proposition \ref{P:Injective}, it is enough to show that for each set map $j\colon X\to N(Q)(V)$, there is an arc $\alpha\in N(Q)(A)$ such that $\partial(\alpha)=j$. the image of $j$ describes a subset $S$ of vertices of $Q$. Therefore, by the above result, there is a hyperedge $e$ with incidence equal to $S$. Let $\alpha\colon I(A)\to Q$ be the arc in $N(Q)$ corresponding to the hypergraph morphism which takes the vertex $x$ to $j(x)$ for each $x\in X$ and the single hyperedge $a\in I(a)$ to $e$. Then $\partial(\alpha)=j$ and thus $N(Q)$ is an injective symmetric $X$-graph.
	
	Conversely, suppose $N(Q)$ is injective and let $S\subseteq Q(V)$. Let $j\colon X\to N(Q)(V)$ be a set map with image equal to $S$. There is an arc $\alpha\in N(Q)(A)$ with incidence $\partial(\alpha)=j$.  Since $\alpha$ corresponds to the hypergraph morphism $\alpha\colon I(A)\to Q$, there must be an edge $e\in Q$ such that $a\in I(A)$ is mapped to $e$, i.e., $e$ has incidence equal to $Q$. Therefore, $Q$ is an injective hypergraph. 
\end{proof}

\ignore{
\section{The Category of Reflexive Hypergraphs}

Although the concept of a hypergraph in which the vertices are to be considered as degenerate edges has not been examined in the literature, it is straightforward to provide a definition in the context of reflexive $F$-graphs.

\begin{definition}
	The category of reflexive hypergraphs $r\C H$ is defined to be the category of reflexive $\C P$-graphs $r\C G_{\C P}$ where $\C P$ is the covariant powerset functor. The natural transformation $\eta\colon \Id_{\mathbf{Set}}\Rightarrow \C P$ is the singleton assignment, i.e., $\eta_{X}\colon X\to \C P(X)$ takes an element $x\in X$ to the singleton $\{x\}\in \C P(X)$. 
\end{definition}

Given the theory $\srG_X$, note that $M_A\iso X\sqcup 1$. We define the interpretation functor $I\colon \srG_X\to r\C H$ by assigning $q\colon X\sqcup 1 \to \C P(X)$ induced by the maps $\eta_X\colon X\to \C P(X)$ which maps $x$ to the singleton $\{x\}$ and .

Given a cardinal number $k$, a reflexive hypergraph $H$ is $k$-uniform provided each non-distinguished edge (i.e., an edge $e\in r\C H(P)$ is) is incident to $k$ distinct vertices. 

\tbf{FINISH THIS!!!!}

}

\section{The Category of Power Graphs} \label{S:UnderlinePiGraphs}

Let $X$ and $Y$ be sets. We define the symmetric $X$-power of $Y$, denoted $\underline \Pi_X(Y)$, as the multiple coequalizer of $(\underline \sigma\colon \Pi_X(Y) \to \Pi_X(Y))_{\sigma\in \sX}$ where $\underline \sigma$ is the $\sigma$-shuffle of coordinates in the product. This definition extends to a functor $\underline{\Pi} _{X}\colon \mathbf{Set} \to \mathbf{Set}$. Note that if $j\colon X'\to X$ is a set map, then there is a natural transformation $\underline{\Pi} _{X}\Rightarrow \underline{\Pi} _{X'}$ induced by the universal mapping property of the product. In particular, when $X\to X'=1$ is the terminal map, we have $\Id_{\mathbf{Set}}=\underline{\Pi} _1\Rightarrow \underline{\Pi} _X$ which we denote by $\eta\colon \Id_\mathbf{Set}\Rightarrow \underline{\Pi} _X$.\footnote{Note that in the case $X=2$, the category of $\underline \Pi_X$-graphs is the category of undirected graphs in the conventional sense in which morphisms are required to map edges to edges.  }

To define an interpretation functor $I\colon \sG_{X}\to \C G_{\underline \Pi_X}$, we let $q$ be the unordered set $(x)_{x\in X}$ in $\underline \Pi_X(X)$.  Since $\underline \Pi_X(\sigma)(x)_{x\in X}=(x)_{x\in X}$ for each automap $\sigma\colon X\to X$, the interpretation is well-defined.

\begin{lemma}
	The interpretation $I\colon \sG_{X}\to \C G_{\underline \Pi_X}$ is dense, full and faithful.
\end{lemma}
\begin{proof}
	It is clearly full and faithful. To show it is dense, let $(E,V,\varphi)$ and $(K,L,\psi)$ be $\C G_{\underline \Pi_X}$-objects and $\lambda\colon D\Rightarrow \Delta(K,L,\psi)$ a cocone on the diagram $D\colon I\ls (E,V,\varphi)\to \C G_{\underline \Pi_X}$. 
	Let $e$ be an edge in $E$ and $f\colon X\to V$ be the set morphism with $\underline{\Pi} _Xf=\varphi(e)$. Then $(\named e, f)\colon I(A)=(1,X,\top)\to (E,V,\varphi)$ is an object in $I\ls (E,V,\varphi)$ and thus there is a morphism $\lambda_{(\named {e},f)}\eqdef(\named{e'},g)\colon D(\named e,f)=(1,X,\top)\to (K,L,\psi)$. By the compatibility of the cocone, this gives us a uniquely defined $h\colon E\to K$, $e\mapsto e'$ on edges. Similarly for each vertex $v\in V$, there is a morphism $(!_E,\named v)\colon I(V)=(\empset, 1,!_1)\to (E,V,\varphi)$ and a cocone inclusion $(!_K,\named w)\colon D(!_E,\named v)=(\empset, 1,!_1)\to (K,L,\psi)$ giving us a factorization on vertices $k\colon V\to L$. Since $\psi\circ h(e)=\underline{\Pi} _X(kf)\circ \top=\underline{\Pi} _X(k)\circ \varphi(e)$ for each edge $E$,  $(h,k)\colon (E,V,\varphi)\to (K,L,\psi)$ is a well-defined $\C G_{\underline \Pi_X}$-morphism which necessarily is the unique factorization of the cocone. Therefore, $I$ is dense. 
	
\end{proof}

Note that the realization functor takes a $\widehat{\sG}_{X}$-object and quotients out the set of arcs by $\sX$. Hence the unit of the adjunction $\eta_P\colon P\to NR(P)$ is bijective on vertices and surjective on arcs. Hence the adjunction is epi-reflective. 

\begin{proposition}\label{P:EENerve}
	The nerve-realization $R\dashv N\colon \C G_{\underline \Pi_X}\to \widehat{\sG}_{X}$ induced by $I\colon \sG_{X}\to \C G_{\underline \Pi_X}$ is the presheaf topos completion of $\C G_{\underline \Pi_X}$.
\end{proposition}
\begin{proof}
	Follows from Proposition \ref{P:PresheafCompletion}, that a dense, full and faithful interpretation induces the presheaf completion. 
\end{proof}
For a $\C G_{\underline \Pi_X}$-object $(B,C,\varphi)$, the embedding given by the nerve functor is given by 
\begin{align*}
N(B,C,\varphi)(V)&=\C G_{\underline \Pi_X}(I(V), (B,C,\varphi))\iso C,\\
N(B,C,\varphi)(A)&=\C G_{\underline \Pi_X}(I(A), (B,C,\varphi))\\ &=\setm{(e,g)\ }{\ e\in B,\ g\colon X\to C\ s.t.\ \underline{\Pi} _Xg=\varphi(e)} 
\end{align*}
The  right-actions are by precomposition, i.e., $(e,g).x=(e,g\circ \named x)$, $(e,g).\sigma=(e,g\circ \sigma)$. 

Let us show that all loops in the objects of the full subcategory of $\widehat{\sG}_{X}$ equivalent to $\C G_{\underline \Pi_X}$ are fixed loops. A loop in a $\C G_{\underline \Pi_X}$-object $(B,C,\varphi)$ is an edge $e\in B$ such that $\varphi(e)$ is $(v)_{x\in X}$ in $\underline{\Pi} _X(C)$ for some $v\in C$.  Therefore, there is only one morphism $(\named e, f)\colon I(A)\to (B,C,\varphi)$ and thus $(\named e,f\circ \sigma)=(\named e,f)$ for each $\sigma\in \sX$. Hence, each object in the reflective subcategory of $\widehat{\sG}_{X}$ equivalent to $\C G_{\underline \Pi_X }$ has only fixed loops.

\begin{corollary}\label{C:Exponentials}
	If $X$ has cardinality greater than $1$, the category $\C G_{\underline \Pi_X}$ does not have exponentials. 
\end{corollary}
\begin{proof}
	By the above observation, it is enough to show that there exist objects $G$ and $H$ in $\C G_{\underline \Pi_X}$ such that $N(G)^{N(H)}$ has a unfixed loop in $\widehat{\DD G}_{(X,\sX)}$. Set $H\defeq I(A)$ and $G$ be the graph with one vertex and an $\sX$-loop. Then $N(G)^{N(H)}=L^{\underline A}$ as defined in Example \ref{E:Exponentials}(\ref{E:XLoops}) which we have shown has a unfixed loop. 
\end{proof}

It is straightforward to show that the full subcategory of simple hypergraphs in $\C G_{\underline \Pi_X}$ is a reflective subcategory which restricts to the nerve-realization adjunction. 

Next, we show that injective and projectives in the category of $\underline \Pi_X$-graphs are precisely those objects which are taken to injective and projective objects n the category of symmetric $X$-graphs.

\begin{proposition}\label{P:InjectiveN}
A $\underline \Pi_X$-graph $Q$ is injective if and only if $N(Q)$ is an injective symmetric $X$-graph.
\end{proposition}
\begin{proof}
If $N(Q)$ is injective, then $Q$ is injective since $N$ is full and faithful and preserves monomorphisms. Conversely, let $Q$ be an injective $\underline \Pi_X$-graph and consider the monomorphism $f\colon G\to H$ and morphism $g\colon G\to N(Q)$ of symmetric $X$-graphs. The realization functor preserves monomorphisms, hence $R(f)\colon R(G)\to R(H)$ is a monomorphism. Since the counit $\varepsilon_Q \colon RN(Q)\to Q$ is an isomorphism, $RN(Q)$ is injective and thus there is a morphism $h\colon R(H)\to RN(Q)$ such that $h\circ R(f)=R(g)$. Therefore, the following diagram commutes
	
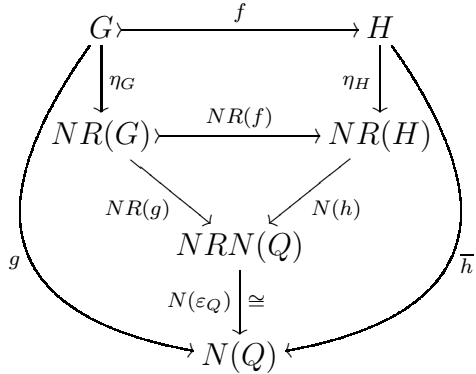
\[
	\xymatrix@C=.25em{ G \ar@/_5em/[dddr]_{g} \ar@{>->}[rr]^-f \ar[d]^{\eta_G} && H \ar@/^5em/[dddl]^{\overline h} \ar[d]_{\eta_H} \\ NR(G) \ar@{>->}[rr]^-{NR(f)} \ar[dr]_{NR(g)} && NR(H) \ar[dl]^{N(h)} \\ & NRN(Q) \ar[d]_{N(\varepsilon_Q)}^\iso & \\ & N(Q) &}
	\]
	where $\overline h\defeq \varepsilon_Q\circ N(h)\circ \eta_H$. Thus, $\overline h\circ f=g$ and hence $N(Q)$ is injective.
	
\end{proof}

\begin{proposition}\label{P:ProjectiveN}
	A $\underline \Pi_X$-graph $P$ is projective if and only if $N(P)$ is a projective $(X,\Aut(X))$-graph. 
\end{proposition}
\begin{proof}
	If $N(P)$ is projective, then $P$ is projective since $N$ is full and faithful and preserves epimorphisms. Conversely, let $P$ be a projective $\underline \Pi_X$-graph. It is clear that $I(V)$ and $I(A)$ are projective objects in $\C G_{\underline \Pi_X}$, thus $R(\Proj(N(P)))\iso \bigsqcup_{N(P)(V)}I(V)\sqcup \bigsqcup_{N(P)(A)}I(A)$ is projective. Since the projective refinement $\Proj(N(P)))\to N(P)$ (see Chapter \ref{S:InjProj}) is an epimorphism and $\varepsilon_P\colon RN(P)\to P$ is an isomorphism, the composition $R(\Proj(N(P)))\to RN(P)\to P$ is an epimorphism. Thus, $P$ is a split subobject of a coproduct of $I(V)$ and $I(A)$. However, the only split subobjects of such a coproduct is itself a coproduct of $I(V)$ and $I(A)$. Then since $N$ preserves coproducts and $NI(V)=\underline V$ and $NI(A)=\underline A$, $N(P)$ is projective. 
	
\end{proof}

\section{The Category of Reflexive Power Graphs}

Let $r\C G_{\underline \Pi_X}$ be the category of reflexive $\underline \Pi_X$-graphs.\footnote{When $X=2$, the category of reflexive $\underline \Pi_X$-graphs is the category of conceptual graphs as given in \cite{dP}.}
To define an interpretation functor $I\colon \srG_{X}\to r\C G_{\underline \Pi_X}$, note that $M_A\iso X\sqcup 1$. Let $\eta\colon \Id_\mathbf{Set}\Rightarrow \underline \Pi_X$ be the natural transformation defined above and let $q\colon X\sqcup 1\to \underline \Pi_X(X)$ the map induced by the singleton assignment $\eta_X\colon X\to \underline \Pi_X(X)$, $x'\mapsto (x')_{x\in X}$ and $\top\colon 1\to \C P(X)$, $x\mapsto (x)_{x\in X}$.  Since $\underline \Pi_X(\sigma)(x)_{x\in X}=(x)_{x\in X}$ for each automap $\sigma\colon X\to X$ and $\underline \Pi_X(\overline{x'})(x)=(x')_{x\in X}$ for each constant map $\overline {x'}\colon X\to X$, the interpretation is well-defined. 

\begin{lemma}
	The interpretation functor $I\colon \srG_{X}\to r\C G_{\underline \Pi_X}$ is dense, full and faithful.
\end{lemma}
\begin{proof}
	It is clearly full and faithful. To show it is dense, let $G$ and $H$ be $r\C G_{\underline \Pi_X}$-objects and $\lambda\colon D\Rightarrow \Delta H$ a cocone on the canonical diagram $D\colon I\ls G \to r\C G_{\underline \Pi_X}$. It can be verified that $I(A)$ classifies the parts set $G(P)$ of a graph $G$ up to precomposition by automorphism $A'\to A'$. In other words, $G(P)\iso \frac{r\C G_{\underline \Pi_X}(I(A),G)}{\sim}$ and $H(P)\iso \frac{r\C G_{\underline \Pi_X}(I(A),H)}{\sim}$ where $\sim$ is the equivalence relation induced by automorphisms of $I(A)$. Thus we define $h_P\colon G(P)\to H(P)$, $[e]\mapsto [\lambda_{e}]$ where $[e]$ is the equivalence class of the morphism $e\colon I(A)\to G$ and $\lambda_e\colon D(e)\to H$ is the component of the natural transformation $\lambda$. Since $\lambda$ is a cocone, the map is compatible with incidence operations and the restriction to vertex sets, $h_V\colon G(V)\to H(V)$. Thus $h\colon G\to H$ is the unique factorization which shows the colimit of $D$ is $G$.
	
\end{proof}

Note that the realization functor takes a $\widehat{\srG}_{X}$-object and quotients out the set of arcs by $\sX$. Hence the unit of the adjunction $\eta_P\colon P\to NR(P)$ is bijective on vertices and surjective on arcs. Hence the adjunction is epi-reflective. 

\begin{proposition}\label{P:EVNerve}
	The nerve-realization $R\dashv N\colon r\C G_{\underline \Pi_X}\to \widehat{\srG}_{X}$ induced by $I\colon \srG_{X}\to r\C G_{\underline \Pi_X}$ is the presheaf topos completion of $r\C G_{\underline \Pi_X}$.
\end{proposition}
\begin{proof}
	Follows from Proposition \ref{P:PresheafCompletion}, that a dense, full and faithful interpretation induces the presheaf completion.
	
\end{proof}
The full subcategory of $\widehat{\srG}_{X}$ induced by the nerve functor consists of reflexive symmetric $X$-graphs which have no unfixed loops. Indeed if $G$ is a $r\C G_{\underline \Pi_X}$-object then $N(G)(A)=r\C G_{\underline \Pi_X}(I(A),G)$ and so if $e\colon I(A)\to G$ is a loop, i.e., for each $x\in X$ there is a $v\colon I(V)\to I(A)$ such that $e\circ I(x)=v$, then $e\circ I(\sigma)=e$. 
\begin{corollary}\label{C:NoExpR}
	If $X$ has cardinality greater than $1$, the category $r\C G_{\underline \Pi_X}$ does not have exponentials.  
\end{corollary}
\begin{proof}
	By the above observation, it is enough to show that there exist objects $G$ and $H$ in $\C G_{\underline \Pi_X}$ such that $N(G)^{N(H)}$ has an unfixed in $\widehat{\DD G}_{(X,\sX)}$. Set $H\defeq I(A)$ and $G$ be the graph with one vertex and two  unfixed loops. Then $N(G)^{N(H)}=L^{\underline A}$ as defined in Example \ref{E:Exponentials}(\ref{E:1Loops}) which we have shown has a unfixed loop.
	
\end{proof}	

It is also straightforward to show that the full subcategory of simple hypergraphs in $r\C G_{\underline \Pi_X}$ is a reflective subcategory which allows a restriction to the nerve-realization adjunction. 

The proofs for the preservation of injective and projective objects by the nerve $N$ are similar to the proofs in Proposition \ref{P:InjectiveN} and \ref{P:ProjectiveN} and are thus omitted.

\begin{proposition}
A reflexive $\underline \Pi_X$-graph $Q$ (resp. $P$) is injective (resp. projective) if and only if $N(Q)$ (resp. $N(P)$) is injective (resp. projective). 
\end{proposition}

\section{Other Categorical Structures}
The presheaf completions of the conventional categories of uniform hypergraphs above can be used to prove what topos structure exists in these categories. Immediate facts which are true for any reflective subcategory are that $\C G_{\underline \Pi_X}$ and $r\C G_{\underline \Pi_X}$ are complete, cocomplete, and well-powered. The full and faithful right adjoint $N$ creates limits, monomorphisms, split monomorphisms, split epimorphisms, and isomorphisms. Moreover, filtered colimits commute with finite limits (\cite{fB}, v1, Proposition 3.5.7).

It is easy to see that regular epimorphisms in $\C G_{\underline \Pi_X}$ and $r\C G_{\underline \Pi_X}$ are also created by $N$ as well, since regular epimorphisms in presheaf toposes are those surjective on components. However, $N$ does not preserve coequalizers. Take for instance the inclusions $I(s), I(t)\colon I(V)\to I(A)$, the coequalizer in $\C G_{\underline \Pi_2}$ and $r\C G_{\underline \Pi_2}$ are the graphs with one vertex and one loop. The coequalizer of $\underline s, \underline t\colon \underline V\to \underline A$ in $\widehat{\sG}_{2}$ and $\widehat{\srG}_2$ is the graph with one vertex and one 2-loop. 

To obtain colimits in $\C G_{\underline \Pi_X}$ and $r\C G_{\underline \Pi_X}$, we need to reduce unfixed loops to 1-loops. By the adjunctions $R\dashv N$ and \cite{fB} (v1, Proposition 5.5.6, p 213) each morphism $f\colon P\to Q$ in $\widehat{\sG}_{X}$ and $\widehat{\srG}_{X}$  factors uniquely up to isomorphism $f=g\circ h$ where $h$ is bijective on vertices and reduces each unfixed edges to fixed edges whenever $f$ reduces such edges. The realization functor takes such an $h$ to an isomorphism.

To show that $\C G_{\underline \Pi_X}$ and $r\C G_{\underline \Pi_X}$ are regular categories, it is left to show that $\C G_{\underline \Pi_X}$ and $r\C G_{\underline \Pi_X}$ have regular images. Indeed, given a morphism $f\colon G\to H$ in either $\C G_{\underline \Pi_X}$ and $r\C G_{\underline \Pi_X}$ we may form the diagram
\[
\xymatrix@R=1em@C=1.5em{ \ker(f) \ar@<+.2em>[r] \ar@<-.2em>[r] & G \ar@{->>}[d] \ar[r]^{f} & H  \ar@<+.2em>[r] \ar@<-.2em>[r] & \coker(f) \\ & \coim(f) \ar[r]^{\theta} & \im(f) \ar@{>->}[u]}
\]
where $\ker(f)$ and $\coker(f)$ denote the kernel and cokernel pairs of $f$, and $\coim(f)$ is the coequalizer of the projections from $\ker(f)$ and $\im(f)$ is the equalizer of the inclusions of into $\coker(f)$. The morphism $\theta$ is induced by the universal mapping properties of the coimage and image. By applying the endofunctor $RN$ there is a factorization
\[
\xymatrix@R=1em@C=1.5em{ R(\ker(N(f))) \ar@<+.2em>[r] \ar@<-.2em>[r] & RN(G) \ar@{->>}[d] \ar[r]^{RN(f)} & RN(H)  \ar@<+.2em>[r] \ar@<-.2em>[r] & R\coker(N(f)) \\  & R\coim(N(f)) \ar[d]_{\iso} \ar[r]^{\iso} & R\im(N(f)) \ar@{>->}[u] \\
	& RN(\coim(f)) \ar[r]^{RN(\theta)} & RN(\im(f)) \ar@{<-}[u]_{\iso} }
\]
where $\coim(N(f))\to N(\coim(f))$, $\coker(N(f))\to N(\coker(f))$, $\im(N(f))\to N(\im(f))$ are induced by the universal property of colimits. The morphism $R\coim(N(f))\to RN(\coim(f))$ is an isomorphism since $R\dashv N$ is a reflective adjunction (\cite{fB}, v1, Proposition 3.5.4). We also have $\coim(N(f))\iso \im(N(f))$ since this is true for all presheaf toposes. The morphism $R\im(N(f))\to RN(\im(f))$ is an isomorphism since $\coker(N(f))\to N(\coker(f))$ is bijective on vertices and reduces unfixed loops to 1-loops and thus is taken to an isomorphism by $R$. Then since the counit $\varepsilon\colon RN\Rightarrow \Id$ is a natural isomorphism, $\theta$ is an isomorphism. Therefore, $\C G_{\underline \Pi_X}$ and $r\C G_{\underline \Pi_X}$ are regular categories.

Presheaf toposes have (epimorphism, monomorphism) factorizations and $N$ is epi-reflective. Therefore, by \cite{jA} (Proposition 16.8), the categories $\C G_{\underline \Pi_X}$ and $r\C G_{\underline \Pi_X}$ are closed under formation of monomorphisms, i.e., if $m\colon P\to N(H)$ is a monomorphism, then $P\cong N(G)$ for some $G$. This proves that $\Sub(G)$ and $\Sub(N(G))$ are isomorphic as meet-complete lattices for each $\C G_{\underline \Pi_X}$ and $r\C G_{\underline \Pi_X}$-object $G$. The union $\cup_I G_i$ of subobjects $(G_i\hookrightarrow G)_{i\in I}$ in $\Sub(G)$ is given by the image of the universal morphism $\bigsqcup_I G_i\to G$. Since $N$ preserves coproducts and is closed under formation of images, $N$ preserves unions. Hence, $\Sub(G)$ and $\Sub(N(G))$ are isomorphic as complete lattices. Given a morphism $f\colon G\to H$, the change of base functor $f^*\colon \Sub(H)\to \Sub(G)$ is constructed in the presheaf topos $N(f)^*\colon \Sub(N(H))\to \Sub(N(G))$. Therefore there is an adjunction $\exists_f\dashv f^*\dashv \forall_f\colon \Sub(G)\to \Sub(H)$ such that the change of base functor preserves unions since this is true in all categories of presheaves.  

The subobject classifiers have only fixed loops, and thus are objects in $\C G_{\underline \Pi_X}$ and $r\C G_{\underline \Pi_X}$. The nerve $N$ creates limits, in particular pullbacks, and so the subobject classifier $\top\colon 1\to \Omega$ classifies the subobjects in $\C G_{\underline \Pi_X}$ and $r\C G_{\underline \Pi_X}$. The natural numbers object in any presheaf category is the discrete presheaf $\Delta \DD N$ for the set of natural numbers $\DD N$. The discrete objects in $\widehat{\DD G}_{X}$ and $\widehat{\rG}_{X}$ have only fixed loops and thus are objects in both $\C G_{\underline \Pi_X}$ and $r\C G_{\underline \Pi_X}$. The universal property of $\Delta \DD N$ comes from the universal property it enjoys in $\widehat{\sG}_{X}$ and $\widehat{\srG}_{X}$.

Recall that a category $\C E$ is geometric provided it is regular, each subobject lattice has arbitrary unions and intersections, and each change of base functor on subobject lattices has a right adjoint. A geometric category is $\infty$-positive when it has disjoint coproducts (see \cite{pJ}, A1.4 p. 43). Since $N$ preserves coproducts and coproducts are disjoint in the corresponding presheaf toposes, we have proven the following.
\begin{proposition}\label{P:InftyGeo}
	The categories $\C G_{\underline \Pi_X}$ and $r\C G_{\underline \Pi_X}$ are $\infty$-positive geometric categories with a subobject classifier and a natural numbers object. 
\end{proposition}

A  $\infty$-positive geometric category is called an $\infty$-pretopos if it also has effective equivalence relations, i.e., every equivalence relation is the kernel pair of its own coequalizer. By Giraud's Theorem (\cite{pJ}, C2.2.8) an $\infty$-pretopos with a separating set of objects is a Grothendieck topos. The categories  $\C G_{\underline \Pi_X}$ and $r\C G_{\underline \Pi_X}$ do not have exponentials and thus cannot be Grothendieck toposes. Then since $\C G_{\underline \Pi_X}$ and $r\C G_{\underline \Pi_X}$ have separating set of objects, they do not have effective equivalence relations. For instance, the kernel pair of a $\widehat{\sG}_2$-morphism $f\colon \underline A \to L_2$ where $L_2$ is the symmetric $2$-graph with one vertex and one 2-loop. Then $\ker(f)\hookrightarrow \underline A\x \underline A$ corresponds to an equivalence relation in $\C G_2$ and $r\C G_2$ which is not effective. 

Therefore, we have the following conceptual equations:
\begin{align*}
	\text{Graphs } + \text{ Exponentials } &= \text{ $(X,M)$-Graphs}\\
	 \text{Graphs } + \text{ Effective Equivalence Relations } &= \text{ $(X,M)$-Graphs}.
\end{align*}
This tells us that if we want to our category of graph or uniform hypergraphs to have exponentials and effective equivalence relations, we should work in a category of $(X,M)$-graphs.

\ignore{
\chapter{Graph Reconstruction}

The graph reconstruction conjecture was first proposed by Ulam \textbf{REFERENCE}. We restate it in the language of $(X,M)$-graphs:
\begin{quote} \textbf{The Graph Reconstruction Conjecture}: Each simple reflexive symmetric $2$-graph on at least three vertices is determined up to isomorphism by its deck. \footnote{Note that we use reflexive graphs since the distinguished 1-loops can be considered as vertex proxies and not as legitimate arcs.} 
\end{quote}
We recall that the category of simple reflexive symmetric $2$-graphs $\Sep(\srG_2,\neg\neg)$ is a quasi-topos, i.e., it is a reflective subcategory $\sigma\dashv i\colon \Sep(\srG_2,\neg\neg)\hookrightarrow \widehat{\srG}_2$  (see Chapter \ref{C:Simple})

\section{ Representation of Symmetric X-Graphs}
Each simple symmetric $X$-graph $G$ is the image of a representable arc object $\underline A$ in a category of $(X,M)$-graphs under a left adjoint functor $F$. Let $n=\#G(V)$ and consider the representable symmetric $n$-graph $\underline A$. Throughout this chapter, we will write $\mathbf{Graph}\defeq \Sep(\srG_2,\neg\neg)$.

\section{Generic Decks and Shuffling}

To each $(X,M)$-graph $G$

Let $\sigma\colon n\to n$ be a fixed-point free automap. Define the morphism of theories $\bar \sigma\colon \sG_n\to \sG_{n-1}$ 

\section{Functorial Graph Reconstruction}

\section{Using Intuitionistic Logic}

\section{Module Classifier}

Functorial description

Inner module classifier

Outer module classifier
\\\\
$M\defeq\ $\framebox{ 
					$\xymatrix@R=1.3em@C=1em{ \bullet \ar@{-}@(ul,ur)[]\ar@{-}[r] & \text{out} \ar@{-}@(ul,ur)^{} \ar@{-}[r] & \text{in} \ar@{-}[r] \ar@{}@(ul,ur) & \bullet  }$
				}

\part{Categories of Hypergraphs}

\chapter{The Complete Picture}
		We have made the argument in the main text that categories of $(X,M)$-graphs provide a nice categorical environment to do graph theory and uniform hypergraph theory. One disadvantage of using $(X,M)$-graphs is that the incidence is bounded by the cardinality of $X$. We step outside the main theory of $(X,M)$-graphs to examine and characterize the various types of hypergraphs the exist in the literature. 
		We show that the various categories of hypergraphs used in mathematics are determined by five main factors: (1) Does it have strict morphisms? (2) Do the objects allows multiplicities of vertices in the incidence of edges? (3) Do the objects have non-empty set of incidence vertices? (4) Do the objects have oriented incidence of edges? (5) Are the incidence of edges of objects bounded by some fixed cardinality?
		Based on these factors, thirty-two categories of hypergraphs are defined and the topos and adhesive structures are examined in each. It is shown that a change to an affirmative answer to any of the above factors is connected to the existence of a right adjoint functor. Moreover, definitions of k-uniform hypergraphs are given for each and shown to embed in certain presheaf categories which exhibit a better categorical environment for constructions used in mathematics.

\section{The Five Factors}

\ignore{Hypergraphs and k-uniform hypergraphs are general mathematical objects which have
a wide array of applications in combinatorics, game theory, computer science, logic, statistical physics, computational biology, network systems among others. Categorical methods using graphs and hypergraphs have been employed in graph theory to define magnitude of a graph \cite{tL}, category theory foundations (Ernst, The prospect of Doing what ...) graph and hypergraph transformations (Bauderon)... category of patches in (Mimram), ...  However, there is a wealth of related, yet distinct objects which bear the name hypergraph and often it is unclear what structure is relevant for a proper definition of morphism required for the definition of a category. Thus it is imperative to clarify the issue regarding the categorical structures involved and make a distinction in terminology.}

We isolate five main factors which determine a category of hypergraph: (1) Does it have strict morphisms? (2) Do the objects allows multiplicities of vertices in the incidence of edges? (3) Do the objects have non-empty set of incidence vertices? (4) Do the objects have oriented incidence of edges? (5) Are the incidence of edges of objects bounded by some fixed cardinality? By isolating these factors, we introduce new categories which are relevant for functorial constructions between categories of hypergraphs. We have addressed the poor categorical properties of k-uniform hypergraphs by introducing categories of $(X,M)$-graphs. We thus obtain a complete picture of graph and hypergraph theory in a categorical context, where each category has adjoint contructions to others. As a result, we obtain a complete picture of the functorial relations between categories of hypergraphs and k-uniform hypergraphs which allows us to transfer topos and adhesive categorical structure.

It is of general consensus that a hypergraph has a set of edges and a set of vertices such that each edge is associated to a set or a structured set of vertices incident to it. The difference between the various definitions of a category of hypergraphs in the literature is explained by isolated the five factors above. 
\begin{Ienum}
	\item Types of Edges
	\begin{aenum}
		\item \tbf{Empty Edges}: An edge may not be incident to any vertex.
		\item \tbf{Inhabited Edges}: Each edge must contain at least one vertex in its set of incidence vertices.
	\end{aenum}
	\item Types of Morphisms 
	\begin{aenum}
		\item \tbf{Lax}: A morphism $f$ is allowed to map a set of vertices incident to an edge $e$ onto a proper subset of the image $f(e)$.
		\item \tbf{Strict}: A morphism $f$ must take the set of vertices incident to an edge $e$ onto the vertices incident to the image $f(e)$.
	\end{aenum}
	\item Types of Incidence
	\begin{aenum}
		\item \tbf{No Multiplicity}: An edge must contain at most a single instance of a vertex in its incidence.
		\item \tbf{Multiplicity}: An edge may contain several copies of a vertex in its incidence.
	\end{aenum}
	\item Structure of Incidence
	\begin{aenum}
		\item \tbf{Unoriented Incidence}: An edge has a structureless set of incidence.
		\item \tbf{Oriented Incidence}: An edge has a linear order or listing of vertices in its incidence.
	\end{aenum}
	\item Size of Incidence
	\begin{aenum}
		\item \tbf{Unbounded Incidence}: For each cardinal $\alpha$ there is a hypergraph which has a set of vertices incident to an edge with cardinality greater than $\alpha$.
		\item \tbf{Bounded Incidence}: An edge has at most $\alpha$ vertices incident to it where $\alpha$ is a fixed cardinal number.
	\end{aenum}
\end{Ienum}
We will associate a $0$ to properties I(a), II(a), III(a), IV(a), and V(a) and a $1$ to properties I(b), II(b), III(b), IV(b), and V(b). It is our goal in this paper to define and examine the category of hypergraph associated to each five-digit binary number. 

We start with the definition of a hypergraph found in \cite{} where a hypergraph $H=(H(E),H(V))$ consists of a set of edges $H(E)$, a set of vertices $H(V)$, and an incidence map $\varphi\colon H(E)\to \C P(H(V))$ where $\C P\colon \mathbf{Set}\to \mathbf{Set}$ is the covariant powerset functor. A morphism of hypergraphs $f\colon H\to H'$ \tbf{FINISH THIS}  This is the category of hypergraph given in \cite{}. Thus this category of hypergraph allows empty edges, has strict morphisms, does not allow multiplicity of edges, has unoriented and unbounded incidence. In other words, it is associated to the binary number $01000$.

Recall in Chapter \ref{} (Example \ref{}), a bipartite graph $G=(G(V_1),G(V_2), G(A))$ consists of a set of $V_1$-vertices, a set of $V_2$-vertices, a set of arcs $G(A)$ and two set maps $\sigma\colon G(A)\to G(V_1)$ and $\tau\colon G(A)\to G(V_2)$ called source and target incidence respectively. There is a well-known representation of a hypergraph as a bipartite graph. Since the category of bipartite graphs is a category of presheaves, it have particularly nice properties. This is the definition of hypergraph used in \cite{}. It is associated to the binary number $00100$.

The representation of a hypergraph as a bipartite graph is well-known to extend to a faithful functor $i\colon \C H\to \C B$ (cf, \cite{}, \cite{}, \cite{}). Conversely, the obvious way to associate a bipartite graph to a hypergraph is not functorial. We address this issue in two ways. First, we define a lax morphism of hypergraphs. The resulting category of lax hypergraphs $\C H_{\text{\tiny lax}}$, associated to the binary digit $00000$, provides a factorization 
\[
\xymatrix{}
\] 
such that the hypergraph construction from a bipartite graph extends to a left adjoint functor $r_{\text{\tiny lax}}\dashv m_{\text{\tiny lax}}\colon \C H_{\text{\tiny lax}}\to \C B$. The embedding $m\colon \C H_{\text{\tiny lax}}\to \C B$ fails to be an equivalence because bipartite graphs allow multiplicities of vertices in the incidence of edges. We have shown in Proposition \ref{} , the category of lax hypergraphs is equivalent to the category of $\neg\neg$-separated bipartite graphs and thus has limits, colimits, exponentials, and a regular subobject classifier.  

A second way to address the issue above, is to restrict the morphisms in bipartite graphs to only those which are strict. The category of strict bipartite graphs, associated to the binary digit $01100$, also provides a factorization 
\[\]

\tbf{INHABITED GRAPHS, ORIENTED GRAPHS short intro uses}

The main structural result in this section is that changing a binary digit associated to a category from a $0$ to a $1$ corresponds to a right adjoint functor. Or equivalently, changing a digit from a $1$ to a $0$ corresponds to a left adjoint functor. This is captured by the notation $0\dashv 1$. 

\begin{table}[htbp]
	\centering
	\begin{tabular}{|c || c | c | c | c | c |} 
		\hline 
		Cat. & IE & Str & Mult & Orient & Bound \\ [0.5ex] 
		\hline\hline
		$\C H_{\text{\normalfont lax}}$ & 0 & 0 & 0 & 0 & 0     \\
		$\C H_{\text{\normalfont lax}}^{\alpha}$ & 0 & 0 & 0 & 0 & 1 \\ 
		$\oH_{\text{\normalfont lax}}$ & 0 & 0 & 0 & 1 & 0    \\
		$\oH_{\text{\normalfont lax}}^\alpha$ & 0 & 0 & 0 & 1 &  1   \\
		$\C B$ & 0 & 0 & 1 & 0 & 0 \\
		$\C B^\alpha$ & 0 & 0 & 1 & 0 & 1 \\
		$\oB$ & 0 & 0 & 1 & 1 & 0 \\
		$\oB^\alpha$ & 0 & 0 & 1 & 1 & 1 \\
		$\C H$ & 0 & 1 & 0 & 0  & 0 \\
		$\C H^\alpha$ & 0 & 1 & 0 & 0  & 1 \\
		$\oH$ & 0 & 1 & 0 & 1 & 0   \\
		$\oH^\alpha$ & 0 & 1 & 0 & 1 & 1   \\
		$\C B_{\text{\normalfont str}}$ & 0 & 1 & 1 & 0 & 0  \\
		$\C B_{\text{\normalfont str}}^\alpha$ & 0 & 1 & 1 & 0 & 1  \\
		$\oB_{\text{\normalfont str}}$ & 0 & 1 & 1 & 1 & 0 \\	
		$\oB_{\text{\normalfont str}}^\alpha$ & 0 & 1 & 1 & 1 & 1 \\
		$\C H_{\text{\normalfont lax}}^+$ & 1 & 0 & 0 & 0 & 0  \\
		$\C H_{\text{\normalfont lax}}^{+\alpha}$ & 1 & 0 & 0 & 0 & 1  \\
		$\oH_{\text{\normalfont lax}}^+$ & 1 & 0 & 0 & 1 & 0  \\	
		$\oH_{\text{\normalfont lax}}^{+\alpha}$ & 1 & 0 & 0 & 1 & 1  \\	
		$\C B^+$ & 1 & 0 & 1 & 0 & 0 \\
		$\C B^{+\alpha}$ & 1 & 0 & 1 & 0 & 1 \\
		$\oB^+$ & 1 & 0 & 1 & 1  & 0 \\
		$\oB^{+\alpha}$ & 1 & 0 & 1 & 1  & 1 \\
		$\C H^+$ & 1 & 1 & 0 & 0 & 0  \\
		$\C H^{+\alpha}$ & 1 & 1 & 0 & 0 & 1  \\
		$\oH^+$ & 1 & 1 & 0 & 1 & 0    \\
		$\oH^{+\alpha}$ & 1 & 1 & 0 & 1 & 1    \\
		$\C B_{\text{\normalfont str}}^+$ & 1 & 1 & 1 & 0 & 0 \\
		$\C B_{\text{\normalfont str}}^{+\alpha}$ & 1 & 1 & 1 & 0 & 1 \\
		$\oB_{\text{\normalfont str}}^+$ & 1 & 1 & 1 & 1 & 0 \\
		$\oB_{\text{\normalfont str}}^{+\alpha}$ & 1 & 1 & 1 & 1 & 1 \\
		[1ex] 
		\hline
	\end{tabular}
	\caption{The Thirty-two Categories of Hypergraphs}
\end{table}

\[
\xymatrix{ \C H^+ \ar@{>->}[r]_\top \ar@<+.75em>@{<-}[r]_\top \ar@<-.75em>@{<-}[r] \ar@<-.4em>@{>->}[d]^\dashv \ar@<+.4em>@{<-}[d] & \C H  \ar@<-.4em>@{>->}[d] \ar@<+.4em>@{<-}[d]_{\dashv} \ar@{>->}@<+.4em>[r]_-\top \ar@{<-}@<-.4em>[r] & \C B_{\text{\normalfont str}} \ar@<-.4em>@{>->}[d]^{\dashv} \ar@<+.4em>@{<-}[d] \ar@<+.75em>[r]_\top \ar@<-.75em>[r] \ar@{<-<}[r]_\top  & \C B_{\text{\tiny str}}^+ \ar@<-.4em>@{>->}[d]^\dashv \ar@<+.4em>@{<-}[d] \\ \C H_{\text{\tiny lax}}^+ \ar@{>->}[r]_\top \ar@<+.75em>@{<-}[r]_\top \ar@<-.75em>@{<-}[r] &
	\C H_{\text{\normalfont \tiny lax}} \ar@{>->}@<+.4em>[r]_-\top \ar@{<-}@<-.4em>[r] & \C B  \ar@<+.75em>[r]_\top \ar@<-.75em>[r] \ar@{<-<}[r]_\top & \C B^+
} 
\qquad \quad
\xymatrix{ \oH^+ \ar@{>->}[r]_\top \ar@<+.75em>@{<-}[r]_\top \ar@<-.75em>@{<-}[r] \ar@<-.4em>@{>->}[d]^\dashv \ar@<+.4em>@{<-}[d] & \oH  \ar@<-.4em>@{>->}[d] \ar@<+.4em>@{<-}[d]_{\dashv} \ar@{>->}@<+.4em>[r]_-\top \ar@{<-}@<-.4em>[r] & \oB_{\text{\normalfont str}} \ar@<-.4em>@{>->}[d]^{\dashv} \ar@<+.4em>@{<-}[d] \ar@<+.75em>[r]_\top \ar@<-.75em>[r] \ar@{<-<}[r]_\top  & \oB_{\text{\tiny str}}^+ \ar@<-.4em>@{>->}[d]^\dashv \ar@<+.4em>@{<-}[d] \\ \oH_{\text{\tiny lax}}^+ \ar@{>->}[r]_\top \ar@<+.75em>@{<-}[r]_\top \ar@<-.75em>@{<-}[r] &
	\oH_{\text{\normalfont \tiny lax}} \ar@{>->}@<+.4em>[r]_-\top \ar@{<-}@<-.4em>[r] & \oB  \ar@<+.75em>[r]_\top \ar@<-.75em>[r] \ar@{<-<}[r]_\top & \oB^+
} 
\]
\[  
\xymatrix{ \C H^{+\alpha} \ar@{>->}[r]_\top \ar@<+.75em>@{<-}[r]_\top \ar@<-.75em>@{<-}[r] \ar@<-.4em>@{>->}[d]^\dashv \ar@<+.4em>@{<-}[d] & \C H^{\alpha}  \ar@<-.4em>@{>->}[d] \ar@<+.4em>@{<-}[d]_{\dashv} \ar@{>->}@<+.4em>[r]_-\top \ar@{<-}@<-.4em>[r] & \C B_{\text{\normalfont str}}^{\alpha} \ar@<-.4em>@{>->}[d]^{\dashv} \ar@<+.4em>@{<-}[d] \ar@<+.75em>[r]_\top \ar@<-.75em>[r] \ar@{<-<}[r]_\top  & \C B_{\text{\tiny str}}^{+\alpha} \ar@<-.4em>@{>->}[d]^\dashv \ar@<+.4em>@{<-}[d] \\ \C H_{\text{\tiny lax}}^{+\alpha} \ar@{>->}[r]_\top \ar@<+.75em>@{<-}[r]_\top \ar@<-.75em>@{<-}[r] &
	\C H_{\text{\normalfont \tiny lax}}^{\alpha} \ar@{>->}@<+.4em>[r]_-\top \ar@{<-}@<-.4em>[r] & \C B^{\alpha}  \ar@<+.75em>[r]_\top \ar@<-.75em>[r] \ar@{<-<}[r]_\top & \C B^{+\alpha}
} 
\qquad \quad
\xymatrix{ \oH^{+\alpha} \ar@{>->}[r]_\top \ar@<+.75em>@{<-}[r]_\top \ar@<-.75em>@{<-}[r] \ar@<-.4em>@{>->}[d]^\dashv \ar@<+.4em>@{<-}[d] & \oH^{\alpha}  \ar@<-.4em>@{>->}[d] \ar@<+.4em>@{<-}[d]_{\dashv} \ar@{>->}@<+.4em>[r]_-\top \ar@{<-}@<-.4em>[r] & \oB_{\text{\normalfont str}}^{\alpha} \ar@<-.4em>@{>->}[d]^{\dashv} \ar@<+.4em>@{<-}[d] \ar@<+.75em>[r]_\top \ar@<-.75em>[r] \ar@{<-<}[r]_\top  & \oB_{\text{\tiny str}}^{+\alpha} \ar@<-.4em>@{>->}[d]^\dashv \ar@<+.4em>@{<-}[d] \\ \oH_{\text{\tiny lax}}^{+\alpha} \ar@{>->}[r]_\top \ar@<+.75em>@{<-}[r]_\top \ar@<-.75em>@{<-}[r] &
	\oH_{\text{\normalfont \tiny lax}}^{\alpha} \ar@{>->}@<+.4em>[r]_-\top \ar@{<-}@<-.4em>[r] & \oB^{\alpha}  \ar@<+.75em>[r]_\top \ar@<-.75em>[r] \ar@{<-<}[r]_\top & \oB^{+\alpha}
} 
\]
and adjoint relations $U\dashv \omega\colon \C A\to \text{o}\C A$ (respectively, $i\dashv \C A\colon \C A\to \C A^{\alpha}$) where $\C A$ is any of the unoriented  category (respectively, unbounded category) above.

We have already discussed adjunctions between $k$-uniform hypergraphs, hypergraphs, and $(X,M)$-graphs in Chapter \ref{S:MainResults}. It is easy to extend these results to an adjunction between oriented $X$-graphs, oriented $k$-hypergraphs, and oriented hypergraphs as well. Thus we obtain a complete picture of the adjoint relationships between the various types of categories of hypergraphs. 

For a cardinal number $k$, a hypergraph $H$ is $k$-uniform provided the cardinality of $\varphi(e)$ is $k$ for each edge $e$ in $H$ (cf. \cite{}). The full subcategory of $\C H$ consisting of $k$-uniform hypergraphs $\kH$ is not well-behaved. It lacks connected colimits, exponentials, and even though it is complete, the inclusion functor $j\colon \kH\to \C H$ is not continuous. We address this issue by showing the inclusion functor $j\colon \kH\to \C H$ factors through a category of presheaves $\widehat{\DD G}$
\[
\xymatrix@!=.3em{ & \widehat{\DD G} \ar[dr]^{h} & \\ \kH \ar@{>->}[ur]^{i} \ar@{>->}[rr]^j && \C H}
\]
such that $i\colon \kH\hookrightarrow \widehat{\DD G}$ is epi-reflective and $h$ is a left adjoint.

\tbf{Finish this.}

\ignore{
\section{Categorical Notions}
\subsection{Basic Notions}
Limits, colimits, exponentials, Adjoint Situations, epi-reflective adjunction, comma category, coalgebra (Show that for any category $\C A$ with  a terminal object  $1$, $\C A\ls F$ is equivalent to coalgebra over $\overline F\colon \C A\x \C A\to \C A\x \C A$ where $\overline F(a,b)=(F(b),1)$.)

The symbols and notation in this section follow from \cite{eR}. 

\subsection{Categories of Presheaves}
Define category of presheaves, Yoneda.

Let $I\colon \DD T \to \C M$ be functor from a small category $\DD T$ to a cocomplete category $\C M$. Since the Yoneda embedding $y\colon \C T\to \widehat{\DD T}$ is the free cocompletion of a small category there is a essentially unique adjunction $R\dashv N\colon \C M\to \widehat{\DD T}$, called the nerve realization adjunction, such that $Ry\iso I$. 
\[
\xymatrix{ \DD T \ar[dr]_I \ar[r]^{y} & \widehat{\DD T} \ar@<-.4em>[d]_{R}^{\dashv} \ar@{<-}@<+.4em>[d]^{N}\\ & \C M }
\]
The nerve and realization functors are given on objects by $N(m)=\C M(I(-), m),\ R(X)=\colim_{(c,\varphi)\in\int F} I(c)$ respectively,
where $\int F$ is the category of elements of $X$ (\cite{hA}, Section 2, pp 124-126).\footnote{In \cite{hA}, the nerve functor is called the singular functor.} The full subcategories of fixed points of this adjunction we denote by $\Cl(\widehat{\DD T})\hookrightarrow \widehat{\DD T}$ and $\Op(\C M)\hookrightarrow \C M$.\footnote{The object classes are given by $\Cl(\widehat{\DD T})=\setm{P\in \widehat{\DD T}}{\eta_P\text{ is an isomorphism}}$ and $\Op(\C M)=\setm{m\in \C M}{\varepsilon_m\text{ is an isomorphism}}$ where $\eta\colon \Id_{\widehat{\DD T}}\Rightarrow NR$ is the unit and $\varepsilon\colon RN\to \Id_{\C M}$ is the counit. Note that the restriction of $R\dashv N$ to fixed points is an adjoint equivalence between $\Cl(\widehat{\DD T})$ and $\Op(\C M)$. } 

We call a functor $I\colon \DD T \to \C M$ from a small category to a cocomplete category an interpretation functor. The category $\DD T$ is called the theory for $I$ and $\C M$ the modeling category for $I$. An interpretation $I\colon \DD T\to \C M$ is dense, i.e., for each $\C M$-object $m$ is isomorphic to the colimit of the diagram
$I\ls m \to \C M,\ (c,\varphi)\mapsto I(c),$
iff the nerve $N\colon \C M\to \widehat{\DD T}$ is full and faithful (\cite{sM}, Section X.6, p 245). When the right adjoint (resp. left adjoint) is full and faithful we call the adjunction reflective (resp. coreflective).\footnote{since it implies $\C M$ is equivalent to a reflective (resp. coreflective) subcategory of $\widehat{\DD T}$} 

We are interested in when the nerve also preserves any exponentials which exist. For the purpose of this paper, we show that if an interpretation is dense, full and faithful, then the nerve not only preserves limits, but also any exponentials which exist.

\begin{lemma}\label{L:FFInt}
	An interpretation functor $I\colon \DD T \to \C M$ is full and faithful iff $\underline c\defeq y(c)$ is a $NR$-closed object for each $\DD T$-object $c$, i.e., the unit $\eta_{\underline c}\colon \underline c\to NR(\underline c)$ at component $\underline c$ is an isomorphism.
\end{lemma}
\begin{proof}
	The unit of the adjunction $\eta_G$ is defined as the following composition
	\[
	\xymatrix{ G \ar[r]^-\varphi_-{\iso} & \widehat{\DD T}(y(-), G) \ar[r]^-{R_{(y,G)}} & \C M(Ry(-),R(G)) \ar[r]^-\psi_-{\iso} & \C M(I(-), R(G))=NR(G)}, 
	\]
	where $\varphi$ is given by Yoneda, $R_{(y,G)}$ is the map of homsets given by application of $R$, and $\psi$ is precomposition by the isomorphism $I\iso Ry$. For a representable, $\underline c$, there is an isomorphism $\rho\colon \C M(I(-), R(\underline c))\to \C M(I(-),I(c))$ by postcomposition by the isomorphism $I\iso Ry$. Thus $\rho\circ \psi\circ R_{(y,G)}$ evaluated at $\DD T$-object $c'$ takes a $\DD T$-morphism $f\colon c'\to c$ to $I(f)\colon I(c')\to I(c)$. Thus $I$ is full and faithful iff $\eta_{\underline c}$ is an isomorphism.
\end{proof}

\begin{prop}\label{P:Exponential}
	If an interpretation functor $I\colon \DD T\to \C M$ is dense, full and faithful, then  $R\dashv N$ is reflective and $N$ preserves any exponentials that exist in $\C M$. 
\end{prop}
\begin{proof}
	Suppose $G$ and $H$ are $\C M$-objects such that the exponential $G^H$ exists in $\C M$. Since $I$ is assumed to be full and faithful, by Lemma \ref{L:FFInt} above,  $\underline c\iso NR(\underline c)$ for each $\DD T$-object. Thus we have the following string of natural isomorphism:
	\begin{align*}
	N(G^H)(c) &\iso \C M(R(\underline c)\x H, G) && \text{(Yoneda, $R\dashv N$, exponential adjunction})\\
	&\iso\widehat{\DD T}(NR(\underline c)\x N(H), N(G)) && \text{($N$ is full and faithful, preserves limits)}\\
	&\iso \widehat{\DD T}(\underline c\x N(H), N(G)) &&\text{($\underline c$ is $NR$-closed)}\\
	&\iso N(G)^{N(H)}(c) && \text{(Exponential adjunction and Yoneda)}.
	\end{align*}
	Since the right-action structures are determined by Yoneda,  $N(G^H)\iso N(G)^{N(H)}$ in $\widehat{\DD T}$.
\end{proof}

\subsection{The Topos Structure of a Category of Presheaves}\label{S:Topos}

Double Negation, Quasi-topos, etc.

\subsection{Adhesive and Quasi-adhesive Categories}
Closure properties of adhesive categories
\begin{prop}
	
\end{prop}

Closure properties of HLR adhesive categories

A grothendieck quasi-topos is an adhesive category iff it preserves binary unions.

}

\section{The Categories of Hypergraphs}
\ignore{ 
	There are various definitions of hypergraphs in the literature. 
Our definition follows \cite{}. The definition given by bipartite graphs is in \cite{} for the purpose of graph rewriting. The definition given by inhabited hypergraphs is found in (). inhabited and unexposed hypergraphs \tbf{Finish this!!! Stress Hypergraph and Bipartite Graph as Fundamental, eg, 01000 and 00100}
}

\subsection{Hypergraphs and Bipartite Graphs}
A hypergraph $H=(H(V),H(E))$ consists of a set of vertices $H(V)$, a set of edges $H(E)$ and an incidence map $\varphi\colon H(E)\to \C P(H(V))$ where $\C P\colon \mathbf{Set}\to \mathbf{Set}$ is the covariant power-set functor (\cite{}). Given hypergraphs $H=(H(V),H(E),\varphi)$ and $H'=(H'(V),H'(E),\varphi')$ a morphism $f\colon H\to H'$ consists of a set map $f_V\colon H(V)\to H'(V)$ and a set map $f_E\colon H(E)\to H'(E)$ such that $\varphi'\circ f_E=\C P(f_V)\circ \varphi$. In other words, the following diagram commutes.
\[
	\xymatrix{ H(E) \ar[d]_{\varphi} \ar[r]^{f_E} & H'(E) \ar[d]^{\varphi'} \\ \C P(H(V)) \ar[r]^-{\C P(f_V)} & \C P(H'(V))}
\]
Thus the category of hypergraphs $\C H$ is the comma category $\mathbf{Set}\ls \C P$ (\cite{sM}). This category can be shown to be equivalent to a category of coalgebras on an endofunctor $F\colon \mathbf{Set}\x \mathbf{Set}\to \mathbf{Set}\x \mathbf{Set}$.\footnote{Take $F(V,E)\defeq (1,\C P(V))$ where $1$ is the terminal set (cf. \cite{cJ}).} 

\begin{proposition}
	The category of hypergraphs is a complete, cocomplete, adhesive category. 
\end{proposition}
\begin{proof}
	Cocomplete due to \cite{}
	
	Complete. Description of Products. Equalizers
	
	Adhesive due to \cite{}.
\end{proof}

A bipartite graph $G=(G(V_1),G(V_2), G(A))$ consists of a set of $V_1$-vertices $G(V_1)$, a set of $V_2$-vertices $G(V_2)$, a set of arcs $G(A)$ and two set maps $\sigma\colon G(A)\to G(V_1)$ and $\tau\colon G(A)\to G(V_2)$
called source and target incidence respectively. For bipartite graphs $G=(G(V_1),G(V_2), G(A), \sigma, \tau)$ and $G'=(G(V_1),G(V_2), G(A), \sigma, \tau)$, a morphism $g\colon G\to G'$ consists of set maps $g_{V_1}\colon G(V_1)\to G'(V_1)$, $g_{V_2}\colon G(V_2)\to G'(V_2)$ and $g_A\colon G(A)\to G'(A)$ such that the following diagram commutes
\[
	\xymatrix{ G(V_1) \ar[d]_{g_{V_1}} & G(A) \ar[d]_{g_A} \ar[r]^{\tau} \ar[l]_{\sigma} & G(V_2) \ar[d]^{g_{V_2}} \\ G'(V_1) & G'(A) \ar[r]^{\tau'} \ar[l]_{\sigma'} & G'(V_2)}
\]
Thus the category of bipartite graphs $\C B$ is the category of presheaves on $\xymatrix@!=.2em{V_1 \ar[r]^s & A & \ar[l]_{t} V_2}$. For an arc $a$ in the bipartite graph $G$, we will use the right-action notation of a presheaf  $a.s\defeq \sigma(a)$ and $a.t\defeq\tau(a)$ (\cite{}\tbf{Reyes}).

\begin{proposition}
	The category of bipartite graphs is an adhesive category. 
\end{proposition}
\begin{proof}
	Set is an adhesive category. Presheaf by \cite{}.
\end{proof}

A hypergraph $H=(H(V), H(E), \varphi)$ has a bipartite graph representation $i(H)$ where the set of $V_1$-vertices is $G(V)$, the set of $V_2$-vertices is the set $G(E)$, the set of arcs is given by \\$\setm{(v,e)\in H(V)\x H(E)}{v\in \varphi(e)}$ with right-actions $(v,e).s=v$ and $(v,e).t=e$ (\cite{}). 

\[
	\tbf{PUT EXAMPLE OF hypergraph to bipartite graph}
\]

Given a morphism $f\colon H\to H'$ of hypergraphs, there is a bipartite graph morphism $i(f)\colon i(H)\to i(H')$ such that $i(f)_{V_1}\defeq f_V$, $i(f)_{V_2}\defeq f_E$ and for $(v,e)\in H(A)$, we define $i(f)_A(v,e)\defeq (f_{V}(v),f_E(e))$. It is straightforward to verify this assignment defines a faithful functor $i\colon \C H\to \C B$ (cf., \cite{wD}, Proposition 3.2). 

Conversely, for a bipartite graph $G$ there is a natural way to associate it to a hypergraph. We define $r(G)\defeq (G(V_1),G(V_2),\varphi_G)$ where $\varphi_G\colon G(V_2)\to \C P(G(V_1))$ takes a $V_2$-vertex $e$ to the set \\$\setm{v\in G(V_1)}{\exists a\in G(A),\ \sigma(a)=v\text{ and } \tau(a)=e}$.
\[
\tbf{PUT EXAMPLE OF Bipartite to hypergraph}
\]
However, in this case the definition of $r$ does not extend to functor. Consider the hypergraphs $L$ and $P$ where $L$ has one $V_1$-vertex $v$, one $V_2$-vertex $e$ and one arc $a$ connecting them, and $P$ has two $V_1$-vertices $v_1$ and $v_2$, one $V_2$-vertex $e$ and two arcs $a_s$ and $a_t$ such that source incidents of $a_s$ and $a_t$ are $v_1$ and $v_2$ respectively. Let $g\colon L\to P$ be the bipartite graphs morphism with takes the arc $a$ to $a_s$, i.e., it is given by the following commuting diagram
\[
\xymatrix{ L \ar[d]_{g}  & \{v\} \ar[d]_{\named{v_1}}  & \{a\} \ar[r] \ar[l] \ar[d]_{\named{a_s}} & \{e\}  \ar[d]^{\Id_e} \\ P & \{v_1,v_2\} & \{a_s,a_t\} \ar[r] \ar[l] & \{e\} }
\]
The associated hypergraph of $L$ is the loop hypergraph $r(L)$ with one vertex $v$, one edge $e$ and the incidence $\named{v}\colon \{e\}\to \C P(\{v\})$ which sends $e$ to $\{v\}$. The associated hypergraph of $P$ is the edge hypergraph with two vertices $v_1$ and $v_2$, one edge $e$ and incidence $\named{v_1,v_2}\colon \{e\} \to \C P(\{v_1,v_2\})$ which sends $e$ to the set $\{v_1,v_2\}$.

Consider the diagram
\begin{align}\label{D:Ob1}
	\xymatrix{ \{e\} \ar@{}[dr]|-{\text{\tiny{not com.}}} \ar[r]^{\Id_e} \ar[d]_{\named v} & \{e\} \ar[d]^{\named{v_1,v_2}} \\ \C P(\{v\}) \ar[r]^-{\C P(\named{v_1})} & \C P(\{v_1,v_2\})}
\end{align}
 Since $\{v_1\}\subsetneq \{v_1,v_2\}$ is a strict inclusion, the diagram does not commute. Thus, the assignment of $r$ does not extend to the morphism $g$. 
\[
\tbf{Put Diagrams of L and P and show obstruction}
\] 
The problem is that morphisms of hypergraphs preserve incidence of edges on the nose and morphisms of bipartite graphs allow subset inclusions. We call such morphisms strict and lax, respectively. Thus we should be careful not to identify hypergraphs with bipartite graphs when making categorical constructions, especially since the structure of a category comes entirely from the behavior of its morphisms. 

Even though $r$ does not lift to a functor, it is reasonable to ask whether $i$ admits a left or right adjoint. However, $i$ neither preserves limits nor colimits. Indeed, consider the hypergraph $H=(2,1,\top)$ where $\top\colon 1\to \C P(2)$ is the element $2$ in $\C P(2)$. Then $i(H)$ is the bipartite graph $P$ above. The product $P\x P$ has four $V_1$-vertices, one $V_2$-vertex, and four arcs connecting them.  The diagonal morphism of hypergraphs $\langle \Id_P,\Id_P \rangle\colon P\to P\x P$ embeds $P$ into $P\x P$ by taking $a_s$ to $(a_s,a_s)$ and $a_t$ to $(a_t,a_t)$. Since the diagonal morphism is lax, $i$ does not preserve binary products. 

The functor $i$ also does not preserve coequalizers. Consider the hypergraphs $H_V\defeq (1,\empset, !_{1})$, $H$ as above and the morphisms $\named v_s, \named v_t\colon H_V\to H$ which sends $\pt\in 1$ to $v_s$ and $v_t$ respectively. Since the forgetful functor $U\colon \C H\to \mathbf{Set} \x \mathbf{Set}$ creates colimits, the coequalizer
is the terminal object $H_1\defeq (1,1,\top)$ where $\top\colon 1\to \C P(1)$ is the element $1\in \C P(1)$. The bipartite graph $i(H_1)$ has one arc connecting one $V_1$-vertex to a $V_2$-vertex. However, the coequalizer of $i(\named{v_s}), i(\named{v_t})\colon i(H_V)\to i(H)$ is the bipartite graph with one $V_1$-vertex, one $V_2$-vertex, and two arcs connecting them. Thus $i$ does not admit a left nor right adjoint functor. 
 \subsection{Lax Hypergraphs}
 
 We address the obstruction related to $r$ extending to a functor by modifying the definition of hypergraph morphisms. Let $H$ and $H'$ be hypergraphs. A lax morphism of hypergraphs $f\colon H\to H'$ consists of set maps $f_V\colon H(V)\to H'(V)$ and $f_E\colon H(E)\to H'(E)$ such that for each edge $e$ in $H$, $\C P(f_V)\circ \varphi(e)\subseteq \varphi'\circ f_E$. The diagram for a lax morphism is given as follows.
 \[
	 \xymatrix{ H(E) \ar[r]^{f_E} \ar[d]_\varphi \ar@{}[dr]|-{\leq} & H'(E) \ar[d]^{\varphi'} \\ \C P(H(V)) \ar[r]^{\C P(f_V)} & \C P(H'(V))}
 \]
 Verification of the associativity and identity laws are straightforward, giving us a category of lax hypergraphs, $\C H_{\text{\normalfont lax}}$. 
 
 To see that the assignment of $r$ to lax hypergraphs lifts to a functor $\rho\colon \C H_{\text{\normalfont lax}} \to \C B$, observe that the obstruction depicted in Diagram $\ref{D:Ob1}$ above disappears since it satisfies the requirement to be a lax morphisms 
 \begin{align}
	\xymatrix{ \{e\} \ar@{}[dr]|-{\leq} \ar[r]^{\Id_e} \ar[d]_{\named v} & \{e\} \ar[d]^{\named{v_1,v_2}} \\ \C P(\{v\}) \ar[r]^-{\C P(\named{v_1})} & \C P(\{v_1,v_2\})}
 \end{align}
Given a bipartite morphism $g\colon G\to G'$, we define the lax hypergraph morphism $\rho(g)\colon \rho(G)\to \rho(G')$ as set maps $\rho(g)_{V}\defeq g_{V_1}$ and $\rho(g)_{E}\defeq g_{V_2}$. Then given an edge $e$ in $\rho(G)$, $\C P(g_{V_1})(\varphi(e))\subseteq \varphi' (g_{V_2}(e))$. Indeed, for $w\in \C P(g_{V_1})(\varphi(e))$, there exists a $V_1$-vertex $v$ in $G$ such that $v\in \varphi(e)$ and $g_{V_1}(v)=w$, and there exists an arc $a$ in $G$ such that $\sigma(a)=g_{V_1}(v)=w$ and $\tau(a)=w$. Since $g$ is a bipartite graph morphism, $g_A(a)$ is an arc in $G'$ such that $\sigma'(g_A(a))=w$ and $\tau'(g_A(a))=g_{V_2}(e)$ and hence $w\in \varphi'(g_{V_2}(e))$. Composition and identity laws are easily verified.

The category of hypergraphs is a wide subcategory of lax hypergraphs.\footnote{Recall a wide subcategory is given by a faithful functor which is bijective on objects (\cite{eR}).} Moreover, the definition of $i\colon \C H\to \C B$ lifts to a functor $\mu\colon \C H_{\text{\normalfont lax}} \to \C B$ giving us the following factorization.
\[
\xymatrix@!=.1em{ &  \C H_{\text{\normalfont lax}} \ar[dr]^{\mu} & \\ 
	\C H \ar[rr]^{i} \ar@{>->}[ur]^{\text{\tiny wide}} && \C B}
\] 
The functor $\mu$ is full and faithful. Observe that for the bipartite graph $\mu(H)$ there is at most one arc connected a $V_1$-vertex $v$ to a $V_2$-vertex $e$ since this is the elementhood relation $v\in \varphi(e)$. If $f, f'\colon H\to H'$ are lax hypergraph morphisms such that $\mu(f)=\mu(f')$, then $\mu(f)_{V_1}=\mu(f')_{V_1}$ and $\mu(f)_{V_2}=\mu(f')_{V_2}$ and thus it is equal on arcs $\mu(f)_A=\mu(f')_A$. Similarly, any morphism $g\colon \mu(H)\to \mu(H')$ is determined by where it sends its vertices and thus is equal to $\mu(f)$ for $f=(g_{V_1},g_{V_2})$. We thus have the following result.

\begin{proposition}
	The category of lax hypergraphs is equivalent to the category of $\neg\neg$-separated bipartite graphs.
\end{proposition}
\begin{proof}
	
\end{proof}
Using the results in Section \ref{S:Topos}, the next two corollaries follow.
\begin{corollary}
	There is an epi-reflective adjunction $\rho \dashv \mu\colon \C H_{\text{\normalfont lax}}\hookrightarrow \C B$.
\end{corollary}
 \begin{corollary}
 	The category of lax hypergraphs is a Grothendieck quasi-topos.
 \end{corollary}
 
 Therefore, $\C H_{\text{\normalfont lax}}$ is complete and cocomplete, locally cartesian closed and has a regular subobject classifier. The limits and exponentials can be computed in bipartite graphs. Colimits  can be constructed by applying the reflector $r$ to the colimit in bipartite graphs. The regular subobject classifier is given but identifying multiplicities of incidence in the subobject classifier $\Omega$ of bipartite graphs. 
 \[
	 \tbf{PUT DIAGRAM HERE}
 \]

 Bipartite graphs outside the essential image of $m$ are said to have multiplicities of vertices in the incidence of edges. The reflector $r$ identifies these multiplicities.
 
  \subsection{Strict Bipartite Graphs}
 A second way to address the obstruction above is to restrict the class of morphisms of bipartite graphs. Let $G$ be bipartite graphs. For each $V_2$-vertex $e$ in $G$ we assign the set $G(e)\defeq \setm{a\in G(A)}{a.t=e}$. A strict bipartite graph morphism $g\colon G\to G'$ is a bipartite graph morphism such that for each $V_2$-vertex $e$ in $G$, the restriction of $g_A$ to $G(e)\to G'(g_{V_2}(e))$ is surjective. It is straightforward to show this definition assembles into the category of strict bipartite graphs, $\C B_{\text{\normalfont str}}$. 
 
 Observe that the obstruction depicted in Diagram $\ref{D:Ob1}$ is not constructible in the category of strict bipartite graphs since $L(e)=\{v\}$ and $P(g_s(e))=P(g_t(e))=\{v_1,v_2\}$.  In fact, $r$ lifts to the functor $\rho_{\text{\normalfont \tiny str}}\colon \C B_{\text{\normalfont str}}\to \C H$ as the restriction of $\rho$ to $\C B_{\text{\normalfont str}}$. 
This follows from $\C P(g_{V_1})(g_{V_1})= \varphi'(g_{V_2}(e))$ for each $e\in G(V_2)$ iff $G(e)\to G'(\varphi(e))$ is surjective for each $e\in G(V_2)$. Moreover, the epi-reflective adjunction $\rho\dashv \mu\colon \C H_{\text{\normalfont lax}} \hookrightarrow \C B$ restricts to an epi-reflective adjunction $\rho_{\text{\normalfont \tiny str}}\dashv \mu_{\text{\normalfont str}}\colon \C H\to \C B_{\text{\normalfont str}}$.

 Note that epimorphisms are strict morphisms and that strict morphisms are right cancellable, i.e., if $g\circ f$ and $f$ is strict, then so is $g$. This shows that $\C B_{=}$ is closed under coequalizers. It is also clear that it is closed under coproducts (which are disjoint). Therefore, $\C B_{=}$ is closed under formation of colimits in $\C B$. However, it is not closed under the formation of products. Consider the diagonal morphism of the bipartite graph $\langle \Id_G,\Id_G\rangle\colon G\to G\x G$ where $G$ is the bipartite graph with two $V_1$-vertices, one $V_2$-vertex and two arcs connecting them. It is also not closed under the formation of equalizers. Let $G'$ is the bipartite graph with one $V_1$-vertex, one $V_2$-vertex and two arcs and let $f\colon G'\to G'$ be the morphism which swaps arcs. The equalizer $E$ of $f$ and the identity is the bipartite graph with one $V_1$-vertex, one $V_2$-vertex and no arcs. Thus the universal morphism of the equalizer $E\hookrightarrow G'$ is not strict. 
 \[
	 \xymatrix{ & G \ar[dl]_{\Id_G} \ar[dr]^{\Id_G} \ar@{..>}[d]|-{\langle\Id_G,\Id_G \rangle} \ar@{}[d]|->>{\text{\tiny not str}} & \\ G & G\x G  \ar[l] \ar[r] & G}\qquad \qquad \xymatrix@R=.5em{  \\E \ar@{>->}[r]_{\text{\tiny not str}} & G' \ar@<+.3em>[r]^f \ar@<-.3em>[r]_{\Id_{G'}} & G'  }
 \]
  Therefore, the inclusion $\C B_{\text{\tiny str}}\hookrightarrow \C B$ does not admit a left adjoint since it does not preserve limits. However, the inclusion $j\colon \C B_{\text{\tiny str}}\hookrightarrow \C B$ does admit a right adjoint $\Sigma\colon \C B\to \C B_=$. Let $G$ be a bipartite graph. We define 
  \begin{align*}
	  \Sigma(G)(V_1) &\defeq G(V_1),\\
	  \Sigma(G)(V_2) &\defeq \setm{(e, X)}{e\in G(V_2),\ X\text{ is a subgraph of }G(e)},\\
	  \Sigma(G)(A) & \defeq \setm{(a,(e,X))\in G(A)\x \Sigma(G(V_2))}{a\text{ is an arc in }X}
 \end{align*}
The source and target incidence maps are given $\sigma(a,(e,X))=\sigma(a)$ and $\tau(a,(e,X))=(e,X)$. For a morphism $f\colon G\to G'$, we define the strict bipartite morphism $\Sigma(f)\colon \Sigma(G)\to \Sigma(G')$ such that 
\begin{align*}
\Sigma(f)_{V_1}& \defeq f_{V_1},\\ 
\Sigma(f)_{V_2} &\colon (e,X)\mapsto (f_{V_2}(e), f_!(X)) \\
\Sigma(f)_{A} &\colon (a,(e,X))\mapsto (f_A(a),(f_{V_2}(e),f_!(X)))
\end{align*}
where $f_!(X)$ is the image of subgraph $X$ under $f$. \tbf{SHOW TRIANGLE IDENTITIES}

The restriction of $\Sigma$ to lax hypergraphs factors through the category of hypergraphs and thus defines the right adjoint to the inclusion $j_{\C H}\colon \C H\to \C H_{\text{\tiny lax}}$. The complete picture of functoral relations between lax/strict hypergraphs and bipartite graphs is given in the following diagram
 \[
 \xymatrix{ \C H  \ar@<-.4em>@{>->}[d]_{j_{\C H}} \ar@<+.4em>@{<-}[d]^{\Sigma_{\C H}}_{\dashv} \ar@{>->}@<+.4em>[r]^{\mu_{\text{\normalfont str}}}_-\top \ar@{<-}@<-.4em>[r]_{\rho_{\text{\normalfont \tiny str}}} & \C B_{\text{\normalfont str}} \ar@<-.4em>@{>->}[d]_{j}^{\dashv} \ar@<+.4em>@{<-}[d]^{\Sigma} \\
 	\C H_{\text{\normalfont lax}} \ar@{>->}@<+.4em>[r]^-{\mu}_-\top \ar@{<-}@<-.4em>[r]_-{\rho} & \C B
 }
 \]
 where the inclusion $i\colon \C H\to \C B$ is the diagonal $j\circ\mu_{\text{\tiny str}}=\mu\circ j_{\C H}$.

 \ignore{ \tbf{TOPOS Properties Adhesive properties}}
 
 \subsection{Inhabited Hypergraphs and Bipartite Graphs}

Many definitions of hypergraphs include the requirement that hyperedges have non-empty incidence (\cite{}\tbf{PUT Reference to texts}). We define the category of inhabited hypergraphs $\C H^+$ as the full subcategory of $\C H$ with objects hypergraphs $H=(H(E), H(V))$ such that the incidence of each edge is non-empty. In other words, $\C H^+$ is equivalent to the comma category $\mathbf{Set}\ls \C P_+$, where $\C P_+\colon \mathbf{Set}\to \mathbf{Set}$ is the covariant functor which takes a set $X$ to the set of non-empty subsets of $X$ and a set map to the direct image function. 

\begin{proposition}
	The category of inhabited hypergraphs is adhesive.
\end{proposition}
\begin{proof}
	
\end{proof}

The inclusion functor $i\colon \C H^+\to \C H$ admits both a left and right adjoint functor. 

Thus the inclusion $i\colon \C H^+\to \C H$ creates all limits and colimits.  

$F\dashv i \dashv G\colon \C B \to \C B^+$

\subsection{Oriented Hypergraphs and Bipartite Graphs}

An oriented hypergraph $H=(H(V),H(E),\varphi)$ consists of a set of vertices $H(V)$, a set of edges $H(E)$, and a set map $\varphi\colon H(E)\to \LinOrd(H(V))$ where $\LinOrd\colon \mathbf{Set} \to \mathbf{Set}$ is the functor which assigns to a set $X$ the set of linear orders on subsets of $X$ and to a set map $f\colon X\to Y$ assigns the set map $\LinOrd(f)\colon \LinOrd(X)\to \LinOrd(Y)$ induced by the image powerset functor. Given oriented hypergraphs $H=(H(V),H(E),\varphi)$ and $H=(H'(V),H'(E),\varphi')$, a morphism $f\colon H\to H'$ consists of a set map $f_V\colon H(V)\to H'(V)$ and a set map $f_E\colon H(E)\to H'(E)$ such that $\varphi'\circ f_E=\LinOrd(f_V)\circ \varphi$.
\[
	\xymatrix{ H(E) \ar[r]^{f_E} \ar[d]_{\varphi} & H'(E) \ar[d]^{\varphi'}\\ \LinOrd(H(V)) \ar[r]^{\LinOrd(f_V)} & \LinOrd(H'(V)) }
\]
In other words, the category of oriented hypergraphs $\oH$ is the comma category $\mathbf{Set}\ls \LinOrd$.
\begin{proposition}
	The category of oriented hypergraphs is a complete, cocomplete, adhesive category. 
\end{proposition}
\begin{proof}
	Cocomplete due to \cite{}
	
	Complete. Description of Products. Equalizers
	
	Adhesive due to \cite{}.
\end{proof}

There is a faithful (non-full) functor $U\colon \oH\to \C H$ which forgets the orientation of the incidence of an edge. The forgetful functor admits a right adjoint $\omega\colon \C H \to \oH$ which assigns to a hypergraph $H$ the oriented hypergraph $\omega(H)$ which consists of the following structure
\begin{align*}
	\omega(H)(V) &\defeq H(V),\\
	\omega(H)(E) &\defeq \textstyle\bigsqcup_{e\in H(E)}\C O(\varphi_H(e)),\\
	\varphi_{\omega(H)}&\colon \omega(H)(E)\to \LinOrd(H(V)),\quad (e,<_e) \mapsto <_e  
\end{align*}
where $\C O(\varphi_H(e))$ is the set of linear orders on the set of vertices $\varphi_H(e)$ and $<_{e}$ is a linear order on $\varphi_H(e)$. Given a morphism of hypergraphs $f\colon H\to H'$  we define a morphism of hypergraphs $\omega (f)\colon \omega(H)\to \omega(H')$ where $\omega  (f)_E\colon \omega(H)(E)\to \omega(H')(E)$ takes $(e,<_e)$ to $(f_E(e),\LinOrd(f_V)(<_e))$.

The unit of the adjunction $\eta\colon \Id_{\oH} \Rightarrow \omega U$ is defined by \tbf{FINISH THIS}
The counit of the adjunction $\varepsilon\colon U\omega\Rightarrow \Id_{\C H}$ is defined by 

We may bound the size of incidence of the linear order $\mathbf{Set}\ls (-)^\alpha$ where $\alpha$ is a cardinal number and $(-)^\alpha\colon \mathbf{Set}\to \mathbf{Set}$ is the functor which assigns to a set $X$ the set of linear orders on subsets of cardinality less than $\alpha$.

\begin{proposition}
	Let $\alpha$ be an cardinal number. The comma category $\mathbf{Set}\ls (-)^\alpha$ is equivalent to a category of presheaves.
\end{proposition}
\begin{proof}
	
\end{proof} 

An oriented bipartite graph $G=(G(V_1), G(V_2), G(A), (<_e)_{e\in G(V_2)})$ consists of a bipartite graph equipped with a total order $<_e$ on $G(e)\defeq \setm{a\in G(A)}{a.t=e}$ for each $V_2$-vertex $e$. A morphism of oriented bipartite graphs $g\colon G\to G'$ is a bipartite graph morphism such that for each $V_2$-vertex $e$ in $G$ the restriction of $g_A$ to $G(e)\to G'(g_{V_2}(e))$ is a morphism of total orders. In other words, for each $a, a'\in G(e)$, $a\leq_e a'$ implies $g_A(a)\leq_{g_{V_2}(e)} g_A(a')$.

\section{Directed Hypergraphs and Bipartite Graphs}

In \cite{}, hypergraph refers to a directed hypergraph

\begin{proposition}
	The category of directed hypergraphs is adhesive.
\end{proposition}
\begin{proof}
	
\end{proof}

Similarly, a directed oriented hypergraph $H=(H(V),H(E),\varphi)$ consists of a set of vertices $H(V)$ a set of edges

Let $\doDDH$ be the category with set of objects $\setm{A_{(i,j)}}{i,j\in \DD N}+\{V\}$ with homsets $\doDDH(A_{(i,j)},A_{(i',j')})=\doDDH(A_{(i,j)},V)=\empset$ for $(i,j)\neq (i',j')$ and $\doDDH(V,A_{(i,j)})=\setm{s_k}{1\leq k\leq i}+\setm{t_k}{i+1\leq k\leq i+k}$.

\begin{proposition}
	The category of directed oriented hypergraphs is equivalent to the presheaf category on $\doDDH$.
\end{proposition}
\begin{proof}
	
\end{proof}
\begin{corollary}
	The category of directed oriented hypergraphs is an adhesive category.
\end{corollary}
\begin{proof}
	
\end{proof}

A directed bipartite graph a presheaf on the category 
\[
	\xymatrix{ V_1 \ar[r]^{s_1} \ar[d]_{s_2} & A_1 \ar@{<-}[d]^{t_1} \\ A_2 \ar@{<-}[r]_{t_2} & V_2 }
\] 

\begin{proposition}
The category of directed bipartite graphs is an adhesive category.	
\end{proposition}
\begin{proof}
	
\end{proof}

\ignore{
\section{The Categories of $k$-Uniform Hypergraphs}

Let $k$ be a cardinal number. A hypergraph $H$ is k-uniform provided for each edge $e$ in $H$, the cardinality of $\varphi(e)$ is $k$. Similarly, a bipartite graph $G$ is k-uniform provided for each $V_2$-vertex $e$ in $G$, the the cardinality of $G(e)\defeq\setm{a\in G(A)}{a.t=e}$ is $k$. Definitions of k-uniform lax hypergraphs and k-uniform strict bipartite graphs follow analogously. We denote the full subcategories of k-uniform hypergraphs, lax hypergraphs, strict bipartite graphs, and bipartite graphs by $\kH$, $\kH_{\text{\normalfont lax}}$, $\kB_{\text{\normalfont str}}$ and $\kB$ respectively. 

\subsection{Unoriented k-Uniform Hypergraphs}

\subsubsection{The Categories of k-Uniform Hypergraphs and Symmetric k-Graphs}
The category of k-uniform hypergraphs has a poor structure. It lacks connected colimits as well as equalizers. In \cite{}, it is shown it has binary products. Given k-uniform hypergraphs $G$ and $G'$, the product $G\x G'$ \tbf{FINISH THIS}. However, the full embedding $j\colon \kH\hookrightarrow \C H$ does not preserve binary products (cf, \cite{}). 

The problem is that $j$ forgets the cohesive preference of the cardinal number $k$ in the incidence of edges. We rectify this by introducing the category of Symmetric k-Graphs. Let $\sG_k$ be the category \tbf{finish this, copy from before} When $k=2$, this is the category of graphs with involution discussed in \cite{}.  

\begin{proposition}
	The fixed points of the adjunction $R\dashv N\colon \C H\to \kS$ is equivalent to the category of {\normalfont k}-uniform hypergraphs.
\end{proposition}
\begin{proof}
	
\end{proof}

When $k=2$ compare this to \cite{wD}, Proposition 3.3.

\begin{corollary}
	The full inclusion $j\colon \kH\hookrightarrow \C H$ factorizes 
\tbf{FINISH THIS}
such that $\kH\to \kS$ preserves binary products and $R\colon \kS\to \C H$ is cocontinuous.
\end{corollary}

\begin{proposition}
	Let $\tbf{Card}_{\text{inj}}$ denote the category of cardinal numbers with injective set maps as morphisms. There is a functor $(-)\C S\colon \tbf{Card}_{\text{inj}}\to \Cat$ which assigns the cardinal number $k$ to the small category $\kS$.
\end{proposition}
\begin{proof}
	
\end{proof}
\begin{corollary}
	Given a set inclusion $m\colon k\to r$ between cardinal numbers $k$ and $r$ there is an essential geometric morphism $m_!\dashv m^*\dashv m_*\colon \kS\to \rS$.	
\end{corollary}

\subsubsection{The Category $k$-Sympower Graphs}
Let $Y$ be a set. We define the symmetric $k$-power of $Y$, denoted $Y^{\underline k}$, as the multiple coequalizer of $(\sigma\colon Y^k \to Y^k)_{\sigma\in \Aut(Y)}$ where $\sigma$ is the $\sigma$-shuffle of coordinates. This definition extends to a functor $(-)^{\underline k}\colon \mathbf{Set} \to \mathbf{Set}$. The category of $k$-sympower graphs is defined as $\kG\defeq \mathbf{Set}\ls (-)^{\underline k}$. In other words, a $k$-sympower graph $G=(G(E),G(V),\varphi)$ consists of a set of edges $G(E)$, a set of vertices $G(V)$ and an incidence map $\varphi\colon G(E)\to G(V)^{\underline k}$. A morphism $f\colon G\to G'$ of $k$-sympower graphs consists of \tbf{finish this} 
Since $k$-sympower graphs can be expressed as a category of coalgebras,\footnote{Set $F\colon \mathbf{Set}\x\mathbf{Set}\to \mathbf{Set}\x \mathbf{Set}$, $(V,E)\mapsto (1,V^{\underline k})$.} the forgetful functor $U\colon \kG\to \mathbf{Set}\x\mathbf{Set}$ creates colimits. 

 To define an interpretation functor $I\colon \sG_k\to \kG$, let $1$ be the terminal set (i.e., a singleton set), $!_1\colon \empset \to 1$ the unique set morphism, $\named x\colon 1\to X$ the set map which picks out the element $x\in X$, and $\top\colon 1\to X^{\underline k}$ which sends $\pt\in 1$ to $(x)_{x\in X}$. We define $I\colon \sG_k\to \kG$ on objects by $V\mapsto (\empset,\ 1,\ !_1),$ and $A\mapsto (1,\ X,\ \top)$. On morphisms, we set
 \begin{align*} 
 & (x\colon V\to A) \quad \mapsto \quad \ I(x)\defeq (!_1,\named x)\colon (\empset,\ 1,\ !_1)\to (1,\ X,\ \top),\\
 & (\sigma\colon A\to A) \quad \mapsto \quad I(\sigma)\defeq (\Id_1, \sigma^{\underline k})\colon (1,X,\top)\to (1,X,\top)
 \end{align*}
 
 \begin{lemma}
 	The interpretation $I\colon \sG_k\to \kG$ is dense, full and faithful.
 \end{lemma}
 \begin{proof}
 	It is clearly full and faithful. To show it is dense, let $(E,V,\varphi)$ and $(K,L,\psi)$ be $\kG$-objects and $\lambda\colon D\Rightarrow \Delta(K,L,\psi)$ a cocone on the diagram $D\colon I\ls (E,V,\varphi)\to \kG$. 
 	Let $e$ be an edge in $E$ and $f\colon X\to V$ be the set morphism with $f^{\underline k}=\varphi(e)$. Then $(\named e, f)\colon I(A)=(1,X,\top)\to (E,V,\varphi)$ is an object in $I\ls (E,V,\varphi)$ and thus there is a morphism $\lambda_{(\named {e},f)}\eqdef(\named{e'},g)\colon D(\named e,f)=(1,X,\top)\to (K,L,\psi)$. By the compatibility of the cocone, this gives us a uniquely defined $h\colon E\to K$, $e\mapsto e'$ on edges. Similarly for each vertex $v\in V$, there is a morphism $(!_E,\named v)\colon I(V)=(\empset, 1,!_1)\to (E,V,\varphi)$ and a cocone inclusion $(!_K,\named w)\colon D(!_E,\named v)=(\empset, 1,!_1)\to (K,L,\psi)$ giving us a factorization on vertices $l\colon V\to L$. Since $\psi\circ h(e)=(lf)^{\underline k}\circ \top=l^{\underline k}\circ \varphi(e)$ for each edge $E$,  $(h,l)\colon (E,V,\varphi)\to (K,L,\psi)$ a well-defined $\kG$-morphism. Therefore, $I$ is dense.
 	
 \end{proof}
 
 Note that the realization functor takes a $\widehat{\DD G}_{(X,\sX)}$-object and quotients out the set of arcs by $\sX$. Hence the unit of the adjunction $\eta_P\colon P\to NR(P)$ is bijective on vertices and surjective on arcs. Hence the adjunction is $\Epi$-reflective. 
  
For a $\kG$-object $(B,C,\varphi)$, the embedding given by the nerve functor is given by 
\begin{align*}
N(B,C,\varphi)(V)&=\kG(I(V), (B,C,\varphi))\iso C,\\
N(B,C,\varphi)(A)&=\kG(I(A), (B,C,\varphi))=\setm{(e,g)\ }{\ e\in B,\ g\colon X\to C\ s.t.\ g^{\underline k}=\varphi(e)} 
\end{align*}
The  right-actions are by precomposition, i.e., $(e,g).x=(e,g\circ \named x)$, $(e,g).\sigma=(e,g\circ \sigma)$. 

Let us show that all loops in the full subcategory of $\widehat{\DD G}_{(X,\sX)}$ equivalent to $\kG$ are 1-loops. A loop in a $\kG$-object $(B,C,\varphi)$ is an edge $e\in B$ such that $\varphi(e)$ is $(v)_{x\in X}$ in $C^{\underline k}$ for some $v\in C$.  Therefore, there is only one morphism $(\named e, f)\colon I(A)\to (B,C,\varphi)$ and thus $(\named e,f\circ \sigma)=(\named e,f)$ for each $\sigma\in \sX$. Hence, each object in the reflective subcategory of $\widehat{\DD G}_{(X,\sX)}$ has only 1-loops.

\begin{corollary}\label{C:Exponentials}
	If $k$ is greater than $1$, the category $\kG$ does not have exponentials. 
\end{corollary}
\begin{proof}
	By the above observation, it is enough to show that there exist objects $G$ and $H$ in $\kG$ such that $N(G)^{N(H)}$ has a loop of involution in $\kS$. Set $H\defeq I(A)$ and $G$ be the graph with one vertex and an $\Aut(X)$-loop. Then $N(G)^{N(H)}=L^{\underline A}$ as defined in Example \ref{E:XLoops} which we have shown has a loop of involution.
\end{proof}

\[
	\xymatrix@R=3em@C=2em{ \sG_k \ar[rr]^y \ar[dr]_I && \kS \ar@<-.3em>[dl]_{R} \ar@<+.3em>@{<-<}[dl]^N \ar@<-.3em>[dr]_{R_{\text{\normalfont str}}} \ar@<+.3em>@{<-<}[dr]^{N_{\text{\normalfont str}}} &  \\ & \kG \ar[rr]^{\simeq} && \kB_{\text{\normalfont str}}}
\]

\begin{proposition}
	The categories $\kB_{\text{\normalfont str}}$ and $\kG$ do not have exponentials. 
\end{proposition}
\begin{proof}
	
\end{proof}

\subsection{Oriented k-Uniform Hypergraphs}

\subsubsection{The Category of k-Uniform Oriented Hypergraphs}

\subsubsection{The Category of Oriented k-Graphs}
}
\ignore{

\section{The Categories of Hypergraphs with Labels and Colorings}

\section{Mixed and Hybrid Structures}

A hybrid hypergraph $H=(H(E_1),H(E_2),H(V),\varphi)$ consists of a set of $E_1$-edges, a set of $E_2$-edges, a set of vertices, and an incidence map $\varphi\colon H(E_1)+H(E_2)\to \C P(H(V))$. The category of hybrid hypergraphs is defined as the comma category $(-)+(=)\ls \C P$ where $(-)+(=)\colon \mathbf{Set}\x \mathbf{Set}\to \mathbf{Set}$ is the functor which takes $(X,Y)$ to $X+Y$.

A hybrid hypergraph $H=(H(E_1),H(E_2),H(V),\varphi)$ consists of a set of $E_1$-edges, a set of $E_2$-edges, a set of vertices, and an incidence map $\varphi\colon H(E_1)+H(E_2)\to \C P(H(V))$. The category of hybrid hypergraphs is defined as the comma category $(-)+(=)\ls \C P$ where $(-)+(=)\colon \mathbf{Set}\x \mathbf{Set}\to \mathbf{Set}$ is the functor which takes $(X,Y)$ to $X+Y$.

\section{Summary of Relations and Structure}
There are eight categories of hypergraphs dependant on three factors: (1) whether there exists empty edges (EE), (2) if the morphisms are lax (Lax), and (3) if edges are allowed to take multiplicities in their incidence maps (Mult).
\begin{table}[htbp]
	\centering
	\begin{tabular}{|c || c | c | c |} 
		\hline 
		Cat. & EE & Lax & Mult   \\ [0.5ex] 
		\hline\hline
		$\C H$ & Y & N & N\\
		$\C B$ & Y & Y & Y\\
		$\C H_{\text{\normalfont lax}}$ & Y & Y & N\\
		$\C H^+$ & N & N & N\\
		$\C B^+$ & N & Y & Y\\
		$\C H^+_{\text{\normalfont lax}}$ & N & Y& N\\
		$\C B_{\text{\normalfont str}}$ & Y & N & Y\\
		$\C B^+_{\text{\normalfont str}}$ & N & N & Y
		\\[1ex] 
		\hline
	\end{tabular}
	\caption{The Categories of Hypergraphs}
\end{table}

\begin{table}[htbp]
	\centering
	\begin{tabular}{|c || c | c | c | c |c | c |} 
		\hline 
		$F$-Graph & Limits & Colimits & Exp. & $\Omega$ & Images & NNO  \\ [0.5ex] 
		\hline\hline
		$\C G_{\C P}$ & Yes & Yes & Yes & No & Yes  & Yes  \\ \hline
		$\crG_{\C P}$ & Yes & Yes & Yes & No & Yes  & Yes  \\ \hline
		$\kG$ & Yes & Yes & Yes & No & Yes  & Yes  \\ \hline
		$\crG_{\underline \Pi_X}$ & Yes & Yes & Yes & No & Yes  & Yes  \\ \hline
		$\C G_{\Pi_J\C P}$ & Yes & Yes & Yes & No & Yes  & Yes  \\ \hline
		$\crG_{\Pi_J\C P}$ & Yes & Yes & Yes & No & Yes  & Yes  \\ \hline
		$\C G_{\underline \Pi_S\x \underline \Pi_T}$ & Yes & Yes & Yes & No & Yes  & Yes  \\ \hline
		$\crG_{\underline \Pi_S\x \underline \Pi_T}$ & Yes & Yes & Yes & No & Yes  & Yes  \\ \hline
		$\C G_{\Sigma_{n\in \DD N}\underline \Pi_n}$ & Yes & Yes & Yes & No & Yes  & Yes  \\ \hline
		$\crG_{\Sigma_{n\in \DD N}\underline \Pi_n}$ & Yes & Yes & Yes & No & Yes  & Yes  \\ \hline
		$\C G_{\underline \Pi_S\x \underline \Pi_T}$ & Yes & Yes & Yes & No & Yes  & Yes  \\ \hline
		$\crG_{\underline \Pi_S\x \underline \Pi_T}$ & Yes & Yes & Yes & No & Yes  & Yes  \\ \hline
		\\[1ex] 
		\hline
	\end{tabular}
	\caption{Topos Properties of (Reflexive) $F$-Graphs}
\end{table}
}

Let $\oDDH$ be the category with a set of objects $\setm{A_i}{i\in \DD N}+\{V\}$ with homsets given by $\oDDH(A_i,A_j)=\oDDH(A_i,V)=\empset$ for $i\neq j$, $\oDDH(V,A_i)=\setm{x_k}{1\leq k\leq i}$.

\begin{proposition}
	The category of oriented hypergraphs is equivalent to the category of presheaves on $\oDDH$.
\end{proposition}
\begin{proof}
	
\end{proof}
\begin{corollary}
	The category of oriented hypergraphs is an adhesive category. 
\end{corollary}

\ignore{
\section{Abstract Categories of Hypergraph}

\begin{definition} Let $\C G$ be an extensive category
	\begin{ienum}
		\item An object $X$ is connected provided the representable $\C G(X,-)\colon \C G\to \mathbf{Set}$ preserves coproducts.
		\item A vertex object $V$ is a connected (strong) projective object which is subterminal, i.e., for each object $G$ if there exists a morphism $G\to V$, it is unique.
		\item An object $X$ is $V$-discrete provided it is a copower of $V$.
		\item An edge object $E$ is a connected object such that if $f\colon E\to E$ is a morphism which is fixed on vertices, then $f=\Id_E$.
		\item A $V$-vertex of an object $G$ is a morphism $x\colon V\to G$ (necessarily a monomorphism). We denote a $V$-vertex $x$ in $G$ by $x\in G(V)$. 
		\item Given a class of edge objects $\DD E$, an edge of an object $G$ is a morphism $e\colon E\to G$ where $E\in \DD E$ such that for each commuting diagram
		\[
		\xymatrix{ E \ar[dr]_{e} \ar[rr]^f && E' \ar[dl]^g \\ & G &}
		\]
		with $E'\in \DD E$ there exists a morphism $r\colon E'\to E$ with $e\circ r=e'$ and $r\circ f=\Id_E$. If $e\colon E\to G$ is an edge we denote this relationship by $e\in G(E)$.
		\item Given a class of edge objects $\DD E$, $V$-vertices $x, y\in G(V)$ are $\DD E$-connected provided there exists commuting diagrams of the form
		\[
		\xymatrix@!=.1em{ & E \ar@{..>}[d]^e \\ V \ar[r]_x \ar@{..>}[ur] & G} \qquad 
		\xymatrix@!=.1em{ & E' \ar@{..>}[d]^{e'} \\ V \ar[r]_y \ar@{..>}[ur] & G}		
		\] 
		where $e\in G(E)$ and $e'\in G(E')$.
		\item A class of objects $\DD E$ is a class of edge objects provided for each morphism $f\colon E\to G$, there exists a unique edge $e\colon E'\to G$ and either a monomorphism $m\colon E\to E'$ such that $e\circ m=f$ or a monomorphism $m'\colon E'\to E$ such that $f\circ m=e$. 
		\item An $E$-edge $e\colon E\to G$ is a loop provided there exists a vertex $x\colon V\to G$ such that for each vertex $y\in E(V)$, $e\circ y=x$.
		\item A set of $\DD E$-edges $(e_j)_{J}$ in an object $G$ is an $\DD E$-intersecting family provided for each pair $e_j, e_{j'}$ in the family there is a commuting diagram of the form
		\[
		\xymatrix@!=.1em{ & V \ar@{..>}[dr] \ar@{..>}[dl] & \\ E \ar[dr]_{e_j} && E' \ar[dl]^{e_{j'}} \\ & G & } 
		\]
		\item An object $G$ has the Helly property provided 
		\item An object $G$ is simple (respectively, complete) provided for each edge object $E$ and each morphism $f\colon \Sigma_{\C G(V,E)}V\to G$, there is at most one (respectively, exactly one) edge $e\colon E\to G$ such that $e\circ \rho=f$.
	\end{ienum}
\end{definition}

\begin{quote}\tbf{The Axioms for a Graphical Category}, $(\C G, \DD V, \DD E)$\\
	AXIOM 0: $\C G$ is locally small.\\
	AXIOM 1: $\C G$ is extensive.\\
	AXIOM 2: $\C G$ has strong images.\\
	AXIOM 3: $\DD V$ is a set of vertex objects.\\
	AXIOM 4: $\DD E$ is a class of edge objects.\\
	AXIOM 5: $\DD S\defeq \DD V\cup \DD E$ is a separating class.\\
	AXIOM 6: Each object has an irreducible connected decomposition.
\end{quote}

For a morphism $f\colon G\to G'$ in a category of hypergraphs, for each vertex $v\in G(\DD V)$, we write $f(v)\defeq f\circ v\in G'(\DD V')$ and for each $e\in G(\DD E)$ we write $f(e)\defeq \overline{f\circ e}\in G(\DD E')$ where $\overline{f\circ e}$ is the unique edge associated to $f\circ e\colon E\to G$. In other words, morphisms preserve vertices and edges. 

\begin{lemma}
	\begin{ienum} 
		\item A colimit of connected objects over a connected diagram is connects.
		\item Connected objects are closed under quotients. 
	\end{ienum}
\end{lemma}

\begin{proposition}
	Let $(\C G, \DD V, \DD E)$ be a category of hypergraphs.
	\begin{ienum}
		\item (Extensive properties)
		\begin{aenum}
			\item
		\end{aenum}
		\item (Vertex Object Properties)
		\begin{aenum}
			\item Every morphism $v\colon V\to G$ is a monomorphism.
			\item The terminal object $1$ has exactly one vertex for each vertex object.
			\item For each object $G$ and vertex object $V$, $V\x G\iso \Sigma_{\C G(V, H)}V$.
			\item The class of discrete objects is a exponential ideal in $\C G$.
			\item For each vertex object $V$, the representable functor $\C G(V,-)\colon \C G\to \mathbf{Set}$ preserves any limits and colimits which exist. \tbf{ETS: $\C G(V,-)$ preserves coequalizers!!!!}
		\end{aenum}
		\item (Edge Object Properties)
		\begin{aenum}
			\item The terminal object $1$ has one edge for each connected component of the full subcategory $\DD E\hookrightarrow \C G$.
			\item 
		\end{aenum}
		\item (Separating Class Properties)
		\begin{aenum}
			\item A morphism $f\colon \sqcup_I A_i\to B$ is a monomorphism iff each $f_i\colon A_i\to B$ is a monomorphism. (Hint: Check on separators)
			\item If $\DD S$ is a set, then there is a faithful $F\colon \C G\to \mathbf{Set}^{\DD S^{op}}$ where $\DD S$ is the full subcategory of $\C G$.
			\item If $\DD S$ is a set and $\C G$ is effective regular then $\C G$ is equivalent to $\mathbf{Set}^{\DD S^{op}}$. More precisely, $\mathbf{Set}^{\DD S^{op}}$ is the effective regular completion of $\C G$.
		\end{aenum}
	\end{ienum}
\end{proposition}
\begin{proof}
	
\end{proof}

\begin{definition}
	A morphism $f\colon G\to G'$ in a category of hypergraphs is strict provided for each edge $e\colon E\to G$, the morphism $f\circ e\colon E\to G'$ is an edge. 
\end{definition}
}

\begin{appendices}
\chapter{Category Theory}

\end{appendices}

}


\begin{thebibliography}{9}{\footnotesize

	\bibitem{jA}
	{\sc J. Ad\'amek, H. Herrlich, and G. E. Strecker.} \textit{Abstract and concrete categories: the joy of cats.} Repr. Theory Appl. Categ. 17 (2006). Reprint of the 1990 original. Wiley, New York.

	
	\bibitem{hA}
	{\sc H. Applegate and M. Tierney.} Categories with models. In \textit{Seminar on triples and categorical homology theory: lecture notes in mathematics}, No. 80. Reprinted with commentary in TAC 18 (2008) 122-180.
	
	\bibitem{dA}
	{\sc D. Archdeacon, J.H. Kwak, J. Lee, M.Y. Sohn.}
	Bipartite covering graphs.
	\tit{Discrete Math.} 214 (2000), 51-63.
	
	\bibitem{gA}
	{\sc G. Ausiello, P. G. Franciosa and D. Frigioni.} 
	Directed hypergraphs: problems, algorithmic results, and a novel decremental approach. In \tit{Proceedings of the Seventh Italian Conference on Theoretical Computer Science (ICTCS)}, v. 2202 of Lecture Notes in Computer Science, 312-327. Springer-Verlag, 2001.

	
	\bibitem{mB}
	{\sc M. Bauderon and H. Jacquet.} Node rewriting in graphs and hypergraphs: a categorical framework.Theoretical Computer Science 266(1-2), 463-487 (2001).

	\bibitem{pB}
	{\sc P. Boldi and S. Vigna.} Fibrations of graphs. \tit{Discrete Math.} 243 (2002), 21-66.

	\bibitem{fB}
	{\sc F. Borceux.}
	\textit{Handbook of categorical algebra}, volume 50-52 of \textit{Encyclopedia of Mathematics and its Applications.} Cambridge University Press, Cambridge (1994).
	
	\bibitem{aB}
	{\sc A. Bretto.}
	\textit{Hypergraph theory: an introduction}.
	Springer International Publishing. Switzerland (2013). 
	
	\bibitem{rBTG}
	{\sc R. Brown.} \tit{Topology and Groupoids.} Booksurge PLC. (2006).
	
	\bibitem{rBS}
	{\sc R. Brown, I. Morris, J. Shrimpton, and C. D. Wensley.} Graphs of morphisms of graphs. \textit{Electon. J. Combin.} 15 (1) (2008) 490-509.
	
	\bibitem{rB}
	{\sc R. T. Bumby and D. M. Latch.}
	Categorical constructions in graph theory. 
	\textit{Internat. J. Math. Math. Sci.} 9 (1) (1986) 1-16.
	
	\bibitem{aC}
	{\sc A. Carboni and E. M. Vitale.} Regular and exact completions. \tit{J. Pure Appl. Algebra} 125 (1998), 79-117.
	
	\bibitem{yC}
	{\sc Y. Caro and J. Lauri.} Non-monochromatic non-rainbow colourings of $\sigma$-hypergraphs. \tit{Discrete Math.} 318 (2014) 96-104. 
	
	\bibitem{aD}
	{\sc A. Daneshgar, M. Hejrati, and M. Madani.} On cylindrical graph construction and its applications. \tit{Electron. J. Comb.} 23(1) (2016) 29-45.
	
	
	\bibitem{wD}
	{\sc W. Dofler and D. A. Waller.}
	A category-theoretic approach to hypergraphs. \tit{Archiv der Mathematik} 34 no. 1 (1980) 185-192.
	
	\bibitem{hE}
	{\sc H. Ehrig, et al. }\textit{Graph and Model Transformations.} Monographs in Theoretical Computer Science. (2015).

	\bibitem{mE}
	{\sc M. El-Zahar and N. Sauer.} The chromatic number of the product of two 4-chromatic graphs is 4. Combinatorica 5 no. 2 (1985) 121-126.

	\bibitem{mE}
	{\sc M. Ernst.} The category-theoretical imperative. PhD thesis, University of California, Irving (2014).
	
	\bibitem{jF1}
	{\sc J. Foniok and C. Tardif.} Adjoint functors and tree duality. \tit{Disc. Math. Theor. Comput. Sci.} 11(2) (2009) 97-110.
	
	\bibitem{jF}
	{\sc J. Foniok and C. Tardif.} Digraph functors which admit both left and right adjoints. \tit{Discrete Math.} 338, 4 (6) (2015) 527-535. 
	
	\bibitem{jFcT}
	{\sc J. Foniok and C. Tardif.} Hedetniemi's conjecture and adjoint functors in thin categories. \tit{Appl. Catgor. Struct.} (2017).
	
	\bibitem{gG}
	{\sc G. Gallo, G. Longo, S. Nguyen and S. Pallottino.} {Directed hypergraphs and applications.} Discrete Applied Mathematics, 42(2) (1993) 177-201.
	
	\bibitem{jG}
	{\sc J. Gray}. Fibred and cofibred categories. \tit{Proc. Conference on Categ. Algebra at La Jolla}. Springer (1966) 21-83.
	
	\bibitem{wG}
	{\sc W. Grilliette.} Injective envelopes and projective covers of quivers. \tit{Electron. J. Combin.} 19 (2) (2012) \#P39.

	\bibitem{hH}
	{\sc H. Hajiabolhassan and F. Meunier.} Hedetniemi's conjecture for Kneser hypergraphs. \tit{Journal of Combinatorial Theory, Series A}. 143 (2016) 42-55. 

	\bibitem{pH}
	{\sc P. Hell.}
	An introduction to the category of graphs. In \textit{Topics in graph theory (New York, 1977)}, volume 328 of \textit{Ann. New York Acad. Sci.,} New York Acad. Sci., New York (1979) 120-136.
	
	\bibitem{jH}
	{\sc J. Hughes.} {A study of categories of algebras and coalgebras.} PhD thesis, Carnegie Mellon University (2001).
	
	\bibitem{cJ}
	{\sc C. Jakel.}
	A coalgebraic model of graphs. 
	\textit{British J. of Math. \& Comp. Sci.,} 15 (5) (2016) 1-6.

	\bibitem{gR}
	{\sc D. Janssens and G. Rozenberg.}
	Hypergraph systems generating graph languages. In \textit{Graph-Grammars and Their Application to Computer Science.} (1982) 172-185.	
	
\bibitem{pJT}
{\sc P.T. Johnstone.}
\textit{Topos Theory.} London Mathematical Society Monographs, vol. 10, Academic Press, London, New York, San Francisco. (1977)	
	
	\bibitem{pJ}
	{\sc  P.T. Johnstone.}
	\textit{Sketches of an elephant: a topos theory compendium, Vol. 1, 2.} Oxford Logic Guides 43. The Clarendon Press, Oxford University Press, New York (2002).
	
	\bibitem{kI}
	{\sc K. Iriye and D. Kishimoto.} Hom complexes and hypergraph colorings. \tit{Topology and its Applications}. 160(12) (2013) 1333-1344.
	
	\bibitem{sL}
	{\sc S. Lack and P. Soboci\'nski.} Toposes are adhesive. In: Walukiewicz, I. (ed.) FOSSACS 2004. LNCS, v. 2987. Springer, Heidelberg (2004) 273-288.

	\bibitem{jL}
	{\sc J. Lauri, R. Mizzi, and R. Scapellato.} Two-fold automorphisms of graphs. \tit{Australasian J. Combinatorics.}, 49 (2011) 165-176.

	\bibitem{wL05}
	{\sc F.W. Lawvere.} Categories of spaces may not be gereralized spaces as exemplified by directed graphs. \textit{Revista Colombiana de Matem\'aticas} XX (1986) 179-186. Reprinted with commentary in TAC 9 (2005) 1-7.
	
	\bibitem{wL89}
	{\sc F.W. Lawvere.} Display of graphics and their applications, as exemplified by 2-categories and the Hegelian "Taco". \textit{Proceedings of the First International Conference on Algebraic Methodology and Software Technology}. The University of Iowa (1989).
	
	\bibitem{wL}
	{\sc F.W. Lawvere.} Qualitative distinctions between some toposes of generalized graphs. In \textit{Categories in computer science and logic (Boulder, Co, 1987)}, volume 92 of \textit{Contemp. Math.} Amer. Math. Soc., Providence, RI (1989) 261-299.
	
	\bibitem{wLrR}
	{\sc F.W. Lawvere and R. Rosebrugh.}
	\textit{Sets for mathematics.} 
	{Cambridge University Press, Cambridge (2003).}	
	
	\bibitem{yL}
	{\sc Y. Long.} Graph relations and constrained homomorphism partial orders. PhD thesis, Universit at Leipzig. (2014).
	
	\bibitem{sM}
	{\sc S. Mac Lane.}
	\textit{Categories for the working mathematician}, volume 5 of \textit{Graduate Texts in Mathematics}. Springer-Verlag, New York, second edition (1998).

	\bibitem{dM}
	{\sc D. Mikl\'os.} Great intersecting families of edges in hereditary hypergraphs. \tit{Discrete Math.} 48 (1) (1984) 95-99.

	\bibitem{sMG}
	{\sc S. Mimram and Cinzia Di Giusto.} A categorical theory of patches. \tit{Electronic Notes in Theoretical Computer Science (ENTCS)}, 298, 283-307, November, 2013.
	
	\bibitem{dP}
	{\sc D. Plessas.}
	The categories of graphs. PhD thesis, The University of Montana. Missoula, MT (2011).
	
	\bibitem{aP}
	{\sc A. Pultr.} The right adjoints into the categories of relational systems. In \tit{Reports of the Midwest Category Seminar, IV}, volume 137 of \tit{Lecture Notes in Mathematics}. Berlin, Springer. (1970) 100-113.
	
	\bibitem{mR}
	{\sc M. Reyes and G. Reyes.}
	\textit{Generic figures and their glueings: A constructive approach to functor categories.} Polimetrica. Milano, Italy. (2004). 
	
	\bibitem{eR}
	{\sc E. Riehl.}
	\textit{Category theory in context.}
	Dover Publications, Inc. (2016).
	
	\bibitem{kR}
	{\sc K.I. Rosenthal.} \'Etendues and categories with monic maps. \tit{J. Pure Appl. Algebra} 22 (1981) 193-212.
	
	\bibitem{lR}
	{\sc L. Rusnak.} Oriented hypergraphs: introduction and balance. \tit{Electron. J. Combin.} 20 (3) (2013) \#P48, 29.
	
	\bibitem{jS}
	{\sc J. Shrimpton.} Some groups related to the symmetry of a directed graph. \tit{J. Pure Appl. Algebra} 72 (3) (1991) 303-318.
	
	\bibitem{zT}
	{\sc Zs. Tuza and V. Voloshin.} Uncolorable mixed hypergrahs. \textit{Discrete Applied Math.,} 99. (2000) 209-227. 
	
	
	\bibitem{sV}
	{\sc S. Vigna.}
	A guided tour in the topos of graphs. Technical Report 199-97, Universit\'a di Milano, Dipartmento di Scienze dell'Informazione (1997).  
	
	\bibitem{vV}
	{\sc V. Voloshin.} \textit{Introduction to Graph and Hypergraph Theory.} Nova Science Publishers, Inc. New York, 2009.
	
	\bibitem{dW}
	{\sc D.A. Waller}. Double covers of graphs. \tit{Bull. Australian Math. Soc.} 14 (1976) 233-248.
	
	\bibitem{tW}
	{\sc T.R.S. Walsh.}
	Hypermaps versus bipartite maps. \textit{J. Combinatorial Theory (B)} 18. (1975) 155-163.
	

	\bibitem{kW}
	{\sc K. K. Williams.}  The category of graphs. Master's thesis, Texas Tech University (1971).
	
	\bibitem{jWW}
	{\sc J. Worrell.} A note on coalgebras and presheaves. \textit{Electronic Notes in Theoretical Computer Science} 65 No. 1. (2002).}

\end{thebibliography}
\end{document}